\documentclass[12pt,a4paper]{article}
\PassOptionsToPackage{numbers, compress}{natbib}
\usepackage{natbib}
\usepackage{mlmath}
\usepackage{braket}
\usepackage{array}
\usepackage{subfigure}
\usepackage{float}
\usepackage{extarrows}
\usepackage{mathptmx}
\usepackage{amsmath}
\usepackage{epsfig,amsthm}
\usepackage{amssymb,amsfonts}
\usepackage{dsfont}
\usepackage{graphicx}
\usepackage{indentfirst}
\usepackage{mathrsfs}
\usepackage{color}
\usepackage{lineno}
\usepackage{url}
\usepackage{longtable}
\usepackage{xparse}
\usepackage{bm}
\usepackage{changes}
\usepackage{authblk}
\usepackage{hyperref}
\usepackage{enumerate}
\usepackage{booktabs}
\usepackage{supertabular}
\usepackage{longtable}
\usepackage[algo2e,ruled,vlined]{algorithm2e}
\usepackage{makecell}
\usepackage{multirow}
\usepackage{geometry}
\newtheorem{theorem}{Theorem}[section]
\newtheorem{remark}[theorem]{Remark}
\newtheorem{lemma}[theorem]{Lemma}
\newtheorem{corollary}[theorem]{Corollary}
\newtheorem{proposition}[theorem]{Proposition}
\newtheorem{assumption}[theorem]{Assumption}
\newtheorem{prop}[theorem]{Proposition}
\newtheorem{hypothesis}[theorem]{Hypothesis}
\newtheorem{example}[theorem]{Example}

\newtheorem{claim_01}{Claim}
\newtheorem{claim_02}{Claim}
\newtheorem{claim_03}{Claim}
\newtheorem{case}{Case}
\theoremstyle{definition}
\newtheorem{definition}[theorem]{Definition}
\theoremstyle{remark}

\newcommand*\patchAmsMathEnvironmentForLineno[1]{%
  \expandafter\let\csname old#1\expandafter\endcsname\csname #1\endcsname
  \expandafter\let\csname oldend#1\expandafter\endcsname\csname end#1\endcsname
  \renewenvironment{#1}%
     {\linenomath\csname old#1\endcsname}%
     {\csname oldend#1\endcsname\endlinenomath}}%
\newcommand*\patchBothAmsMathEnvironmentsForLineno[1]{%
  \patchAmsMathEnvironmentForLineno{#1}%
  \patchAmsMathEnvironmentForLineno{#1*}}%
\AtBeginDocument{%
\patchBothAmsMathEnvironmentsForLineno{equation}%
\patchBothAmsMathEnvironmentsForLineno{align}%
\patchBothAmsMathEnvironmentsForLineno{flalign}%
\patchBothAmsMathEnvironmentsForLineno{alignat}%
\patchBothAmsMathEnvironmentsForLineno{gather}%
\patchBothAmsMathEnvironmentsForLineno{multline}%
}

\newcommand{\keywords}[1]{\textbf{Keywords}: #1}
\newcommand{\AMS}[1]{\textbf{AMS}: #1}

\title{On Residual Minimization for PDEs: Failure of PINN, Modified Equation, and Implicit Bias}
\author[a,b \footnote{luotao41@sjtu.edu.cn}]{Tao Luo}
\author[a \footnote{zhouqixuan@sjtu.edu.cn}]{Qixuan Zhou}
\affil[a]{School of Mathematical Sciences, Shanghai, 200240, China.}
\affil[b]{Institute of Natural Sciences,  CMA-Shanghai, MOE-LSC and Qing Yuan Research Institute, Shanghai Jiao Tong University, Shanghai, 200240, China}

\begin{document}

\maketitle
\allowdisplaybreaks

\begin{abstract}
    As a popular and easy-to-implement machine learning method for solving differential equations, the physics-informed neural network (PINN) sometimes may fail and find poor solutions which bias against the exact ones. In this paper, we establish a framework of modified equation to explain the failure phenomenon and characterize the implicit bias of a general residual minimization (RM) method. We provide a simple way to derive the modified equation which models the numerical solution obtained by RM methods. Next, we show the modified solution deviates from the original exact solution. The proof uses a by-product of this paper, that is, a necessary and sufficient condition on characterizing the singularity of the coefficients. This equivalent condition can be extended to other types of equations in the future. Finally, we prove, as a complete characterization of the implicit bias, that RM method implicitly biases the numerical solution against the exact solution and towards a modified solution. In this work, we focus on elliptic equations with discontinuous coefficients, but our approach can be extended to other types of equations and our understanding of the implicit bias may shed light on further development of deep learning based methods for solving equations.
\end{abstract}

\keywords{residual minimization, PINN, implicit bias, modified equation}

\AMS{35D30, 35D35, 35R05, 35R06, 65N15}

%


\tableofcontents

\section{Introduction}
The application of machine learning, particularly deep neural networks (DNNs), has gained significant attention in recent years for solving partial differential equations (PDEs)~\cite{weinan2017deep, weinan2018deep,he2020relu,li2020amulti, liu2020multi,raissi2019physics,zhang2022mod}. Machine learning techniques show great potential in addressing challenging problems~\cite{weinan2021dawning}. Compared with traditional numerical schemes, such as finite difference, finite elements methods, and spectral methods, which are often limited by the ``curse of dimensionality", DNNs have demonstrated success in solving many high-dimensional problems~\cite{han2018solving}. Although the traditional numerical methods are powerful for low-dimensional problems, it can be challenging to design a proper scheme to solve low-dimensional problems with low-regularity solutions or boundaries~\cite{li2006immersed,osher2004level,osher1988fronts,peskin1972flow,peskin2002immersed, sethian1999level}. Therefore, DNNs are also promising in solving low-dimensional problems with low-regularity solutions or complex boundaries, such as problems with discontinuous elastic or dielectric constants in composite materials. 

A widely used approach to solving PDEs is to utilize DNNs to parameterize the solution and optimize the parameters in an objective function which usually formulated as a least-squares or variational loss function (also known as risk function). The physics-informed neural network (PINN) method was first proposed in the 1990s~\cite{lee1990neural}, then also studied by Sirignano and Spiliopoulos under the name Deep Galerkin Method (DGM)~\cite{sirignano2018DGM}, and later popularized and known as PINN by Raissi et al.~\cite{raissi2019physics}. In this method, a DNN is trained to minimize the sum of the residual of the PDE and the residual of the boundary condition. The Deep Ritz method (DRM) was proposed by~\cite{weinan2017deep}, where a variational formulation is used to obtain a neural network solution by minimizing an energy functional. Alongside PINN and DRM, many other methods have been proposed or developed for solving PDE problems, for example, \cite{ming2021deep,lu2021deepxde,bao2021weak}. For further advances in PINN, we refer readers to the review articles \cite{cai2022physics, lawal2022physics} and the references therein.
For completeness, let us mention that the operator learning also seems promising in both solving PDE problems and their inverse problems~\cite{goswami2020physics,li2020neural,li2021fourier, lu2019deeponet}. Despite the increasing variety of PDE-solving methods, PINN has gained substantial attention due to its simplicity and ease of implementation. The PINN risk function is simply the residual of the PDE, without requiring additional knowledge such as the variational form in DRM, which is difficult or even impossible to obtain in many problems.

A theoretical study of DNNs can have significant implications and applications to design DNN-based PDE solvers. For instance, the universal approximation theorem~\cite{cybenko1989approximation} highlights the strength of wide (two-layer) NNs in approximating functions. Recent research in this direction includes the convergence rate with respect to network size, which has been explored in~\cite{lu2021deep}, and studies on the rate of DNN approximation with respect to depth and width for PINNs when solving second-order elliptic equations with Dirichlet boundary condition~\cite{jiao2022arate}. Several works have also focused on the generalization error of PINNs when the exact solution to the PDE has high regularity. For example, Lu et al.~\cite{lu2021priori} derive the generalization error bounds of two-layer neural networks in the framework of DRM for solving equation and static Schr\"odinger
equation on the $d$-dimensional unit hypercube; Mishra and Molinaro \cite{mishra2020estimates} provide upper bounds on the generalization
error of PINNs approximating solutions of the forward problem for PDEs; and Shin et al.~\cite{shin2020convergence} prove that the minimizers of PINN converge uniformly to the exact solution under the case of second order elliptic and parabolic PDEs. 

As low-regularity problem plays the role of potential application of DNNs in solving (probably low-dimensional) PDEs, in this paper, we investigate the use of neural networks with a least-squares risk function to solve linear and quasilinear elliptic PDEs with solutions that exhibit low regularity. This is a challenging problem due to several features inherent in the problem. Firstly, the exact solution is less regular, as it is affected by the discontinuity of coefficients. Additionally, the function space of neural networks is typically higher in regularity, as certain order derivatives are required in neural network methods, such as the derivative in the PDE and gradient training. Finally, the risk function is discretized, which brings another type of complexity to the problem. We conduct a detailed analysis of the neural network solution that results from these interactions between the neural network and the PDE, with the aim of gaining a deeper understanding of whether and how neural networks can be effectively used in the context of less regular solutions to PDEs.

In this work, we first point out key observations from one-dimensional numerical experiments using PINN, and then we develop continuum model and prove theorems to explain the observed phenomena. Roughly speaking, for an elliptic equation with discontinuous coefficients, the PINN finds a numerical solution which deviates from the exact solution. With strong numerical evidences, we model this numerical solution by a modified solution, where the latter satisfies a modified equation. We will derive this modified equation in a very simple and heuristic way. It is proved that the modified solution severely deviates from the original exact solution in a quite generic way. Moreover, we obtain necessary and sufficient condition under which this deviation occurs. Furthermore, we prove theorems to explain the implicit bias phenomenon that even given a good initial guess (such that the NN function is sufficiently close to the exact solution), after training, the numerical solution still deviates from the original exact solution and approximates to the modified solution.
Besides, we extend most of our results to the case of quasilinear elliptic equations. 

Our results are independent of the structure of neural network and work for any RM methods, not being exclusive to the PINN. They are starkly different from previous results that focus on the high-regularity problem and show PINN solution can converge to the exact solution as the sample number and the network size increase~\cite{chen2021representation, lu2021priori, mishra2020estimates, shin2020convergence}. In contrast, we point out that there is an essential gap between the exact solution and the limit of numerical solution, and that the RM method in solving low-regularity PDEs may implicit bias the numerical solution against the exact solution and has a non-convergence issue. By unravelling the theoretical mechanisms, our work not only explains the implicit bias phenomena but also provides a theoretical guidance for further design of DNN algorithms in solving low-regularity PDEs. 

We believe this phenomenon of ``failure'' is more or less known to many researchers, and our main contribution of this work is to design a mathematical framework for systematically studying this failure and hence to shed light on understanding the implicit bias of the RM methods. In particular, let us highlight our understanding of the implicit bias of RM method: \textbf{RM method implicitly biases the numerical solution towards the solution to a modified equation.}
We expect our approach will be quite general and can be developed for many other problem concerning the implicit bias of optimization algorithms for deep learning problems. 

The rest of this paper is organized as follows. A brief introduction to DNNs and PINN and the failure of PINN example can be found in Section~\ref{sec::Preliminaries}. In Section~\ref{sec::MainContribution}, we summarize our main contribution with thorough discussion. Moreover, for readers' convenience, two flow charts are provided at the end of this section showing the connection of main results. In Section~\ref{sec::RemovSing}, we obtain a necessary and sufficient condition of removable singularity for equations with BV functions. Section~\ref{sec::Lin} proves that the deviation of the modified solution occurs generically and can be large and that the RM methods implicit bias the RM solution towards the modified solution. Hence this explains the failure phenomenon. In Section~\ref{sec::ExtQausiLinEllipEq}, we extend some of the results to the case of quasilinear elliptic equations. In appendix, we collect notations, recall the definition of BV function, and rephrase some well-known existence theorem for linear and quasilinear elliptic equations from the literature. We also give a short proof of the failure example in one-dimension and complete the proof of extensions to the quasilinear case.

\section{Preliminaries}\label{sec::Preliminaries}
We begin this preliminary section by introducing basic concepts of deep neural networks and deep learning based methods for solving PDEs, in particular, the method of physics-informed neural networks (PINN). Next, by a one-dimensional example, we illustrate that PINN can fail even in an extremely simple situation. We hope the reader keep this inspiring example in their mind when reading the analytic details in the rest part of this paper. The last subsection is left to depict the general setting, namely the linear elliptic equations and systems with BV coefficients, under which we will derive a continuum model for the numerical solution and thus prove theorems based on this model. 

\subsection{Deep neural networks (DNN)}\label{sec::DNN}

An \textbf{$L$-layer fully-connected neural network} function $u_\theta:\sR^d\to\sR^{d'}$ is defined as for each $x\in \sR^d$
\begin{equation}
    u_\theta(x) = W^{[L-1]} \sigma(W^{[L-2]}\sigma(\cdots (W^{[1]} \sigma(W^{[0]} x + b^{[0]} ) + b^{[1]} )\cdots)+b^{[L-2]})+b^{[L-1]}, 
\end{equation}
where the matrix $W^{[l]} \in \sR^{m_{l+1}\times m_{l}}$ and the vector $b^{[l]}\in\sR^{m_{l+1}}$ are called \textbf{parameters}, $m_l\in \sN^+$ is the width of the $l$-th layer, and the (nonlinear) function $\sigma:\sR\to\sR$ is known as the \textbf{activation function}. With a little bit abuse of notation, $\sigma$ applied on a vector means entry-wise operation, namely $(\sigma(z))_i=\sigma(z_i)$ for any subscript $i$.
The matrices are usually reshaped and concatenated into a column vector $\theta$, that is, $\theta=\operatorname{vec}\left(\{W^{[l]}\}_{l=0}^{L-1},\{b^{[l]}\}_{l=0}^{L-1}\right)$. Note that the input dimension $d=m_0$ and the output dimension $d'=m_{L}$.

\subsection{Residual minimization (RM) and physics-informed neural networks (PINN)}\label{sec::RMPINN}	
The physics-informed neural network (PINN) is a popular method for solving PDEs via neural networks. It is proposed by many research groups. Among them, the first one might belong to Sirignano and Spiliopoulos~\cite{sirignano2018DGM}, although they use a different name, the deep Galerkin method. The most famous work is perhaps the one written by Raissi, Perdikaris and Karniadakis~\cite{raissi2019physics}, and it gives rise to the more commonly-used name, the PINN. For further developments of the PINN, we refer the readers to the review paper~\cite{cai2022physics} and the references therein. 

In our numerical experiments to be presented in the next subsection, the PINN is used to solve a given \textbf{boundary value problem} (BVP) of a partial differential equations (PDE). We emphasize that the phenomenon recognized and analyzed in this work will remain the same if we replace the PINN by any other residual minimization method. Here the residual of an equation refers to the difference between the left-hand-side and the right-hand-side, and we say \textbf{residual minimization} because these PINN type methods follow a common approach, that is, minimizing the residual risk of both the PDE and boundary conditions.  The term residual minimization is also used by other researchers, for example,~\cite{shin2020error}. Besides, some works refer the PINN to the least-squares method, for example,~\cite{cai2020deep}.

For a given function $w$, a residual minimization method for solving (the BVP of) an equation is to minimize the following \textbf{population risk} 
\begin{equation}\label{eq::PopRM4OriginalEq}
    R(w)=\int_{\Omega}(Lw-f)^2\diff{x} + \gamma \int_{\partial\Omega}(Bw-g)^2\diff{x}.
\end{equation}
Here $L$ is the differential operator, $B$ is the boundary condition operator, and $f$ (respectively, $g$) is a given function defined in the interior (respectively, on the boundary) of the domain $\Omega$. The factor $\gamma$ is the weight for adjusting the importance of the boundary constraints versus the one in the interior of the domain.
But in practical applications, one often uses the \textbf{empirical risk} (also known as the objective function in optimization) as:
\begin{align}
    R_{S}(w)
    &= R_{S,\mathrm{int}}(w)+\gamma R_{S,\mathrm{bd}}(w)\label{eq::EmpiRM4OriginalEq}\\
    R_{S,\mathrm{int}}(w)&= \frac{\abs{\Omega}}{n_{\mathrm{int}}}\sum_{x\in S_{\mathrm{int}}}(Lw(x)-f(x))^2\\
    R_{S,\mathrm{bd}}(w)&=\frac{\abs{\partial\Omega}}{n_{\mathrm{bd}}}\sum_{x\in S_{\mathrm{bd}}}(Bw(x)-g(x))^2,
\end{align}
where $R_{S,\mathrm{int}}(w)$ is the \textbf{interior empirical risk}, $R_{S,\mathrm{bd}}(w)$ is the \textbf{boundary empirical risk}, $n_{\mathrm{int}}\in\sN^+$ and $n_{\mathrm{bd}}\in\sN^+$ are the numbers of samples in the interior dataset $S_{\mathrm{int}}$ and the boundary dataset $S_{\mathrm{bd}}$, respectively; $\abs{\Omega}$ and $\abs{\partial\Omega}$ are the Lebesgue measure $\fL^d$ of $\Omega$ and Hausdorff measure $\fH^{d-1}$ of $\partial\Omega$, respectively.

In the method of PINN, the output is a neural network and denoted by $u_{\theta}(x)$. Thus we usually abuse notation and write 
\begin{equation}
    R(\theta)=R(u_{\theta})=\int_{\Omega}(Lu_{\theta}(x)-f(x))^2\diff{x} + \gamma \int_{\partial\Omega}(Bu_{\theta}(x)-g(x))^2\diff{x}.
\end{equation}
Similarly, for empirical risks, we write $R_S(\theta)=R_S(u_{\theta})$, $R_{S,\mathrm{int}}(\theta)=R_{S,\mathrm{int}}(u_{\theta})$, and $R_{S,\mathrm{bd}}(\theta)=R_{S,\mathrm{bd}}(u_{\theta})$.

Given an initial parameter $\theta_0$, then the parameters will be updated by first order optimization methods such as the \textbf{gradient descent} which is the forward Euler scheme of the (negative) gradient flow with step size $\Delta t$ (also known as the \textbf{learning rate}) as follows
\begin{equation}
    \theta_{k+1}= \theta_{k}-\Delta t D_{\theta}{R_{S}}(\theta_{k}).
\end{equation}
The samples $S$ can be chosen as quadrature method in low dimensions, say no larger than three, while it can be chosen randomly as Monte Carlo sampling method in any dimensions including very high-dimension cases. 
In implementation, the sample $S$ can vary from one iteration to another.

\subsection{Numerical example: failure of PINN}\label{sec::FailExPINN}

In this subsection, we demonstrate that using PINN to solve PDEs which has no strong (or classical) solutions could be problematic and cause a non-infinitesimal error. This is illustrated by a one-dimensional example via numerical experiments. 
Although the example is simple, it provides us correct and insightful intuition. Throughout this work, we will revisit this example several times and explain new understanding with our theorems to be proved in later sections. We also emphasize that this kind of failure of PINN in learning the solution to certain PDE is essential and usually can not be resolved by merely selecting the network architecture, adjusting the optimization algorithm, or tuning hyperparameters of the network.

Let us first briefly mention the setup of the experiment. The considered equation is one-dimensional and reads as
\begin{equation}\label{eq::1dEq}
    \left\{
    \begin{aligned}
        L u =-D_{x}(AD_{x} u) &=f\quad \text{in}\  \Omega=(-1,1), \\
         u &=0\quad \text{on}\  \partial\Omega=\{-1,1\},
    \end{aligned}
    \right.
\end{equation}
where the coefficient function $A$ and the interior data $f$ in~\eqref{eq::1dEq} are both piece-wise continuous and read as
\begin{equation}\label{eq::1dCoeff}
    A(x)=\left\{
    \begin{aligned}
        & \tfrac{1}{2}, & & x\in (-1,0), \\
        &1, & & x\in [0,1),
    \end{aligned}
    \right.\quad 
    f(x)=\left\{
    \begin{aligned}
        & 0, & & x\in (-1,0), \\
        & -2, & & x\in [0,1).
    \end{aligned}
    \right.
\end{equation}
Clearly, there is no strong (or classical) solution, while the weak solution $u\in H^1((-1,1))$ to this equation is
\begin{equation}\label{eq::SolTo1dEq}
    u(x)=
    \left\{
    \begin{aligned}
        & -\tfrac{2}{3}x-\tfrac{2}{3}, & & x\in (-1,0), \\
        & x^2-\tfrac{1}{3}x-\tfrac{2}{3}, & & x\in [0,1).
    \end{aligned}
    \right.
\end{equation}
In the series of numerical experiments, we use a $1$-$256$-$256$-$256$-$1$ residual network (ResNet)~\cite{e2019apriori}.
The empirical risk, interior empirical risk, and boundary empirical risk for a given function $w$ read as
\begin{align*}
    R_{S}(w)
    &=R_{S,\mathrm{int}}(w)+\gamma R_{S,\mathrm{bd}}(w),\\
    R_{S,\mathrm{int}}(w) 
    &= \frac{\abs{\Omega}}{n_{\mathrm{int}}}\sum_{x\in S_{\mathrm{int}}}[D_{x}(A(x)D_{x} w(x))+f(x)]^2,\\
    R_{S,\mathrm{bd}}(w)
    &= \frac{\abs{\partial\Omega}}{n_{\mathrm{bd}}}[(w(-1))^2+(w(1))^2] ,
\end{align*}
where we use $\abs{\Omega}=2$, $\abs{\partial\Omega}=2$, $n_{\mathrm{bd}}=2$, and $\gamma=1$ in the experiments. We choose $1000$ uniformly-sampled points in the interior of the region (namely $n_{\mathrm{int}}=1000$). We use ``$\tanh$'' activation function, Adam optimizer, and the Xavier initialization, where the variance of each entry of $W^{[l]}$ is $\frac{2}{m_{l-1}+m_{l}}$ and $m_l$ represents the width of $l$-th layer. 

Under the above setting, we obtain the network function $u_{\theta}$ numerically via the PINN (or more generally the RM method) after training process. This solution $u_{\theta}$ will be called the RM solution to the original equation. We stress that the PINN as one realization of the RM method, is not essential and can be replaced by any other RM method. In practice, $u_\theta$ is obtained when the risk function $R_S$ is small and does not decay anymore along the training. We plot the RM solution $u_{\theta}$ and compare it with $u$ in Figure \ref{fig::deviation} (a). Obviously, the RM method fails to find the exact solution $u$. The gap between $u$ and $u_{\theta}$ is as large as magnitude of $u$. Thus we call this phenomenon the \textbf{failure of RM method} (or \textbf{failure of PINN}).
By the way, we also notice that the first order derivative of $u_{\theta}$ seems to be piece-wisely parallel to that of $u$ (see Figure \ref{fig::deviation} (b)). This suggests us to focus on the derivatives and regularity of the solutions.

\begin{figure}
	\centering	
	\subfigure[Comparison of solutions]{\includegraphics[scale=0.35]{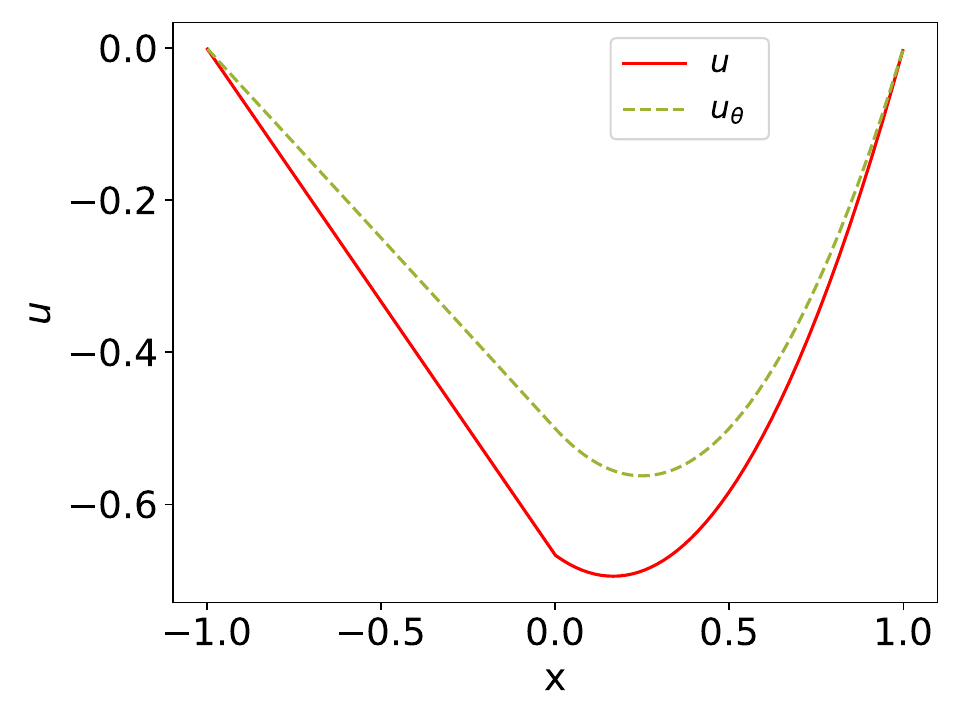}}
	\subfigure[Comparison of the first order derivatives]{
	\includegraphics[scale=0.35]{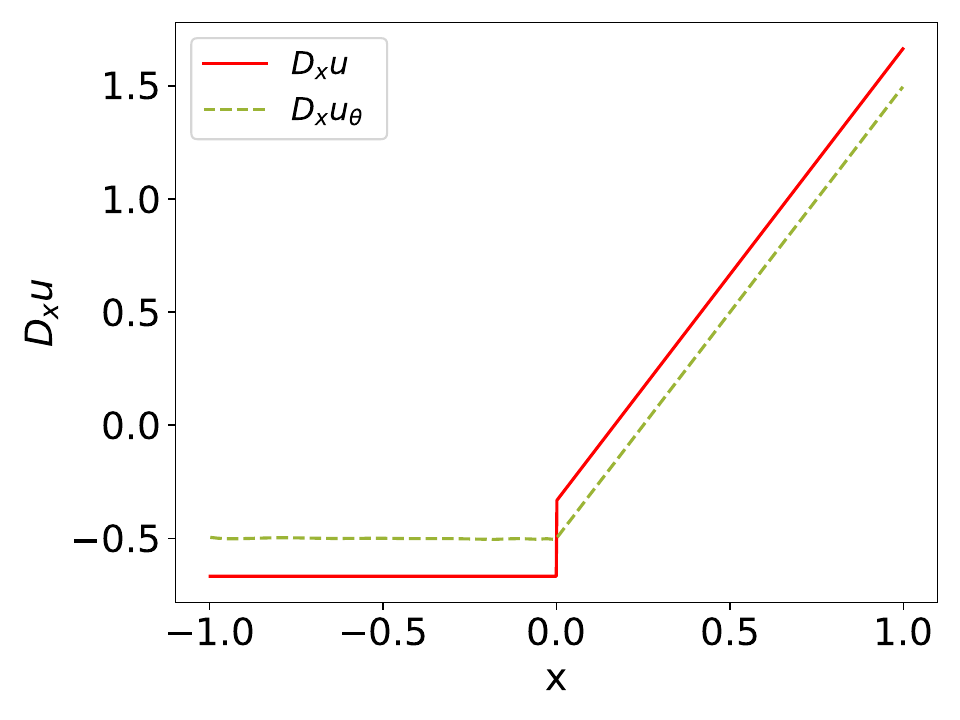}}
	\caption{PINN cannot find the exact solution to  equation~\eqref{eq::1dEq} with the coefficient function given by~\eqref{eq::1dCoeff}.}
	\label{fig::deviation}
\end{figure}

The solution $u$ at all points except for $x=0$ is smooth (even in $C^\infty$ locally), while the first order derivative of the solution $u$ at $x=0$ has a jump. Therefore, whether $x=0$ contributes in the numerical method is decisive. Here comes the key observation that in practical experiments the $x=0$ point can only be sampled with nearly zero probability as long as the distribution for sampling is absolutely continuous with respect to the Lebesgue measure. Hence the point $x=0$ almost never contributes to numerical experiments! By the product rule, it holds that $D_{x}(AD_{x}u)=AD^2_{x}u+(D_{x}A)D_{x}u$ on $(-1,1)\backslash \{0\}$. Since the derivative $D_{x}A(x)=0$ for the piece-wise constant function $A$ and for all $x\in (-1,1)\backslash \{0\}$, the left-hand-side of the PDE is effectively equal to $AD^2_{x}u$ on the whole interval $(-1,1)$. Therefore, in the numerical experiments, the empirical risk function equals to an effective empirical risk function with probability nearly one, namely
\begin{equation}
    R_S(u_\theta)=\tilde{R}_S(u_\theta),
\end{equation}
where the latter at a given function $w$ reads as
\begin{equation}\label{eq::EffEmpRisk}
    \tilde{R}_S(w)=\frac{2}{n_{\mathrm{int}}}\sum_{x\in S_{\mathrm{int}}}(A(x)D^2_{x}w(x)+f(x))^2+\gamma[(w(-1))^2+(w(1))^2].
\end{equation} 
This effective empirical risk function is in turn the empirical risk function of the RM method for the modified equation
\begin{equation}\label{eq::1dModEq}
    \left\{
    \begin{aligned}
        \tilde{L} \tilde{u} =-AD^2_{x}\tilde{u} &=f\quad \text{in}\  \Omega=(-1,1), \\
         \tilde{u} &=0\quad \text{on}\  \partial\Omega=\{-1,1\}.
    \end{aligned}
    \right.
\end{equation}
whose population risk function at a given function $w$ reads as
\begin{equation}\label{eq::EffRisk}
    \tilde{R}(w) = \int_{-1}^{1}(A(x)D_x^2w(x)+f(x))^2\diff{x}+\gamma(w^2(-1)+w^2(1)).
\end{equation}
In other words, the risk \eqref{eq::EffEmpRisk} is the discretization of \eqref{eq::EffRisk}.

The exact solution to the modified equation~\eqref{eq::1dModEq} is denoted by $\tilde{u}$ and explicitly reads as
\begin{equation*}
    \tilde{u}= \left\{
    \begin{aligned}
        & -\tfrac{1}{2}x-\tfrac{1}{2}, & & x\in (-1,0), \\
        & x^2-\tfrac{1}{2}x-\tfrac{1}{2}, & & x\in [0,1).
    \end{aligned}
    \right.
\end{equation*}

To sum up, now we have altogether three solutions: $u$ (the exact solution to the original equation \eqref{eq::1dEq}), $u_\theta$ (the RM solution  to the original equation \eqref{eq::1dEq}) and $\tilde{u}$ (the exact solution to the modified equation \eqref{eq::1dModEq}). For completeness, we can also consider the RM solution to the modified equation \eqref{eq::1dModEq}, and we denote it as $\tilde{u}_{\theta}$. 

\begin{figure}[H]
	\centering	
	\subfigure[Comparison of solutions]{\includegraphics[scale=0.35]{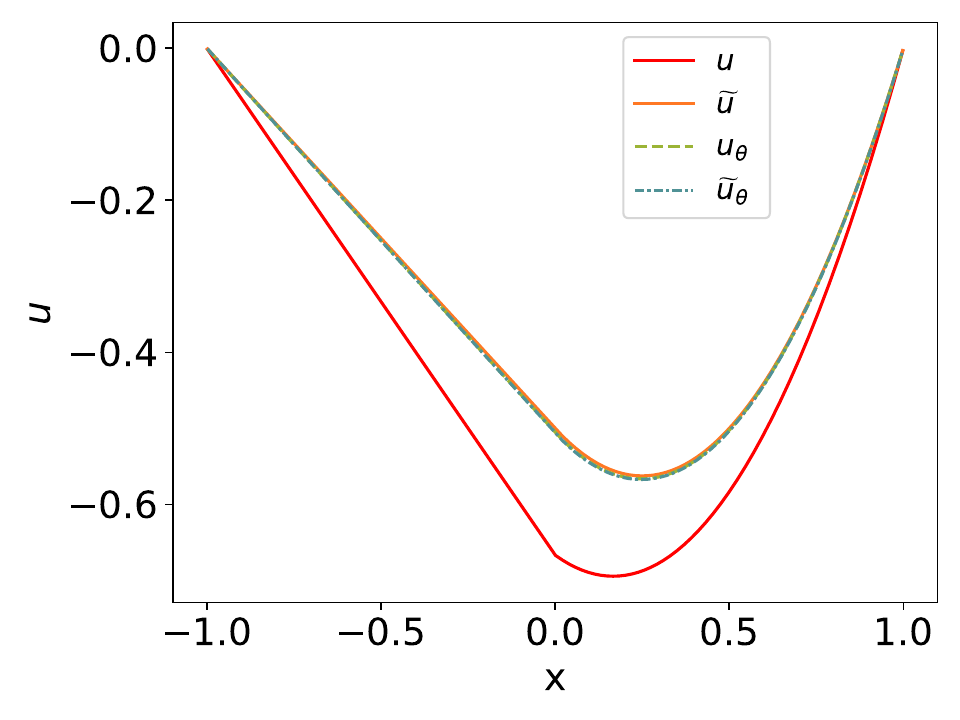}}
	\subfigure[Comparison of the first order derivatives]{
	\includegraphics[scale=0.35]{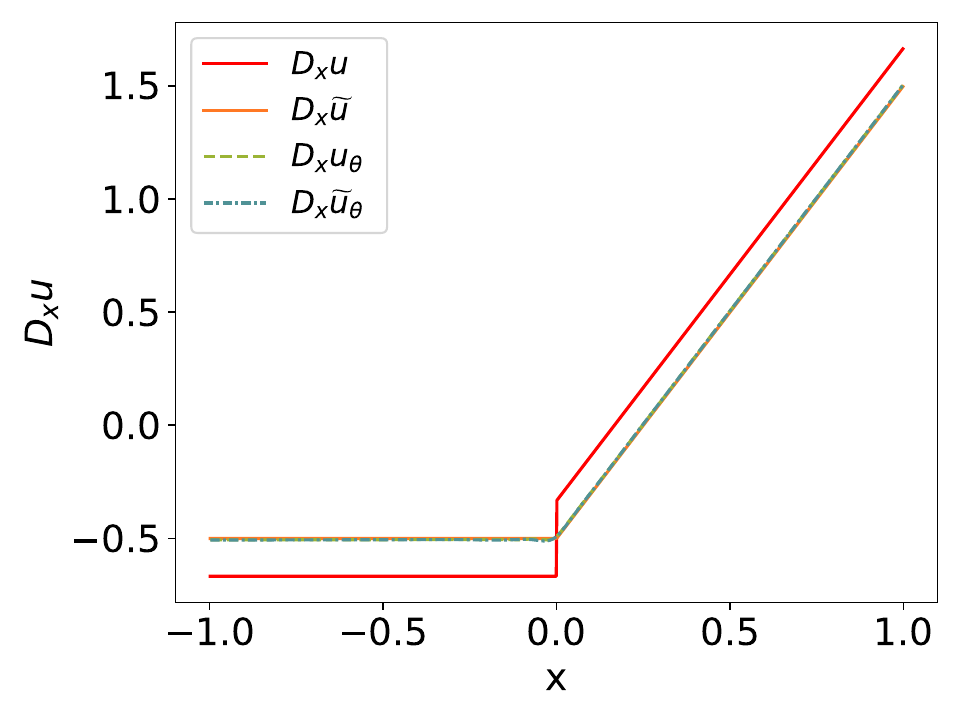}}

    \subfigure[empirical risk $R_S(\theta)$]{\includegraphics[scale=0.35]{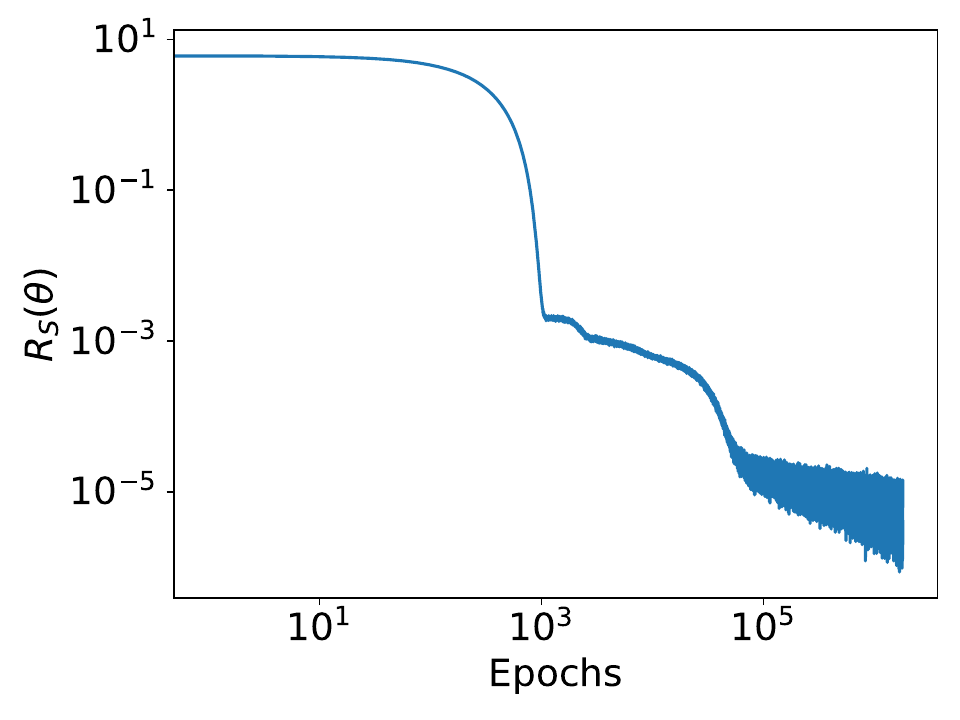}}
	\subfigure[empirical risk $\tilde{R}_S(\theta)$]{
	\includegraphics[scale=0.35]{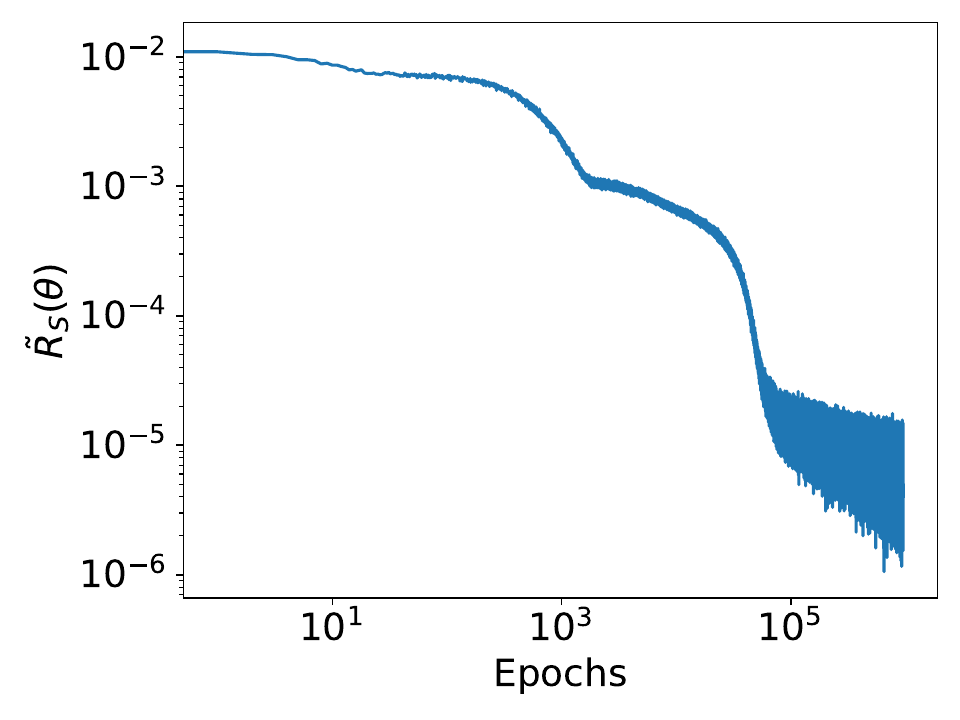}}
	\caption{PINN can fit the exact solution to the equation~\eqref{eq::1dEq} with coefficients as~\eqref{eq::OneDimNonDeviate}.}
	\label{fig::modifedloss}
\end{figure}

Figure \ref{fig::modifedloss} and Table \ref{tab:Dif&RelatDif} show a detailed and quantitative comparison between all these four solutions. 
Throughout the paper, a Banach space $Y$ equipped with the norm $\norm{\cdot}_{Y}$ will be written as an ordered pair $(Y,\norm{\cdot}_{Y})$ when it is needed to emphasize the norm. If the norm is obvious from the context, we will simply denote it as $Y$. In this work, we mainly focus on Hilbert spaces such as $L^2(\Omega)$, $H^1(\Omega)$, and $H^2(\Omega)$. However, to verify our intuition, we also add one row of the $L^{\infty}(\Omega)$ (relative) deviation to Table \ref{tab:Dif&RelatDif}. 

Notice that $u$ and $u_{\theta}$ can both access very small empirical risk, while they have a finite gap in terms of $L^\infty(\Omega)$ norm as well as $H^1(\Omega)$ (and hence $L^2(\Omega)$) norm. In some sense, it indicates the non-uniqueness of the solutions as the local/global minima of the empirical risk function. It seems that the method bias to a special solution in some implicit way. This leads to one of the central problems, the implicit bias problem, of deep learning methods for solving PDEs --- why a method find such a particular solution from the infinitely many minima. This implicit bias is obviously connected to the success or failure of the methods, and hence it will be the central object of this research work. 

A take-away-message is the relation $u_{\theta}\approx \tilde{u}_{\theta}\approx \tilde{u}\neq u$ and the smallness $R_S(u_{\theta})\ll 1$ and $\tilde{R}_S(\tilde{u}_{\theta})\ll 1$. Roughly speaking, the failure of RM method occurs and the RM solution can be modelled by the modified equation. Looking more carefully at Figure~\ref{fig::modifedloss} (b), we observe that the first-order derivative of the RM solution $u_{\theta}$ (as well as $\tilde{u}_{\theta}$) is piece-wisely parallel to the one of the exact solution $u$. Therefore, except for finitely many points, that is, the only point $x=0$ in this case, the second-order derivatives of the RM solution and the exact solution are the same. Since the point $x=0$ can not be sampled with nearly probability one, it is expected that the RM solution have the possibility to achieve very small empirical risk. This is validated in practical experiments. 
Figure~\ref{fig::modifedloss} (c) (or (d), respectively) shows the evolution of the empirical risk $R_S$ (or $\tilde{R}_S$, respectively) along training dynamics of the RM method applied to the problem~\eqref{eq::1dEq} (or \eqref{eq::1dModEq}, respectively) with the coefficient function given by~\eqref{eq::1dCoeff}. At the initial stage of the training, the empirical risk is of order one, while at the final stage this risk can reduce to $10^{-5}$ or $10^{-6}$. 

\begin{table}[H]
 \begin{longtable}{p{0.4cm}<{\raggedright}p{3.5cm}<{\raggedright}p{2.7cm}<{\raggedright}p{2.7cm}<{\raggedright}p{2.7cm}<{\raggedright}} 
        \toprule[1.5pt] 
       & $\norm{\cdot}_Y$&  $Y=L^{\infty}(\Omega)$ & $Y=L^{2}(\Omega)$ & $Y=H^1(\Omega)$ 
        \\\toprule[1.5pt] 
       \multirow{5}{*}{(a)} & $\norm{u-\tilde{u}}_Y$
        & $\frac{1}{6}\approx 1.667\times 10^{-1}$ & $\frac{\sqrt{6}}{18}\approx 1.361\times 10^{-1}$ & $\frac{\sqrt{6}}{9} \approx 2.722\times 10^{-1}$  
        \\ 
        \multirow{3}{*}{} & $\norm{u-\tilde{u}}_Y/\norm{\tilde{u}}_Y$
        & $\frac{8}{27}\approx 2.963\times 10^{-1}$ & $\frac{\sqrt{170}}{51}\approx 2.557\times 10^{-1}$ & $\frac{2\sqrt{10}}{3\sqrt{67}}\approx 2.576\times 10^{-1}$ 
        \\
        \multirow{3}{*}{} & $\norm{u-\tilde{u}}_Y/\norm{u}_Y$
        & $\frac{6}{25}\approx 2.400\times 10^{-1}$ & $\frac{\sqrt{5}}{\sqrt{119}}\approx 2.050\times 10^{-1}$ & $\frac{2\sqrt{5}}{\sqrt{449}}\approx 2.111\times 10^{-1}$ 
        \\ \hline
         \multirow{2}{*}{(b)} & $\norm{u_{\theta}-u}_Y$ & $1.646\times 10^{-1}$ & $1.345\times 10^{-1}$ & $2.691\times 10^{-1}$
        \\
         \multirow{2}{*}{} & $\norm{u_{\theta}-u}_Y/\norm{u}_Y$ & $2.371\times 10^{-1}$ & $2.026\times 10^{-1}$ & $2.087\times 10^{-1}$
        \\ \hline
         \multirow{2}{*}{(c)} & $\norm{u_{\theta}-\tilde{u}}_Y$ & $2.880\times 10^{-3}$ & $2.044\times 10^{-3}$ & $2.281\times 10^{-3}$
        \\
         \multirow{2}{*}{} & $\norm{u_{\theta}-\tilde{u}}_Y/\norm{\tilde{u}}_Y$ & $5.120\times 10^{-3}$ & $3.839\times 10^{-3}$ & $2.161\times 10^{-3}$
        \\ \hline
         \multirow{3}{*}{(d)} & $\norm{u_{\theta}-\tilde{u}_{\theta}}_Y$ & $1.465\times 10^{-3}$ & $1.267\times 10^{-3}$ & $5.702\times 10^{-3}$
        \\
         \multirow{3}{*}{} & $\norm{u_{\theta}-\tilde{u}_{\theta}}_Y/\norm{\tilde{u}_{\theta}}_Y$ & $2.598\times 10^{-3}$ & $2.372\times 10^{-3}$ & $5.384\times 10^{-3}$
        \\
        \multirow{3}{*}{} & $\norm{u_{\theta}-\tilde{u}_{\theta}}_Y/\norm{u_{\theta}}_Y$ & $2.592\times 10^{-3}$ & $2.371\times 10^{-3}$ & $5.374\times 10^{-3}$
        \\
        \toprule[1.5pt] 
    \end{longtable}    
    \caption{The deviation and relative deviation from one solution to another under 
    and $L^{\infty}(\Omega)$, $L^2(\Omega)$  and $H^1(\Omega)$ norms. In part (a), the comparison is between two exact solutions $u$ and $\tilde{u}$. In part (b) and (c), the RM solution $u_{\theta}$ is compared with $u$ and $\tilde{u}$, respectively. In part (d), the comparison is between two RM solutions $u_{\theta}$ and $\tilde{u}_{\theta}$.}\label{tab:Dif&RelatDif} 
\end{table} 

We believe that such example where $u_{\theta}$ fails to learn $u$ has been seen by many researchers. But our approach is novel and we establish a complete framework to handle such problems with focus on $\tilde{u}$. 
This eventually leads to the understanding of the failure and implicit bias of the RM methods.

\subsection{Linear elliptic equations with BV coefficients}\label{sec::LinEllipEqSysBVCoeff}
In this subsection, we introduce the general setting on the (systems of) elliptic PDEs used for the main results of this work. Some assumptions are given for the linear elliptic equations and systems. Although main contributions (See detailed description in Section~\ref{sec::MainContribution}) of this work focus on these linear problems, we nevertheless stress that some key results will be extended to the quasilinear setting in Section~\ref{sec::ExtQausiLinEllipEq}, and hopefully it can be transferred to more general setting and other PDEs in the future.

For the linear case, we consider the system of elliptic equations written in the divergence form:
\begin{equation}\label{eq::OriginEq}
    \left\{
    \begin{aligned}
        L u &=f & & \text{in}\ \Omega, \\
         u &=0 & & \text{on}\ \partial\Omega,
    \end{aligned}
    \right.
\end{equation}
with 
\begin{equation*}
     (L u)^{\alpha} = -\sum_{\beta=1}^{d'}\divg\cdot(A^{\alpha\beta}(x)Du^{\beta})
     =-\sum_{\beta=1}^{d'}\sum_{i,j=1}^{d}D_{i}(A_{ij}^{\alpha\beta}D_{j}u^{\beta}),
\end{equation*}
where $\alpha,\beta\in \{1,2,\ldots,d'\}$, $i, j\in \{1,2,\ldots,d\}$, $\Omega\subseteq \sR^d$ is a bounded domain with $C^{1,1}$ boundary, measurable functions $A^{\alpha\beta}\in S^{d\times d}$ are also symmetric in $\alpha,\beta$, namely $A^{\alpha\beta}=A^{\beta\alpha}$, and $f\in L^2(\Omega;\sR^{d'})$. 

Now we mention the basic assumption to be used throughout the paper. 

\begin{assumption}[BV coefficients]\label{assum::BVCoeff} 
    Let $L$ be the operator defined in~\eqref{eq::OriginEq}. Assume that for each $\alpha, \beta\in\{1,\ldots,d'\}$, there exist a scalar function $\chi^{\alpha\beta}\in SBV^{\infty}(\Omega)$ with $\fH^{d-1}(J_{\chi^{\alpha\beta}})< +\infty$ and a matrix-valued function $\bar{A}^{\alpha\beta}\in C^{1}(\bar{\Omega};S^{d\times d})$ such that $A^{\alpha\beta} = \chi^{\alpha\beta}\bar{A}^{\alpha\beta}$. Also assume that there is some $\alpha_0,\beta_0\in\{1,\ldots,d'\}$ satisfying $\fH^{d-1}(J_{\chi^{\alpha_0\beta_0}})>0$. Furthermore, we assume there are constants $\chi_{\min},\chi_{\max},\bar{\lambda},\bar{\Lambda}>0$ such that for each $\alpha,\beta\in\{1,\ldots,d'\}$ and for all $\xi\in\sR^d$, $x \in \Omega$
    \begin{align}\label{eq::SpecRad}
        &\chi_{\min}\leq \chi^{\alpha\beta}(x)\leq \chi_{\max},\\
        &\bar{\lambda}\abs{\xi}^2\leq \xi^{\T}\bar{A}^{\alpha\beta}(x)\xi\leq \bar{\Lambda}\abs{\xi}^2.
    \end{align}
\end{assumption}

\begin{figure}[H]
	\centering	
	\includegraphics[scale=0.55]{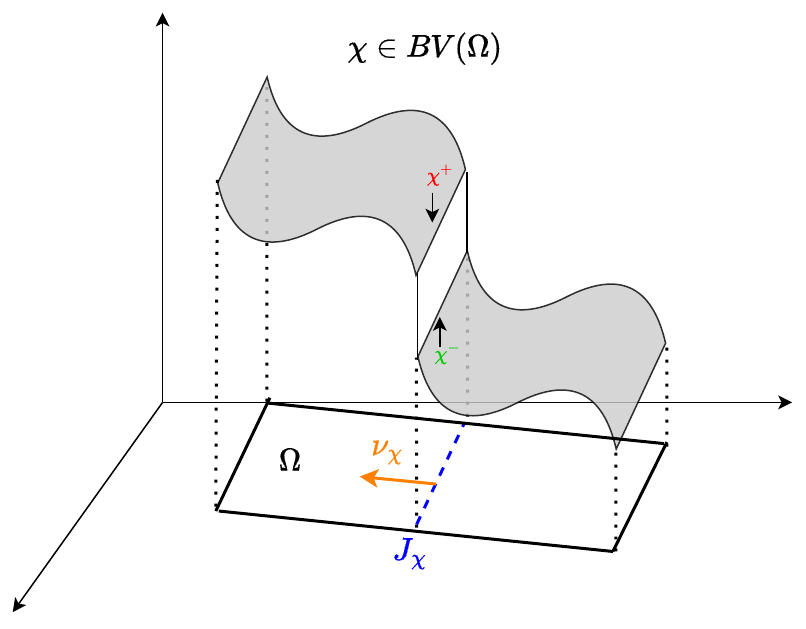}
	\caption{Illustration of a function of bounded variation $\chi$. The set of approximate jump points $J_\chi$ is determined by the triplet $(\chi^+,\chi^-,\nu_\chi)$. See details in Appendix~\ref{sec::BVFun}.}
	\label{fig::SBVFun}
\end{figure}

Here an $SBV^\infty$ function is a \textbf{special function of bounded variation} (SBV) whose absolutely continuous part of the gradient has an $L^\infty$ density. The precise definition and basic properties of SBV functions are given in Appendix~\ref{sec::BVFun}. Figure~\ref{fig::SBVFun} is an illustration of a SBV function. It is no harm for the readers to think each $\chi^{\alpha\beta}$ as a piece-wise constant function throughout the paper. 

Here we give some comments on our main assumption.

\begin{remark}[Hadamard--Legendre condition]\label{rmk::HadamardLegendreCond}
    Assumption~\ref{assum::BVCoeff} implies the \textbf{Hadamard--Legendre condition}, which is a standard condition for the existence of solution to systems of elliptic PDEs.
    Let $L$ be the operator defined in~\eqref{eq::OriginEq}. We say that $\{A^{\alpha\beta}\}_{\alpha,\beta=1}^{d'}$ satisfy the Hadamard--Legendre condition if there exist constants $\lambda,\Lambda>0$ such that for all $ \xi^{\alpha} \in \sR^{d}, x \in \Omega$
    \begin{equation}\label{eq::HadamardLegendreCond} 
        \lambda\abs{\xi}^{2}\leq \sum_{\alpha,\beta=1}^{d'}(\xi^{\alpha})^\T A^{\alpha\beta}(x)\xi^{\beta} \leq \Lambda\abs{\xi}^{2}.
    \end{equation} 
    Here $\abs{\xi}^2=\sum_{\alpha=1}^{d'}\abs{\xi^{\alpha}}^2$.
\end{remark}

\begin{remark}[uniform ellipticity condition]\label{rmk::UnifEllipCond}
    When $d'=1$, the superscripts $\alpha,\beta$ can only take value $1$. Thus for simplicity of notation, we will drop the superscripts throughout the paper when $d'=1$.
    In particular, for the case $d'=1$, the Hadamard--Legendre condition in Remark~\ref{rmk::HadamardLegendreCond} coincides with the uniform ellipticity condition, that is to say, there exist constants $\lambda,\Lambda>0$ satisfying 
    \begin{equation}\label{eq::UnifEllipCond} 
        \lambda\abs{\xi}^{2}\leq \xi^\T A(x)\xi \leq \Lambda\abs{\xi}^{2}
    \end{equation}
    for all $\xi \in \sR^{d}, x \in \Omega$.
\end{remark}

We also remark that in the proofs throughout the paper, the constant $C$ may be different from line to line, but we usually keep track of its dependence on basic constants such as $\chi_{\min}$ or $\bar{\Lambda}$ and thus make the paper more readable. In the proofs, the expression $U \prec\zeta\prec U'$ means $\zeta=1$ in $U$, $\zeta=0$ outside $U'$ and $\zeta\in C_c^{\infty}(\sR^d)$, where $U, U'$ are bounded open sets and $U$ is compactly contained in $U'$.

\section{Main contributions}\label{sec::MainContribution}
In this section, we describe our main contributions of this paper. After the introduction to each contribution point, several related theorems will be mentioned in an intuitive way. Most of them not only work for elliptic equations, but also work for elliptic systems, although some technical conditions may be inevitably assumed for the latter.

From now on, the term ``residual minimization'' (or in short ``RM'') is used to replace ``PINN'' in the main results because these analyses provided in this and later sections work for general residual minimization methods, and are not exclusive for PINN. In particular, the DNN representation is not explicitly used in the analysis. Nevertheless, the type of the risk function (also known as loss function) is more responsible to the failure or success of the machine learning based PDE solvers. 

In Section~\ref{sec::DeriveModEq}, we propose a hypothesis that $\tilde{u}$ approximates $u_{\theta}$ well and derive the modified equation for $\tilde{u}$, in a general setting, to model the numerical solution obtained by RM method. This hypothesis serves as our starting point of the analysis and understanding of the implicit bias of RM method. 
In Section~\ref{sec::CharacterRemovSing}, we provide an if-and-only-if condition to characterize the singularity which appears naturally because of the discontinuous coefficients in the equations. In particular, this condition characterizes whether $u$ is equal to $\tilde{u}$ or not.
Next, in Section~\ref{sec::RMInvSubSpace}, we introduce the RM-invariant subspace $\Ker(T-I)$, defined as the set of all $f$ which leads to $u=\tilde{u}$. This subspace $\Ker(T-I)$ allow us to identify the occurrence of deviation of the numerical solution. We will show the deviation occurs generically and the relative deviation, even for the data near the RM-invariant subspace, is not small.
The last contribution point is mentioned in Section~\ref{sec::ImplicitBiasness},
where we prove that the exact solution is unstable and hence the RM method implicit bias the exact solution towards the solution to the modified equation. In the last subsection, we present the connections of the main contributions, as well as preliminaries and by-products, by two flow charts.

\subsection{Modeling numerical solution by modified equation}\label{sec::DeriveModEq}
In any sense, it is very difficult, if not impossible, to study $u_{\theta}$ directly. Fortunately, as shown in Figure \ref{fig::modifedloss} (a), we have the key observation: $u_{\theta}\approx \tilde{u}_{\theta}\approx \tilde{u}$, that is, the RM solutions $u_{\theta}$ and $\tilde{u}_{\theta}$ both looks very close to the exact solution $\tilde{u}$, and they are almost indistinguishable. 
More precisely, part (c) and part (d) of Table \ref{tab:Dif&RelatDif} provide more quantitative evidences to show that the (relative) distances among $u_{\theta}$, $\tilde{u}_{\theta}$, and $\tilde{u}$ are very small in either $L^\infty(\Omega)$, $L^2(\Omega)$, or $H^1(\Omega)$ norm. 
This closeness is already intuitively explained in Section~\ref{sec::FailExPINN} and it provides us a solid evidence to model $u_{\theta}$ by using $\tilde{u}$ for the one-dimensional example. 

We would like to extend this idea and model $u_{\theta}$ by using $\tilde{u}$ to more general equations and to the case of system (that is $d'>1$). Let us start with $d'=1$ and a general coefficient $A$. The derivation is almost the same as the one in Section \ref{sec::FailExPINN}. But here the coefficient $A$ may not be piece-wise constant. Moreover, $A$ is a $d\times d$ matrix-valued function instead of a scalar function. By the product rule, the divergence operator in~\eqref{eq::OriginEq} is applied to $A$ and $Du$ respectively. Furthermore, by the decomposition of SBV function (See Definition~\ref{def::SBVFun}) with $d'=1$, we have
\begin{equation}
    Lu=-\sum_{i,j=1}^dA_{ij}D_{ij}u -\sum_{i,j=1}^d(D_{i}^{\mathrm{a}}A_{ij}+D_{i}^{\mathrm{j}}A_{ij})D_{j}u.
\end{equation}
Here $D_{i}^{\mathrm{j}}A_{ij}$ is supported on an $\fL^d$ null set $J_{\chi}$. Since the number of samples is at most countable in any practical applications, the probability of selecting some points in the support of $D_{i}^{\mathrm{j}}A_{ij}$ is zero. Unless the algorithm is specifically designed, the contribution of $D_{i}^{\mathrm{j}}A_{ij}$ to the risk function is zero. In other words, $D_{i}^{\mathrm{j}}A_{ij}$ will not affect the optimization process. As a result, we simply omit $D_{i}^{\mathrm{j}}A_{ij}$ and obtain the approximate model~\eqref{eq::ModEq} for studying the RM methods.

For the general setting with $d'\geq 1$ (that is including systems), we follow the same idea and thus arrive at the modified equation for~\eqref{eq::OriginEq} as follows
\begin{equation}\label{eq::ModEq}
    \left\{
    \begin{aligned}
        \tilde{L} \tilde{u} &=f & & \text{in}\  \Omega, \\
        \tilde{u} &=0 & & \text{on}\  \partial\Omega,
    \end{aligned}
    \right.
\end{equation}
where for each $\alpha\in\{1,\ldots,d'\}$
\begin{equation}
    (\tilde{L} \tilde{u})^{\alpha} =-\sum_{\beta=1}^{d'}\sum_{i,j=1}^{d}A_{ij}^{\alpha\beta}D_{ij}\tilde{u}^{\beta} -\sum_{\beta=1}^{d'}\sum_{i,j=1}^{d}D_{i}^{\mathrm{a}}A_{ij}^{\alpha\beta}D_{j}\tilde{u}^{\beta}.
\end{equation}

Throughout the paper, we call $\tilde{L}$ the modified operator of $L$ and denote $u$ the solution to~\eqref{eq::OriginEq}, and $\tilde{u}$ the solution to~\eqref{eq::ModEq}. Also, let $u_{\theta}$ and $\tilde{u}_{\theta}$ be the RM solutions, that is, the numerical solutions under RM methods (such as PINN), to the original equation~\eqref{eq::OriginEq} and modified equation~\eqref{eq::ModEq}, respectively. When the equation is clear from the context, we call $\tilde{u}$ the solution to the modified equation (or simply, the modified solution) and call $u_{\theta}$ the numerical solution (or RM solution). 

The previous numerical experiments suggest us to make the following hypothesis to model the numerical solution .

\begin{hypothesis}[modified solution approximates RM solution]\label{hypo::ModSolApproxRMSol}
    The RM solution to the problem \eqref{eq::OriginEq} can be approximated by the solution to its modified equation \eqref{eq::ModEq}. More precisely, for any given $\eps>0$, the RM method can find the a numerical solution $u_{\theta}$ such that $\norm{u_{\theta}-\tilde{u}}_{H^1(\Omega)}\leq \eps$ for all meaningful $f$. 
\end{hypothesis}

Here we neglect the details of the neural networks such as how to design the network architecture, how to tune the hyper-parameters, and how to train the neural network parameter. These will be left to the future research. This hypothesis will be the foundation of our work, and from now on, we will focus on $\tilde{u}$ which is more amenable because it satisfies a modified equation \eqref{eq::ModEq}. We emphasize that Hypothesis \ref{hypo::ModSolApproxRMSol} and our point of view on the modelling of the RM solution are novel. For the modified problem, we obtain a series of theorems. These with Hypothesis \ref{hypo::ModSolApproxRMSol} lead to the understanding of the behavior and properties of the RM solution, in particular, its implicit bias. 

We will work on both elliptic equations and systems. To prove the results in the case of elliptic systems, we need a further technical condition as follows.
\begin{assumption}[\emph{A priori} estimates for linear system]\label{assum::AprioriEstLinSys}
    Assume that for all $\tilde{u}\in H_0^1(\Omega;\sR^{d'})\cap H^2(\Omega;\sR^{d'})$, there is constant $C>0$ such that $\norm{\tilde{u}}_{H^2(\Omega;\sR^{d'})}\le C\norm{\tilde{L}\tilde{u}}_{L^2(\Omega;\sR^{d'})}$, where $\tilde{L}$ is defined as in~\eqref{eq::ModEq}.
\end{assumption}

We remark that Assumption~\ref{assum::BVCoeff} with $d'=1$ implies Assumption~\ref{assum::AprioriEstLinSys} (See Theorem \ref{thm::ExistAprioEstLinEllipEq&Sys}).

\subsection{Characterizing removable singularity}\label{sec::CharacterRemovSing}

Intuitively, it is clear that the failure of PINN, or more generally, the RM method, is due to the singularities in the coefficients of the PDE. Moreover, whether the singularity exists is highly related to whether the exact solution coincides with the modified solution, where the later is introduced in the above subsection.
In particular, if the singularity is removable then there is no such failure. 
In order to study the removable singularity in the coefficients, we introduce a quantity $\mu$ based on which we can establish necessary and sufficient condition. 

Given any $\fH^{d-1}$ measurable set $B\subseteq\Omega$, $\chi\in SBV^\infty(\Omega)$, $\Upsilon: H_0^1(\Omega) \to L^{\infty}(\Omega; S^{d \times d})$, and $\varphi\in C^1(\Omega)$, we define
\begin{equation}\label{eq::Def4mu}
    \mu(B;\chi,\Upsilon,\varphi)=\int_{B\cap J_{\chi}}(\chi^+-\chi^-)\nu_{\chi}^{\T} \Upsilon[\varphi] D\varphi\diff{\fH^{d-1}}.
\end{equation}
For example, if $B=\Omega$, $\varphi\in C_c^1(\Omega)$, and $\Upsilon[w]=\bar{A}$ for all $w\in H_0^1(\Omega)$, then we have 
\begin{equation*}
    \mu(\Omega;\chi,\bar{A},\varphi)=\int_{J_{\chi}}(\chi^+-\chi^-)\nu_{\chi}^{\T} \bar{A}D\varphi\diff{\fH^{d-1}}.
\end{equation*}

We thus have an essential result (Theorems~\ref{thm::NecSufCondRemovSing} and~\ref{thm::NecSufCondRemovSing4Sys}) to characterize the removable singularity in terms of $\mu$. Applying these theorems, we can consequently find smooth function $v_{\delta}$ such that 
\begin{equation*}
    \mu(\Omega;\chi,\bar{A},v_{\delta})\neq 0.
\end{equation*}
See Theorems~\ref{thm::NecSufCondRemovSingLin} and~\ref{thm::NecSufCondRemovSingLinSys}.

\subsection{Identifying the occurrence of deviation}\label{sec::RMInvSubSpace}
We study the occurrence of deviation $u\neq \tilde{u}$.
Thanks to Theorems~\ref{thm::NecSufCondRemovSingLin} and~\ref{thm::NecSufCondRemovSingLinSys}, we prove in Theorem \ref{thm::WorstCaseDeviLinEllipEq} that for specific interior data $f$ the deviation occurs. To step further, we ask whether this occurrence of deviation is generic and whether it is large in some sense. Affirmative answers to these questions will be obtained by studying the RM-invariant subspace $\Ker(T-I)$. The $\Ker(T-I)$ basically identifies the interior $f$ where $u=\tilde{u}$.

To define $\Ker(T-I)$, we should study the properties of $\tilde{L}$ first. Let us consider the largest possible domain of $\tilde{L}$, which is naturally a subset of $H^1_0(\Omega;\sR^{d'})$. In fact, it should be a subset of $H^2(\Omega;\sR^{d'})$. Otherwise, we have $v\in H^1_0(\Omega;\sR^{d'})\backslash H^2(\Omega;\sR^{d'})$, and hence $D_{ij}v^{\beta}\in H^{-1}(\Omega)$ and $A^{\alpha\beta}_{ij}\in SBV(\Omega)$ imply that the point-wise product $A^{\alpha\beta}_{ij}D_{ij}v^{\beta}$ may not be a classical function. Thus the largest possible domain of $\tilde{L}$ is 
\begin{equation}
    \operatorname{dom}(\tilde{L})=H^1_0(\Omega;\sR^{d'})\cap H^{2}(\Omega;\sR^{d'}).
\end{equation}
Consequently, the image of $\tilde{L}$, denoted by $X$, is 
\begin{equation}\label{eq::Range4data}
    X=\{\tilde{L}w\colon\,w\in H^1_0(\Omega;\sR^{d'})\cap H^{2}(\Omega;\sR^{d'}) \}.
\end{equation}

Next, we introduce the \textbf{RM-transformation} $T$. Suppose that Assumptions~\ref{assum::BVCoeff} and~\ref{assum::AprioriEstLinSys} hold. Let $u$ and $\tilde{u}$ be solutions to~\eqref{eq::OriginEq} and~\eqref{eq::ModEq} with data $f\in L^2(\Omega;\sR^{d'})$, respectively. Then we can define the operator $T$ 
\begin{align}\label{eq::RMTransformation}
    T: X&\to H^{-1}(\Omega;\sR^{d'})\\
    f&\mapsto Tf= L\tilde{u}.    
\end{align}
Clearly, $\tilde{u}$ is the weak solution to
\begin{equation}\label{eq::RMTrans}
    \left\{\begin{aligned}
    L \tilde{u} &= T f & & \text {in}\  \Omega, \\
    \tilde{u} &=0 & & \text {on}\  \partial\Omega.
    \end{aligned}\right.
\end{equation}
In particular, when $d'=1$, Assumption~\ref{assum::AprioriEstLinSys} holds automatically by Assumption~\ref{assum::BVCoeff}. Hence, for the single elliptic equation case, namely $d'=1$, we have $X=L^2(\Omega)$ and $T:L^2(\Omega)\to H^{-1}(\Omega)$.

Let $\sigma(T)$ be the spectrum of $T$. Then we will show in Theorem~\ref{thm::EigenValueEigenSpace} that the only eigenvalue of $T$ is $1$, namely $\sigma(T)=\{1\}$. Thus to identify the occurrence of deviation $u\neq \tilde{u}$, we only need to characterize the invariant subspace of $X$ under RM-transformation $T$. 
This naturally leads to the following kernel $\Ker_{X}(T-I)$, that is, the eigenspace of $T$ corresponding to the eigenvalue $1$ restricted to $X$:
\begin{equation}
    \Ker(T-I)=\Ker_{X}(T-I)=\{f\in X \colon\, Tf=f\}=\{\tilde{L}w\in X\colon\,\tilde{L}w= Lw\}.
\end{equation}
When the $X$ is clear from the context, we will drop it in the subscript and denote the kernel as $\Ker(T-I)$. We also denote its complement with respect to the whole space $L^2(\Omega;\sR^{d'})$ as $(\Ker(K-I))^c=L^2(\Omega;\sR^{d'})\backslash\Ker(T-I)$.

Let us explain why the space $\Ker(T-I)$ can characterize the occurrence of the deviation. If there is a non-zero $f\in X\backslash \Ker(T-I)$, then the unique solution $\tilde{u}$ to~\eqref{eq::ModEq} (namely $\tilde{L}\tilde{u}=f$) satisfies $\tilde{L}\tilde{u}\neq L\tilde{u}$. In other words, $f\neq L\tilde{u}$, and hence $\tilde{u}$ deviates from $u$, implying there is a deviation. Therefore, to understand the implicit bias of RM method, we only need to study the properties of $\Ker(T-I)$. In particular, Theorem \ref{thm::EigenValueEigenSpace} shows that under a mild condition that the jump is not omnipresent, the complement of the kernel $(\Ker(T-I))^{c}$ is open and dense
(see also Theorem \ref{thm::EigenValueEigenSpace} for the case of system). As a direct result, for almost all $f\in L^2(\Omega;\sR^{d'})$, the  deviation occurs. Furthermore, Theorem \ref{thm::UnbddRLinEllipEq} shows the relative deviation $\norm{u-\tilde{u}}_{H^1(\Omega;\sR^{d'})}/\norm{\tilde{u}}_{H^1(\Omega;\sR^{d'})}$ can be even unbounded.

Now, with these theorems on the relation between $\tilde{u}$ and $u$, we are ready to explain the phenomenon $u_{\theta}\neq u$ shown in the previous numerical experiments. Let us recall the phenomenon, discuss first intuitively, and then give a more quantitative explanation.

Recall that, according to Figure~\ref{fig::modifedloss} (a), the RM solution $u_{\theta}$ is entirely deviated from the exact solution $u$. More precisely, the part (b) of Table \ref{tab:Dif&RelatDif} show that the (relative) numerical errors between $u$ and $u_{\theta}$ are not small in both $L^\infty(\Omega)$ and $L^2(\Omega)$ (and hence $H^1(\Omega)$) norm. 

We first provide an intuitive understanding on the non-zero difference between the exact solution and the RM solution.  
For the exact solution $u$, the term $AD_{x}u$ has to be continuous, otherwise its derivative would be a Dirac-like function. Roughly speaking, the Dirac-like function is only non-zero at a single point but its integration on the whole space is non-zero. Therefore, even in the weak sense, the effect of Dirac function can not be ignored. However, the source term $f$ is a classical function which is defined pointwisely. In particular, it is not a Dirac-like function. Note that $A$ is discontinuous. In order to make $AD_{x}u$  be continuous, $D_{x}u$ has to be discontinuous as shown in Figure~\ref{fig::deviation}(b).    
For the RM solution $u_{\theta}$, as the isolated discontinuities of $A$ can not be exactly sampled, any solution, whose first-order derivative is piece-wisely parallel to that of the exact solution, is a solution that minimizes the empirical loss. The frequency principle~\cite{luo2019theory, xu2019frequency, xu2019training, zhang2021linear} shows that deep neural network implicitly prefers a low-frequency function to fit training data. Roughly speaking, compared with all feasible solutions, the one with continuous first-order derivative is a function has low frequency, which is learned by RM as shown in Figure~\ref{fig::deviation}(b). The rigorous connect between the frequency principle and the implicit bias of RM method is beyond the scope of this paper, and will be left to the future work.

Next, we conclude the contribution on the occurrence of the deviation with a more quantitative remark which explains the phenomenon $u_{\theta}\neq u$.
\begin{remark}[error of RM solution]\label{rmk::AbsErrRMSol}
    Theorem~\ref{thm::WorstCaseDeviLinEllipEq} leads to a finite error of RM solution. In fact, if we assume Hypothesis \eqref{hypo::ModSolApproxRMSol} with $\eps\leq \frac{C}{2}$, then we can numerically achieve $u_{\theta}$ such that $\norm{\tilde{u}-u_{\theta}}_{L^2(\Omega;\sR^{d'})}\leq \frac{C}{2}$. Consequently, the deviation $\norm{u-\tilde{u}}_{H^1(\Omega;\sR^{d'})}\geq C$ leads to the estimate $$\norm{u-u_{\theta}}_{L^2(\Omega;\sR^{d'})}\geq \norm{u-\tilde{u}}_{L^2(\Omega;\sR^{d'})}-\norm{\tilde{u}-u_{\theta}}_{L^2(\Omega;\sR^{d'})}\geq \frac{C}{2},$$ which implies a finite (non-infinitesimal) numerical error when the gap $\norm{u-\tilde{u}}_{L^2(\Omega;\sR^{d'})}$ takes a non-infinitesimal value.
\end{remark}

\subsection{Understanding the implicit bias}\label{sec::ImplicitBiasness}
Compared to the above results on the deviation, it is more important to understand the implicit bias of the RM method. The latter has to be more dynamical. One may ask whether the dynamics (and more precisely the initialization of the dynamics) matters. In particular, shall we still expect the failure of RM method for the previous example, if we take the initial output function $u_{\theta(0)}$ being sufficiently close to the exact solution $u$? We should study this in both numerical and theoretical way, and eventually this leads to the understanding of the implicit bias of RM methods.

Thanks to the well-known universal approximation theorem, the exact solution $u$ from the above failure example can be approximated well by, for example, a two-layer neural network. Thus we can first use supervised learning to find sufficiently good parameter and then apply RM methods with such good initialization.

In numerical experiments, we still take the example \eqref{eq::1dEq} and the results are shown in Figure \ref{fig::deviationofwelltrained}. We use supervised learning to find neural network function $u_{\theta}^{\mathrm{SV}}$ with parameter $\theta^{\mathrm{SV}}$, where the empirical risk function reads as
\begin{equation*}
    R_{S}^{\mathrm{SV}}(u_{\theta}^{\mathrm{SV}})=\frac{\abs{\Omega}}{n_{\mathrm{int}}}\sum_{x\in S_{\mathrm{int}}}\left((u_{\theta}^{\mathrm{SV}}(x)-u(x))^2+(D_xu_{\theta}^{\mathrm{SV}}(x)-D_xu(x))^2\right)+\gamma  R_{S,\mathrm{bd}}(u_{\theta}^{\mathrm{SV}}).
\end{equation*}
By Figure~\ref{fig::deviationofwelltrained} (a), $u_{\theta}^{\mathrm{SV}}$ almost overlaps $u$ as is expected. 
Next, let $\theta^{\mathrm{SV}}$ be the initial parameter and apply the RM method to the original equation \eqref{eq::1dEq} and the modified equation \eqref{eq::1dModEq}, respectively.
In other words, we train the neural network with $R_{S}$ and $\tilde{R}_{S}$, respectively, until the risk is small and does not decay anymore. After training, the output functions are denoted as $u_{\theta}^{\mathrm{SV}\to\mathrm{RM}}$ and $\tilde{u}_{\theta}^{\mathrm{SV}\to\mathrm{RM}}$, respectively. We observe that both output functions are very close to $\tilde{u}$, as shown in Figure~\ref{fig::deviationofwelltrained} (b). For a comprehensive comparison, we also plot $u_{\theta}$ and $\tilde{u}_{\theta}$ and their derivatives in Figure~\ref{fig::deviationofwelltrained}.
Roughly speaking, we have $u\approx u_{\theta}^{\mathrm{SV}}$ and $u_{\theta}\approx \tilde{u}_{\theta}\approx \tilde{u}\approx u_{\theta}^{\mathrm{SV}\to\mathrm{RM}}\approx \tilde{u}_{\theta}^{\mathrm{SV}\to\mathrm{RM}}$.
Therefore, \textbf{even given a sufficiently good initialization, the RM method implicitly biases the numerical solution against the exact solution $u$ and towards the solution $\tilde{u}$ to the modified equation.}

\begin{figure}[H]
	\centering	
	\subfigure[Output at initial]{\includegraphics[scale=0.35]{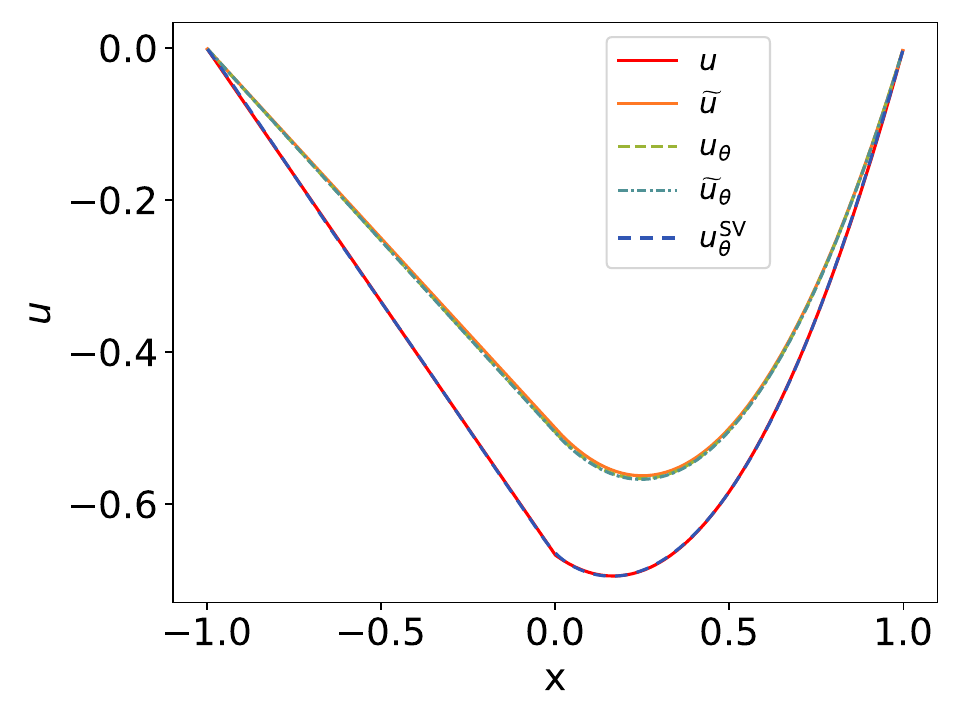}}
         \subfigure[Output at end]{\includegraphics[scale=0.35]{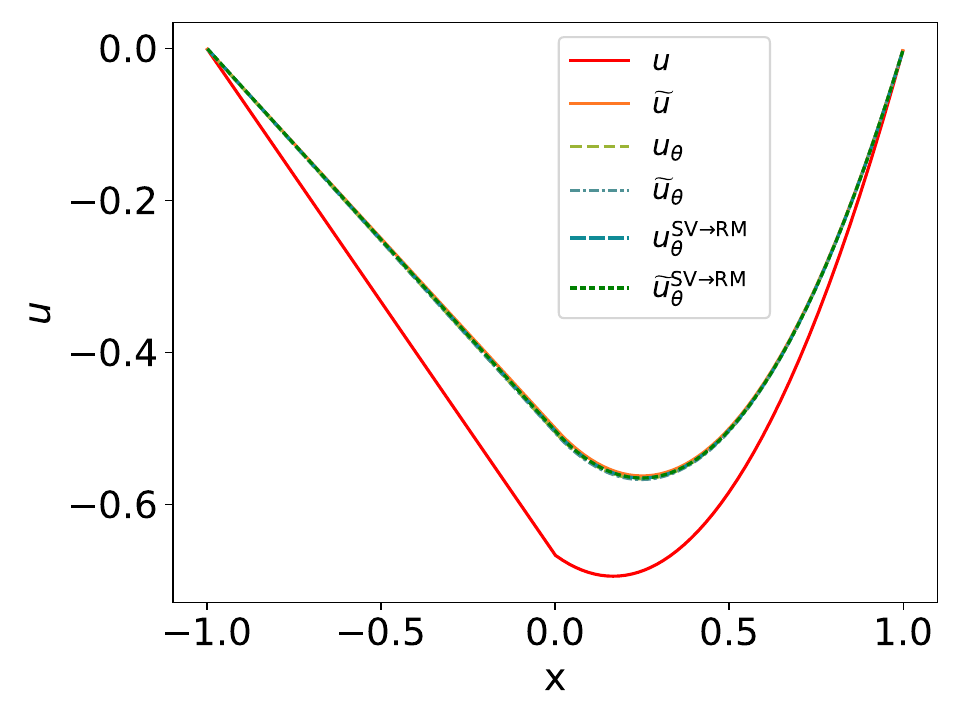}}
 
	\subfigure[Derivative of output at initial]{
	\includegraphics[scale=0.35]{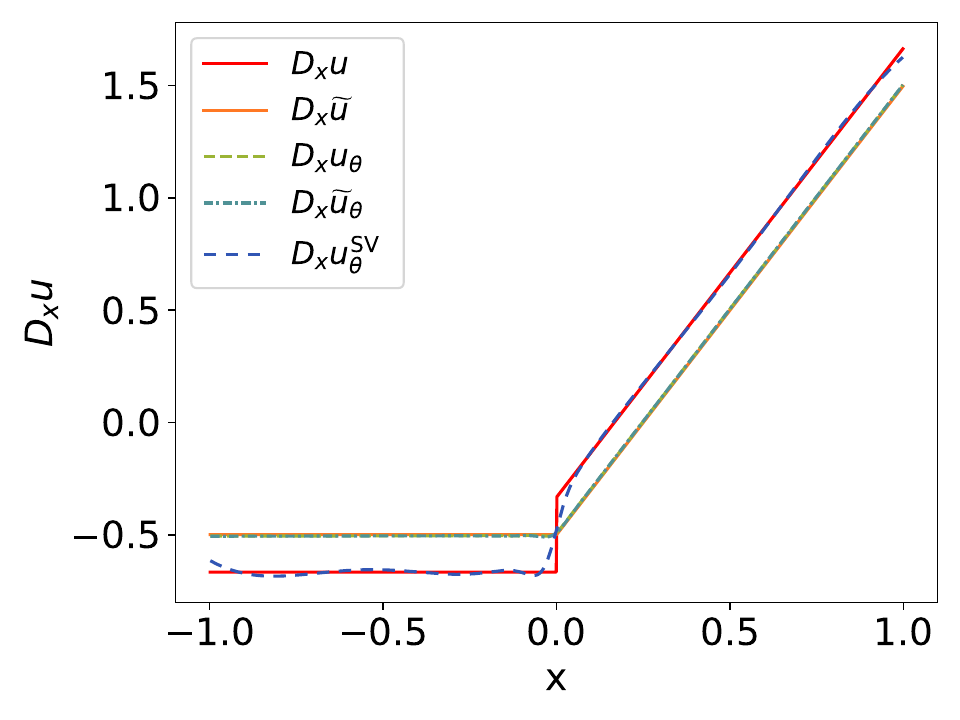}}
	\subfigure[Derivative of output at end]{
	\includegraphics[scale=0.35]{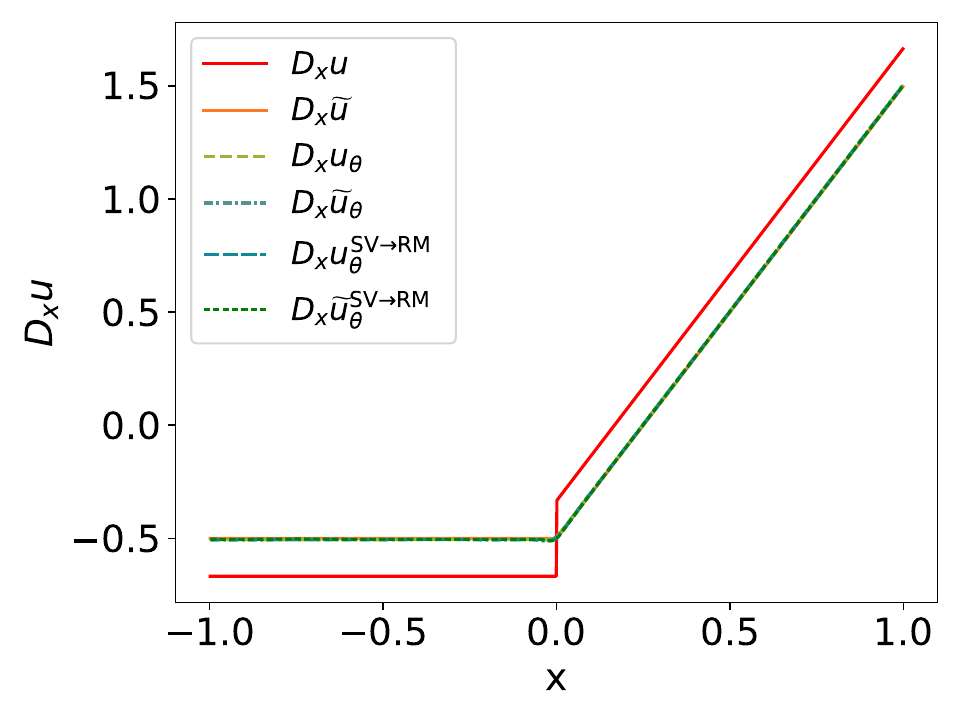}}

    \subfigure[$u$ v.s. $u_{\theta}^{\text{SV}}$ v.s. $\tilde{u}$]{
	\includegraphics[scale=0.22]{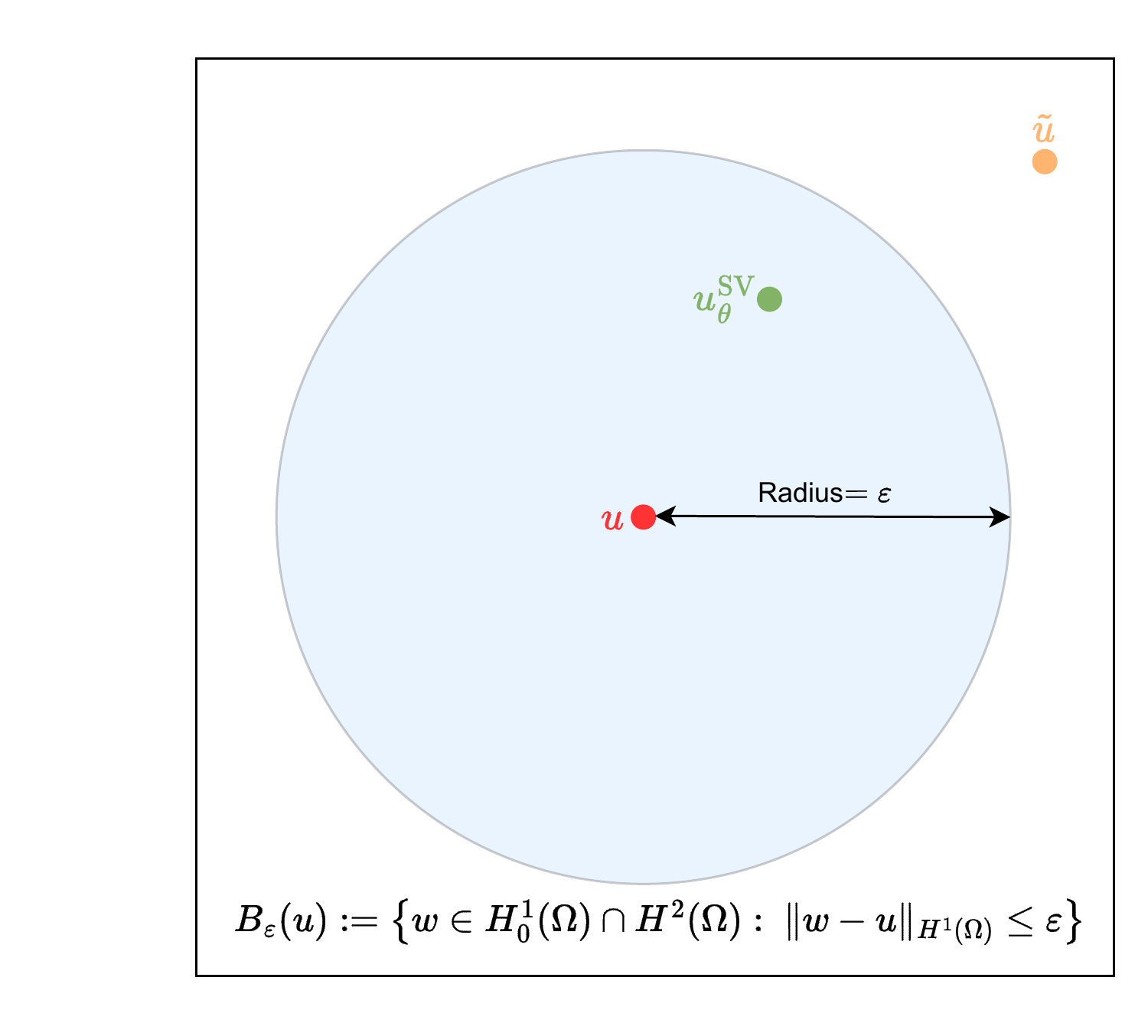}}\hspace{+3mm}
	\subfigure[$u$ v.s. $u_{\theta}^{\text{SV}\to \text{RM}}$ v.s. $\tilde{u}_{\theta}^{\text{SV}\to \text{RM}}$ v.s. $\tilde{u}$]{
	\includegraphics[scale=0.22]{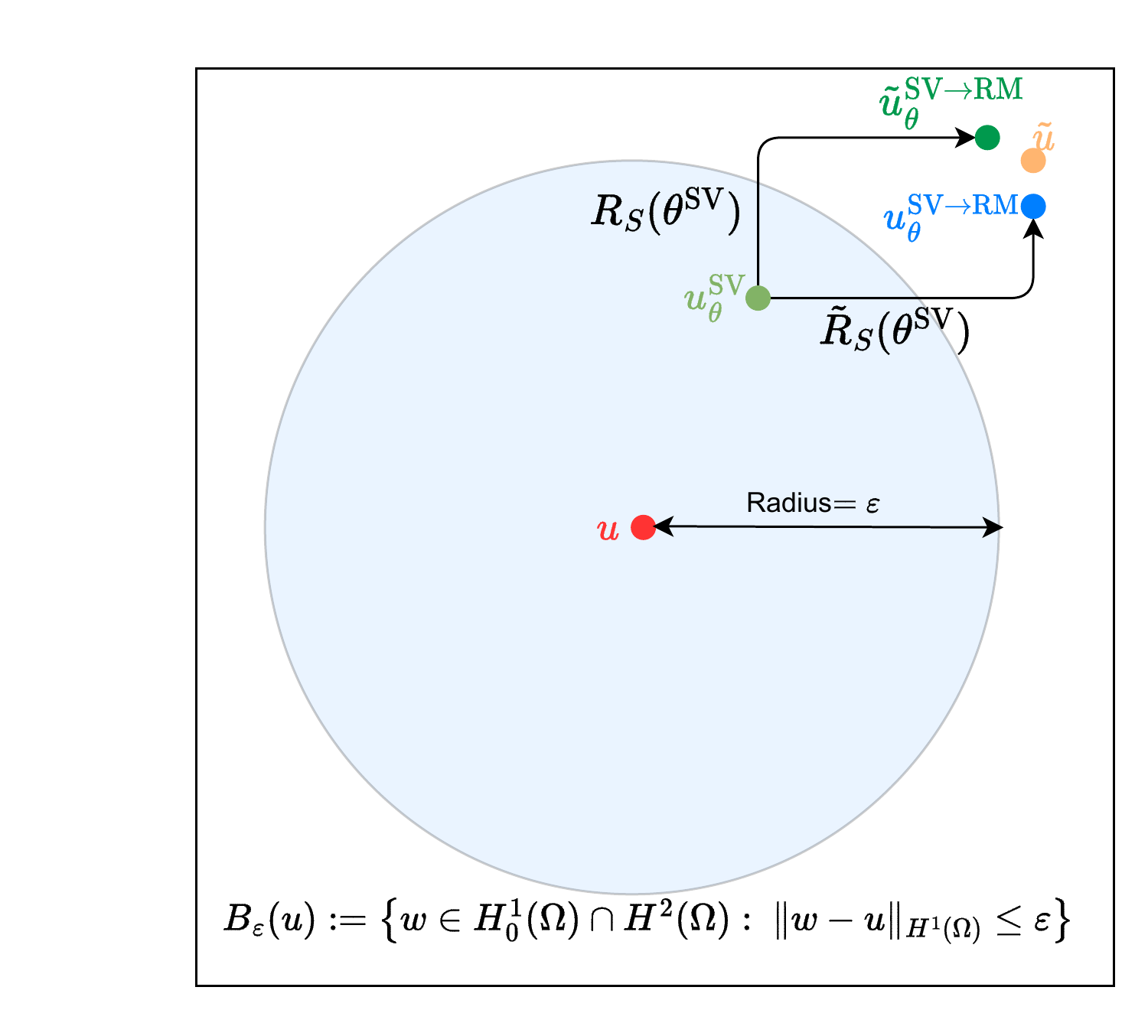}}
	\caption{evolution of modified solution with supervised initialization.} 
	\label{fig::deviationofwelltrained}
\end{figure}
We prove Theorems~\ref{thm::ImplicitBiasLinSys} and Proposition~\ref{prop::Bias} which essentially explains the implicit bias. In the following remark, we provide such theoretical explanation to the implicit bias phenomenon shown in Figure~\ref{fig::deviationofwelltrained}.

\begin{remark}[implicit bias of RM method]\label{rmk::ImplicitBias}
    (\romannumeral1) The effective risk is large at $u_{\theta}^{\mathrm{SV}}$. More precisely, $\tilde{R}_S(u_{\theta}^{\mathrm{SV}})\approx\tilde{R}(u_{\theta}^{\mathrm{SV}})\geq C_0$ for some finite $C_0>0$. Hence the risk is very likely to be decreased along the training dynamics. This explains the RM method implicitly biases against $u_{\theta}^{\mathrm{SV}}$.
    
    (\romannumeral2) Suppose that the RM method achieves a very small (effective) empirical risk after training. That is $\tilde{R}_S(u_{\theta}^{\mathrm{SV}\to\mathrm{RM}})\leq \eps$ for very small $\eps>0$. Then $$\norm{u_{\theta}^{\mathrm{SV}\to\mathrm{RM}}-\tilde{u}}_{H^1(\Omega;\sR^{d'})}\leq C\sqrt{\tilde{R}(u_{\theta}^{\mathrm{SV}\to\mathrm{RM}})}\approx C\sqrt{\tilde{R}_S(u_{\theta}^{\mathrm{SV}\to\mathrm{RM}})}\leq C\sqrt{\eps}.$$
    This explains the RM method implicitly biases towards $\tilde{u}$. 
\end{remark}

This contribution can also be understood as follows.
From the Observation (\romannumeral2), we have the commonly-seen phenomenon that $R(u_{\theta})\ll 1$ while $\norm{u-u_{\theta}}\ge C$ for some finite $C>0$. Now these experiments and phenomenon answers the reverse statement is also true: if $\norm{u-u_{\theta}}\ll 1$, then $R(u_{\theta})\ge C_0$ for some constant $C_0$. It shows that the exact solution is unstable in the sense of RM method and implicitly biases towards the solution to the modified equation.

\subsection{Connection of the contributions}
We conclude the main contribution section by presenting two figures within which the readers may find the connections between our main results mentioned above as well as more preliminary lemmas, propositions, etc. The arrows show the logic flow and various colors correspond to different types of the results.
\begin{figure}[H]
    \centering
    {\includegraphics[scale=0.65]{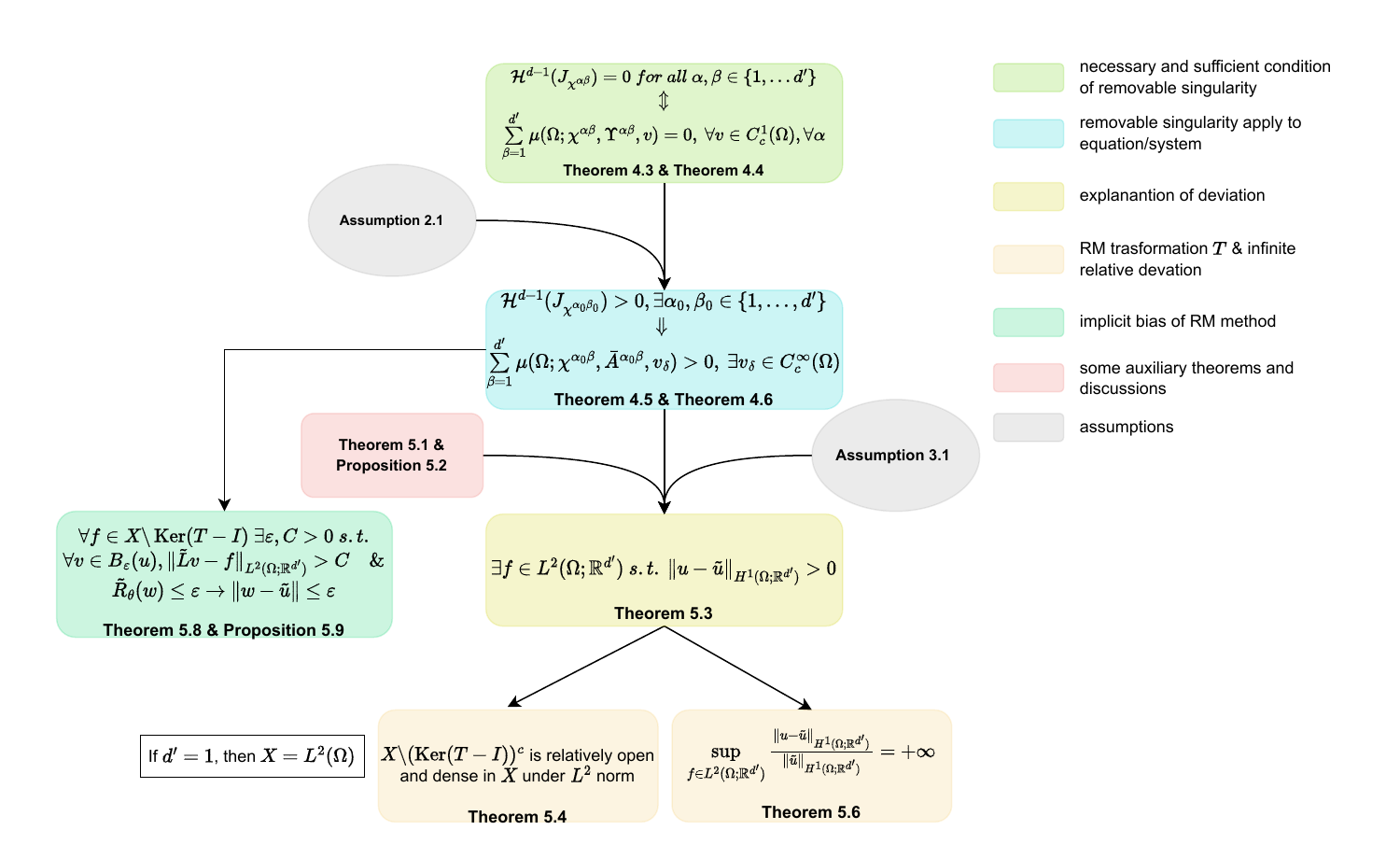}}
    \label{fig::ThmDiagram4EllipEq}
    \caption{The schematic proof of linear elliptic equation/system}
\end{figure}

\begin{figure}[H]
	\centering	
	{\includegraphics[scale=0.7]{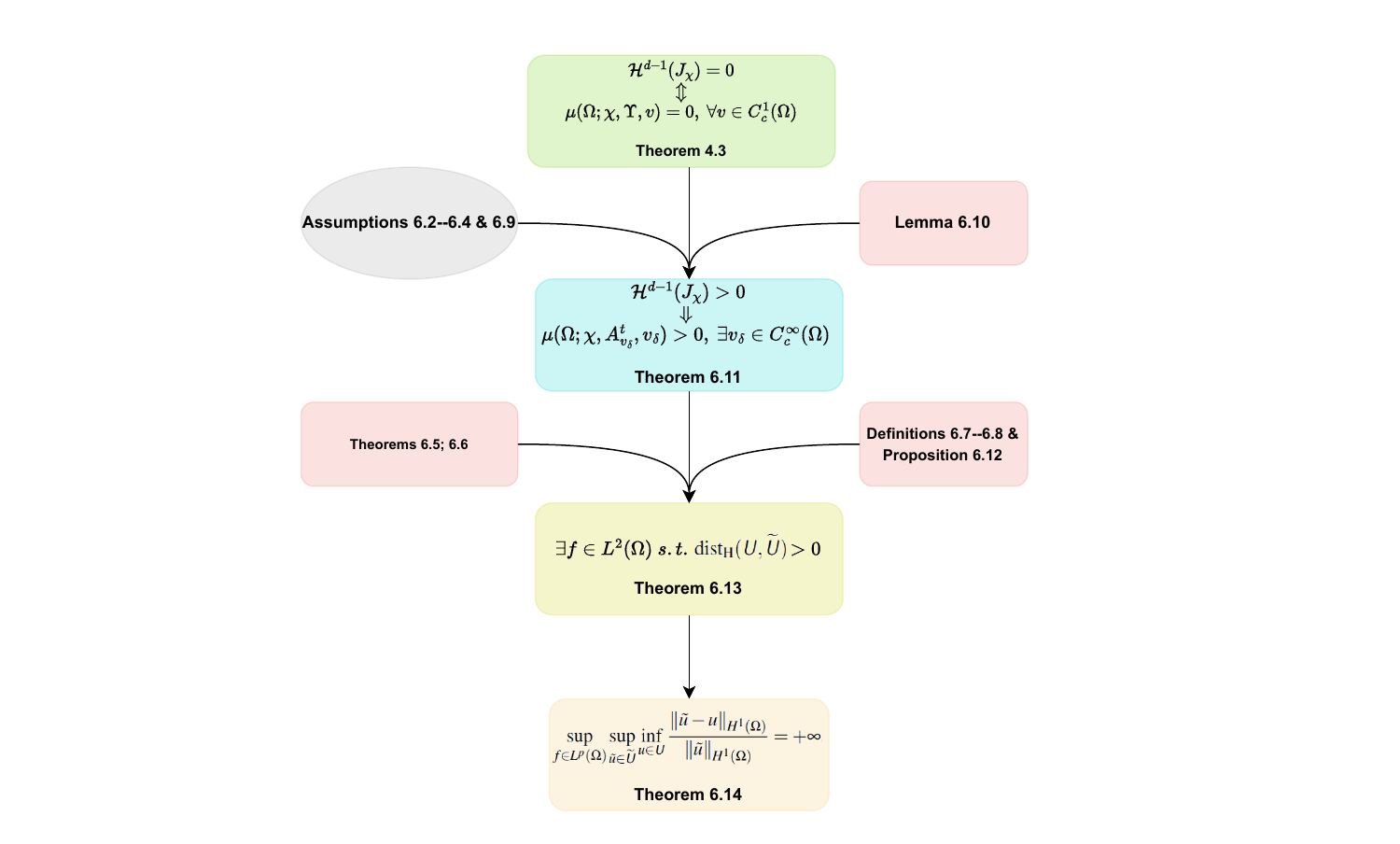}}
	\caption{The schematic proof of quasilinear elliptic equation}
	\label{fig::ThmDiagram4EllipSys}
\end{figure}

\section{Removability of singularity}\label{sec::RemovSing}
Let us explain the title of this section. For $\chi\in SBV(\Omega)$, we say the set of singularity $J_{\chi}$ is removable if $\fH^{d-1}(J_{\chi})=0$.
In this section, we obtain the equivalence between the removability of singularity and the condition $\mu=0$. The quantity $\mu$ is essentially an integral and defined by \eqref{eq::Def4mu} in Section~\ref{sec::CharacterRemovSing}. The advantage of using $\mu$, as what we do in latter sections, is that it allows us to estimate the pairing $\langle Lu-L\tilde{u},\varphi\rangle_{H^{-1}(\Omega),H^1_0(\Omega)}$ in a quantitative way, and hence we can estimate the difference $\norm{u-\tilde{u}}_{H^1(\Omega)}$.

\subsection{Necessary and sufficient condition of removable singularity}\label{sec::ThmRemovSing}
We begin with two simple lemmas: one to bound $\chi^{\pm}$ (defined in Definition~\ref{def::ApproxJump}) and the other to construct a smooth cutoff function. The latter one (namely Lemma \ref{lem::SmoothCutoffFunContrDer}) is standard, but we provide the proof for completeness.
\begin{lemma}[boundedness of $\chi^{\pm}$ on jump set]\label{lem::BddJump}
    If $\chi\in L^{\infty}(\Omega)\cap  SBV^{\infty}(\Omega)$ with $\fH^{d-1}(J_{\chi})>0$, then $\Abs{\chi^\pm(x)} \leq \norm{\chi}_{L^{\infty}(\Omega)}$ for all $x\in J_{\chi}$.
\end{lemma}
\begin{proof}
    We prove by contradiction. Suppose there is an $x_0\in J_{\chi}$ satisfying $\Abs{\chi^\pm(x_0)}>\norm{\chi}_{L^{\infty}(\Omega)}$. Denote $B^\pm_{\rho}(x_0,\nu)=\{x\in B_{\rho}(x_0)\colon\,\pm (x_0-x)^{\T}\nu>0\}$. Thus we have
    \begin{equation*}
        \lim_{\rho\to 0} \frac{1}{\Abs{B^\pm_{\rho}(x_0,\nu)}}\int_{B^\pm_{\rho}(x_0,\nu)} \Abs{\chi(y)-\chi^{\pm}(x_0)}\diff{y}>0,    
    \end{equation*}
    which contradicts the definition of $\chi^\pm$.
\end{proof}

\begin{lemma}[smooth cutoff function with controlled derivative]\label{lem::SmoothCutoffFunContrDer}
     Suppose that bounded open sets $U\subseteq U'\subseteq \sR^d$ satisfy $\dist(U,\partial U')>2\delta>0$. Then there is a cutoff function $\zeta\in C_c^{\infty}(U')$ such that $U \prec\zeta\prec U'$ and $\norm{D\zeta}_{L^{\infty}(U')}\leq \frac{C}{\delta}$, where $C$ depends only on $d$.
\end{lemma}

\begin{proof}
    Define $\rho(x)=\frac{1}{C}\exp(\frac{1}{\abs{x}^2-1})$ for $x\in B_1(0)$ and $\rho(x)=0$ for $x\in \sR^d\backslash B_1(0)$, where the constant $C=\int_{B_1(0)}\exp(\frac{1}{\abs{x}^2-1})\diff{x}$ only depends on $d$. Thus $\int_{\sR^d}\rho(x)\diff{x}=1$ and $\rho\in C_c^\infty(\sR^d)$ with compact support $\overline{B_1(0)}$.
    For each $\delta>0$, define $\rho_{\delta}(x)=\delta^{-d}\rho(\delta^{-1}x)$ and $\sone_{U_\delta}$ the characteristic function of $U_{\delta}$, where $U_{\delta}=\{x\in \sR^d\colon\,\operatorname{dist}(x,U)<\delta\}$. Let $\zeta = \rho_{\delta}*\sone_{U_{\delta}}$. It is clear that $U \prec\zeta\prec U'$. 
    Therefore, we have for all $x\in U'$
    \begin{align*}
        \Abs{D\zeta(x)}=\Abs{(\sone_{U_\delta} * D\rho_{\delta})(x)}
        &=\Abs{\int_{B_{\delta}(0)}\sone_{U_{\delta}}(x-y)\delta^{-d-1}D\rho(\delta^{-1}y)\diff{y}}\\
        &\leq \delta^{-1}\int_{B_1(0)}\Abs{D\rho(y)}\diff{y}
        \le\frac{C}{\delta},
    \end{align*}
    where $C$ only depends on $d$.
\end{proof}

\begin{figure}[H]
	\centering	
	\includegraphics[scale=0.45]{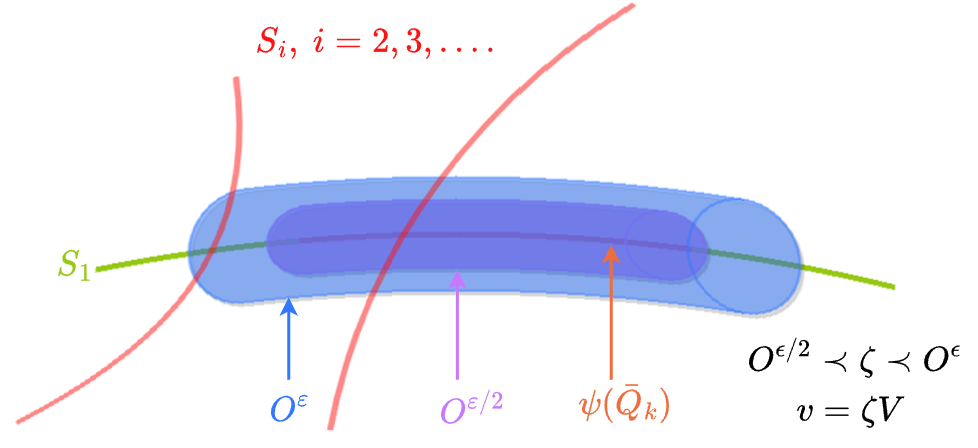}
	\caption{Schematic diagram of the proof of Theorem~\ref{thm::NecSufCondRemovSing}}
	\label{fig::SchematicDiagProof}
\end{figure}

Now we are ready to obtain the first important result in this work. That is a necessary and sufficient condition for the removable singularities of SBV functions. Some key quantities and sets constructed in the proof of Theorem~\ref{thm::NecSufCondRemovSing} are illustrated with various colors in the schematic diagram, namely Figure~\ref{fig::SchematicDiagProof}. To clarify our ideas, we also make three claims which play the roles as milestones on the road of our proof.
    
\begin{theorem}[necessary and sufficient condition of removable singularity]\label{thm::NecSufCondRemovSing}
    Let $\Upsilon: H_0^1(\Omega) \to L^{\infty}(\Omega; S^{d \times d})$ and $\chi\in L^{\infty}(\Omega)\cap SBV(\Omega)$ with $\fH^{d-1}(J_{\chi})< +\infty$. Suppose there exist constants $\lambda_0,\Lambda_0>0$ satisfying
    \begin{equation*}
        \lambda_0\abs{\xi}^2\le\xi^{\T}\Upsilon[w](x)\xi\le \Lambda_0\abs{\xi}^2
    \end{equation*}
    for all $w\in H_0^1(\Omega)$, $\xi\in \sR^d$, and $x\in \Omega$.
    We have $\fH^{d-1}(J_{\chi})=0$ if and only if     
    \begin{equation}\label{eq::JumpPartRemov}
        \mu(\Omega;\chi,\Upsilon,v)=0,\quad \forall v\in C_c^{1}(\Omega).
    \end{equation}
\end{theorem}

Before we prove this theorem, we would like to stress  that the sufficient part of this theorem is non-trivial.
Even for a special case where $\Upsilon[w]=\bar{A}$ holds for all $w\in H^1_0(\Omega)$, the result of Theorem \ref{thm::NecSufCondRemovSing} is not standard.
By the definition of $\mu$ in \eqref{eq::Def4mu}, the quantity in \eqref{eq::JumpPartRemov} with $\Upsilon[w]=\bar{A}$ reads as
\begin{equation*}
    \mu(\Omega;\chi,\bar{A},v)=\int_{J_{\chi}}(\chi^+-\chi^-)\nu_{\chi}^{\T} \bar{A}Dv\diff{\fH^{d-1}}=\bar{\mu}(Dv),
\end{equation*}
where $\bar{\mu}$ is a Radon measure on $\Omega$ such that for all $\bar{\varphi}\in C(\Omega;\sR^{d})$
\begin{equation*}
    \bar{\mu}(\bar{\varphi})=\int_{J_{\chi}}(\chi^+-\chi^-)\nu_{\chi}^{\T} \bar{A} \bar{\varphi} \diff{\fH^{d-1}}.
\end{equation*}
At first glance, it looks like a version of the fundamental lemma of the calculus of variation (sometimes also known as the Du Bois-Reymond lemma), which says that if a Radon measure $\hat{\mu}$ satisfying $\hat{\mu}(\hat{\varphi})=0$ for all $\hat{\varphi}\in C_c^\infty(\Omega)$, then $\hat{\mu}=0$ on $\Omega$.
However, the setting in Theorem \ref{thm::NecSufCondRemovSing} is quite different. In fact $Dv$ is in the form of a gradient. Therefore, what we have is not the fact that $\bar{\mu}(\bar{\varphi})=0$ for all $\bar{\varphi}\in C_c^\infty(\Omega;\sR^{d})$, but only the condition that $\bar{\mu}(Dv)=0$ for all $v\in C_c^{1}(\Omega)$. In general, such condition will not imply $\bar{\mu}=0$. In our proof of Theorem \ref{thm::NecSufCondRemovSing}, we essentially make use of the geometric structure of the jump part of BV function.

Besides, $\Upsilon$ takes a quite general form depending on $v$ and not necessarily being constant $\bar{A}$. This general form is necessary to develop our theory into the quasilinear case (See Section \ref{sec::ExtQausiLinEllipEq}).
\begin{proof}
    The necessary part of the statement is easy. In fact, by Corollary~\ref{cor::PropOfSBV}, $\fH^{d-1}(J_{\chi})=0$ implies $\chi\in W^{1,1}(\Omega)$, and hence $\mu(\Omega;\chi,\Upsilon,v)=0$ holds.    
    In the rest of the proof, we show by the method of contradiction that the condition~\eqref{eq::JumpPartRemov} is also sufficient. Suppose that $\mu(\Omega;\chi,\Upsilon,v)=0$ for all $v\in C_c^1(\Omega)$ and $\fH^{d-1}(J_{\chi})>0$.
    
    \begin{claim_01}
        There exist an $\fH^{d-1}$ measurable set $\Sigma$, a $C^1$-hypersurfaces $S$, and a function $V$ such that $\Sigma\subseteq J_{\chi}\cap S$ and
        \begin{equation*}
            \mu(\Sigma;\chi,\Upsilon,V)>0.
        \end{equation*}
    \end{claim_01}

    Let $J_{\chi}^{\pm}=\{x\in J_{\chi}\colon\,\chi^+(x)-\chi^-(x)\gtrless 0\}$. Clearly, $J_{\chi}^{\pm}$ are $\fH^{d-1}$ measurable sets.
    Noticing $\fH^{d-1}(J_{\chi})=\fH^{d-1}(J_{\chi}^+)+\fH^{d-1}(J_{\chi}^-)>0$, without loss of generality, we suppose that $\fH^{d-1}(J_{\chi}^+)>0$. By the structure theorem of $J_{\chi}$, namely Theorem~\ref{thm::PropOfApproxJumpSetForBV}, there is a countable sequence of $C^1$-hypersurfaces $\{S_i\}_{i=1}^{\infty}$ satisfying $\fH^{d-1}(J_{\chi}\backslash(\bigcup\nolimits_{i=1}^{\infty}S_i))=0$.
    Thus we can pick an $S$ from $\{S_i\}_{i=1}^{\infty}$ satisfying $\fH^{d-1}(J_{\chi}^+\cap S)>0$. 

    Note that $S$ is a $C^1$-hypersurface. Thus for each $z\in S$, there exist a $C^1$ function $h$, a permutation of coordinates mapping $\tau: \sR^d\to\sR^d, x\mapsto \tau(x)$, and an open ball $B_{r_z}(z'_{\tau})\subseteq \sR^{d-1}$ with $r_z>0$ such that $S$ can locally be represented by the graph of $h$: $(x_{\tau})_{d}=h(x'_{\tau})$, $x'_{\tau}\in B_{r_z}(z'_{\tau})$. Here the shorthand notation $z'_{\tau}$ means $((z_{\tau})_{1},\cdots,(z_{\tau})_{d-1})$. We also write $x_{\tau}=\tau(x)$ for simplicity. Consider the transformation $\phi: B_{r_z}(z'_{\tau})\times\sR\to B_{r_z}(z'_{\tau})\times\sR$, $x_{\tau}\mapsto y=\phi(x)$ defined by
    \begin{equation}
        \left\{\begin{array}{l}
        y_{i}=\phi_{i}(x_{\tau})=(x_{\tau})_{i},\quad i\in\{1, \ldots, d-1\}, \\
        y_{d}=\phi_{d}(x_{\tau})=(x_{\tau})_{d}-h\left(x_{\tau}'\right),
        \end{array}\right.
    \end{equation}
    and its inverse transformation $\phi=\psi^{-1}$
    \begin{equation}
        \left\{\begin{array}{l}
        (x_{\tau})_{i}= \psi_{i}(y)=y_{i},\quad i\in\{1, \ldots, d-1\},\\
        (x_{\tau})_{d}= \psi_{d}(y)=y_{d}+h\left(y'\right).
        \end{array}\right.
    \end{equation}
    In short, we write
    \begin{align}
        y
        &=\phi(x_{\tau})=(\phi_1(x_{\tau}),\ldots,\phi_d(x_{\tau})),\\
        x_{\tau}
        &=\psi(y)=(\psi_1(x_{\tau}),\ldots,\psi_d(x_{\tau})).
    \end{align}
    By this construction, we have $\Abs{\det D\phi}=\Abs{\det D\psi}=1$.

    Let $B_z=B_{r_z}(z'_{\tau})\times [-r_z,r_z]$. For each $y\in B_z$, define
    \begin{align*}
        \Psi(y)&=\tau^{-1}\circ\psi(y)=\tau^{-1}(y_1,\ldots,y_{d-1},y_d+h(y_1,\ldots,y_{d-1})).
    \end{align*}
    Thus $\Psi$ is a $C^1$ diffeomorphism between $B_z$. Denote $U_z=\Psi(B_z)$. If we define for $x\in U_z$
    \begin{align*}
        \Phi(x)
        &=\phi\circ\tau(x)\\
        &=\left((x_{\tau})_{1},\ldots,(x_{\tau})_{d-1},(x_{\tau})_{d}-h\left((x_{\tau})_{1},\ldots,(x_{\tau})_{d-1}\right)\right),
    \end{align*}
    then $\Phi\circ\Psi=I$ on $B_z$ and $\Psi\circ\Phi=I$ on $U_z$ with $\Abs{\det D\Phi}=\Abs{\det D\Psi}=1$. See Figure~\ref{fig::CoordTrans} for the relations between these transformations.

    \begin{figure}[H]
	\centering	
	\includegraphics[scale=0.45]{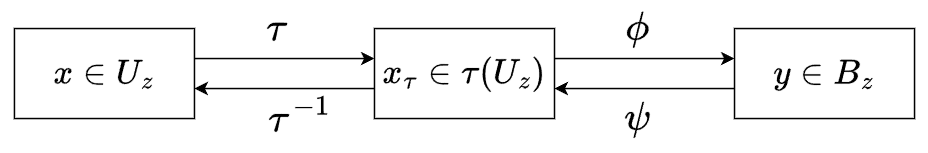}
	\caption{Sketch map of coordinate transformations between $U_z$ and $B_z$.}
	\label{fig::CoordTrans}
    \end{figure}

    Let $V(x)=\Phi_d(x)$ for all $x\in U_z$. Clearly, $V(x)=0$ on $U_z\cap 
 S$.
    By choosing $r_z$ small enough, we can require that $0<\abs{DV(x)}<C$ on $U_z$, where $C$ only depends on $d$ and $J_{\chi}$. Let us choose an appropriate $U_z$ such that $\fH^{d-1}(\Psi(B_{r_z/2}(z'_{\tau})\times [-\frac{r_z}{2},\frac{r_z}{2}])\cap S)>0$. Thanks to Lindelöf's lemma, such $U_z$ exists because the manifold $S$ is second countable and a countable union of such $\Psi(B_{r_z/2}(z'_{\tau})\times [-\frac{r_z}{2},\frac{r_z}{2}])$ must cover $S$, and hence also cover $J_{\chi}^+\cap S$. Let $\Sigma=J_{\chi}^+\cap \Psi(B_{r_z/2}(z'_{\tau})\times \{0\})\subseteq J_{\chi}\cap S$ and $\Sigma'=J_{\chi}^+ \cap U_z\cap S.$
    
    Note that $DV= \Abs{DV}\nu_{\chi}$ on $\Sigma'$ and
    $\chi^+-\chi^->0$ on $\Sigma$. Therefore
    \begin{equation*}
        \mu(\Sigma;\chi,\Upsilon,V)=\int_{\Sigma}(\chi^+-\chi^-)\nu_{\chi}^{\T} \Upsilon[V]\nu_\chi\Abs{DV}\diff{\fH^{d-1}}>0,
    \end{equation*}
    where we use the uniform ellipticity of $\Upsilon[V]$. 

    \begin{claim_01}
        There is a ``$(d-1)$-dimensional cube'' $Q_{k}\subseteq B_{r_z/2}(z_{\tau}')$ with $\bar{Q}_k=Q_k\times\{0\}$ satisfying
        \begin{equation*}
            \mu(\Psi(\bar{Q}_k);\chi,\Upsilon,V)>0.
        \end{equation*}
    \end{claim_01} 
    
    For each $\delta>0$, there exists countably many dyadic cubes $\{Q_{k}\}_{k=1}^{\infty}$ which are almost disjoint and satisfy $Q_k\subseteq B_{r_z/2}(z'_{\tau})$, $\Sigma\subseteq \Psi\left(\bigcup_{k=1}^{\infty}\bar{Q}_{k}\right)$,
    and $\fH^{d-1}\left(\Psi(\bigcup_{k=1}^{\infty}\bar{Q}_{k})\backslash \Sigma\right)< \delta$. Thus we have
    \begin{align}\label{eq::DyadicCube}
        \sum_{k=1}^{\infty}\mu(\Psi(\bar{Q}_k);\chi,\Upsilon,V)
        &=\mu\left(\Psi(\textstyle \bigcup_{k=1}^{\infty}\bar{Q}_{k});\chi,\Upsilon,V\right)\nonumber\\
        &=\mu(\Sigma;\chi,\Upsilon,V)+\mu\left(\Psi(\textstyle \bigcup_{k=1}^{\infty}\bar{Q}_{k})\backslash \Sigma;\chi,\Upsilon,V\right)\nonumber\\ 
        &\geq \mu(\Sigma;\chi,\Upsilon,V)-2\norm{\chi}_{L^{\infty}(\Omega)}\Lambda_0\norm{DV}_{L^{\infty}(U_z)}\fH^{d-1}\left(\Psi(\textstyle \bigcup_{k=1}^{\infty}\bar{Q}_{k})\backslash \Sigma\right)\nonumber\\ 
        &>0,
    \end{align}
    where the first equality is due to that $\{Q_{k}\}_{k=1}^{\infty}$ are almost disjoint, in the third inequality we use Lemma~\ref{lem::BddJump} and uniform ellipticity, and $\delta$ is taken small enough in the last line. Thus there is some $Q_k$ satisfying $\mu(\Psi(\bar{Q}_k);\chi,\Upsilon,V)>0$.
    
    \begin{claim_01}
        There is a function $v\in C_c^1(\Omega)$ such that
        \begin{equation*}
            \mu(\Omega;\chi,\Upsilon,v)>0.
        \end{equation*}
    \end{claim_01}

    This claim contradicts $\mu(\Omega;\chi,\Upsilon,v)=0$, and hence the statement of the theorem is completed. 
    Towards the claim, we consider the $\eps$-neighborhood of $\Psi(\bar{Q}_k)$ as follows
    \begin{equation*}
        O^{\eps}=\{y\in \Omega\colon\,\Abs{y-x}<\eps  \text{ for some } x\in \Psi(\bar{Q}_k)\},    
    \end{equation*}
    where $\eps < \frac{1}{2}\dist\left(\Psi(\bar{Q}_k), \partial U_z\right)$. This is admissible since     \begin{equation}\label{eq::DistCubeLin}
        \dist\left(\Psi(\bar{Q}_k),\partial U_z\right)> \dist\left(\Psi( B_{r_z/2}(z'_{\tau})\times [-\frac{r_z}{2},\frac{r_z}{2}]), \partial U_z\right)\geq 0.
    \end{equation}
    The neighborhood $O^{\eps/2}$ is defined similarly. By Lemma~\ref{lem::SmoothCutoffFunContrDer} with  $\dist(O^{\eps/2},\partial O^{\eps})>\frac{\eps}{4}$, there is a cutoff function $\zeta\in C_c^{\infty}(O^{\eps})$ such that
    \begin{equation}\label{eq::D_etaUpperBound}
        O^{\eps/2} \prec \zeta \prec O^{\eps} \text{ and }\norm{D\zeta}_{L^{\infty}(O^{\eps})}\leq \frac{C}{\eps},
    \end{equation}  
    where $C$ depends only on $d$.
    Let $v = \zeta V$.    
    Note that $\mu\left(O^{\eps};\chi,\Upsilon,v\right)=\mu\left(\Omega;\chi,\Upsilon,v\right)$ due to $v\in C_c^1(O^{\eps})$. It is sufficient to consider the integration $\mu\left(O^{\eps};\chi,\Upsilon,v\right)$.  
    Let $O_{S}^{\eps}=O^{\eps}\cap S$. Take the decomposition as shown in Figure~\ref{fig::SchematicDiagProof}:
    \begin{equation*}
        \mu\left(O^{\eps};\chi,\Upsilon,v\right)
        =\mu\left(O_{S}^{\eps};\chi,\Upsilon,v\right)+\mu\left(O^{\eps}\backslash O_{S}^{\eps};\chi,\Upsilon,v\right).
    \end{equation*}
    We will later show that the first integration has a lower bound, that is,
    $$
        \mu\left(O_{S}^{\eps};\chi,\Upsilon,v\right)\geq \frac{3}{4} \mu(\Psi(\bar{Q}_k);\chi,\Upsilon,V),
    $$
    while the second integration has a small contribution, that is,
    $$
        \Abs{\mu\left(O^{\eps}\backslash O_{S}^{\eps};\chi,\Upsilon,v\right)}\leq \frac{1}{2}\mu(\Psi(\bar{Q}_k);\chi,\Upsilon,V).
    $$
    These estimates with the decomposition together lead to 
    $$
    \mu\left(O^{\eps};\chi,\Upsilon,v\right)\geq \frac{1}{4}\mu(\Psi(\bar{Q}_k);\chi,\Upsilon,V)>0,
    $$
    and therefore the proof is completed.
    
    Notice that $V(x) = 0$ for all $x\in \Psi(B_{r_z}(z'_{\tau})\times\{0\})$. Thus
    \begin{equation}\label{eq::V(x)UpperBound}
        V(x)\leq \eps\norm{DV}_{L^{\infty}(O^{\eps})} <C\eps,\quad \forall x\in O^{\eps},
    \end{equation}
    where $C$ depends on $d$ and $J_{\chi}$.
    This with \eqref{eq::D_etaUpperBound} leads to
    \begin{equation}\label{eq::BoundDeriOfC1TestFunc}
        \Abs{Dv(x)}\leq \Abs{D\zeta(x)}\Abs{V(x)}+\Abs{\zeta(x)}\Abs{DV(x)}<C,\quad \forall x\in O^{\eps},
    \end{equation} 
    where $C$ depends on $d$ and $J_{\chi}$.
    By monotonicity of the measure $\fH^{d-1}$, we have
    \begin{equation*}
        \lim_{\eps\to 0} \fH^{d-1}\left(O_{S}^{\eps}\cap J_{\chi}\right)= \fH^{d-1}\left(\Psi(\bar{Q}_k)\cap J_{\chi}\right).    
    \end{equation*}
    Hence for small enough constant $\eps_1>0$, we obtain
    \begin{align*}
        \mu\left(O_{S}^{\eps};\chi,\Upsilon,v\right)
        &=\mu\left(\Psi(\bar{Q}_k);\chi,\Upsilon,v\right)+\mu\left(O_{S}^{\eps}\backslash \Psi(\bar{Q}_k);\chi,\Upsilon,v\right)\\
        &=\mu\left(\Psi(\bar{Q}_k);\chi,\Upsilon,V\right)+ \int_{(O_{S}^{\eps}\backslash \Psi(\bar{Q}_k))\cap J_{\chi}}(\chi^+-\chi^-)\nu_{\chi}^{\T} \Upsilon[v]Dv\diff{\fH^{d-1}}\\
        &\geq \mu\left(\Psi(\bar{Q}_k);\chi,\Upsilon,V\right)-2\norm{\chi}_{L^{\infty}(\Omega)}\Lambda_0\norm{Dv}_{L^{\infty}(O^{\eps})}\fH^{d-1}\left((O_{S}^{\eps}\backslash \Psi(\bar{Q}_k))\cap J_{\chi}\right)\\
        &\geq \frac{3}{4}\mu\left(\Psi(\bar{Q}_k);\chi,\Upsilon,V\right),
    \end{align*}
    where the second equality is due to $v=V$ on $\Psi(\bar{Q}_k)\cap J_{\chi}$ and the last step holds by choosing $\eps<\eps_1$ .
    
    Similarly, by monotonicity of the measure $\fH^{d-1}$, we have
    \begin{equation*}
        \lim_{\eps\to 0}\fH^{d-1}\left((O^{\eps}\backslash O_{S}^{\eps})\cap J_{\chi}\right)=0.
    \end{equation*}
    By choosing $\eps<\eps_2$ for small enough constant $\eps_2>0$ and combining~\eqref{eq::BoundDeriOfC1TestFunc}, we obtain 
    \begin{align*}
        \Abs{\mu\left(O^{\eps}\backslash O_{S}^{\eps};\chi,\Upsilon,v\right)}
        &\leq 2\norm{\chi}_{L^{\infty}(\Omega)}\Lambda_0\norm{Dv}_{L^{\infty}(O^{\eps})}\fH^{d-1}\left((O^{\eps}\backslash O_{S}^{\eps})\cap J_{\chi}\right)\\
        &\leq \frac{1}{2}\mu\left(\Psi(\bar{Q}_k);\chi,\Upsilon, V\right).
    \end{align*}
    As discussed above, this completes the whole proof.
\end{proof}

Now we extend this result to the case of systems. 

\begin{theorem}[necessary and sufficient condition of removable singularity for system]\label{thm::NecSufCondRemovSing4Sys}
    Let $\Upsilon^{\alpha\beta}: H_0^1(\Omega) \to L^{\infty}(\Omega; S^{d \times d})$ and $\chi^{\alpha\beta}\in L^{\infty}(\Omega)\cap SBV(\Omega)$ with $\fH^{d-1}(J_{\chi^{\alpha\beta}})< +\infty$ for each $\alpha, \beta\in\{1,\ldots,d'\}$ with $d'\geq 1$.
    Suppose that there exist $\lambda_0,\Lambda_0>0$ such that for all $\alpha,\beta\in\{1,\ldots,d'\}$, $w\in H^1_0(\Omega)$, $\xi\in \sR^{d}$, $x\in \Omega$
    \begin{equation*}
         \lambda_0\abs{\xi}^{2}\leq \xi^\T \Upsilon^{\alpha\beta}[w]\xi \leq \Lambda_0\abs{\xi}^{2}.
    \end{equation*}

    Then $\fH^{d-1}(J_{\chi^{\alpha\beta}})=0$ for all $\alpha, \beta\in\{1,\ldots,d'\}$ if and only if
    \begin{equation}~\label{eq::JumpPartRemovSys}
        \sum_{\beta=1}^{d'}\mu(\Omega;\chi^{\alpha\beta},\Upsilon^{\alpha\beta},v^{\beta})=0,\quad \forall v\in C^1_c(\Omega;\sR^{d'}),\,\,\forall \alpha\in\{1,2,\ldots,d'\}.
    \end{equation}
\end{theorem}

\begin{proof}
    The proof is a standard adaptation of proof of Theorem~\ref{thm::NecSufCondRemovSing}. We highlight some key modification, omit the details, and outline the proof by just stating three claims which are counterpart of those in Theorem~\ref{thm::NecSufCondRemovSing}. The necessary part is clear. As above, we show by the method of contradiction that the condition~\eqref{eq::JumpPartRemovSys} is also sufficient still by dividing this proof into three parts. Suppose that for some $\alpha_0, \beta_0\in\{1,\ldots,d'\}$, $\sum_{\beta=1}^{d'}\mu(\Omega;\chi^{\alpha_0\beta},\Upsilon^{\alpha_0\beta}, v^{\beta})=0$ for all $v\in C_c^1(\Omega;\sR^{d'})$ and $\fH^{d-1}(J_{\chi^{\alpha_0\beta_0}})>0$.
    
    \begin{claim_02}
        There are $\alpha_0,\beta_0 \in \{1,\ldots,d'\}$, an $\fH^{d-1}$ measurable set $\Sigma$, a $C^1$-hypersurface $S$ and a function $V\in C^1(\Omega;\sR^{d'})$ such that $\Sigma\subseteq J_{\chi^{\alpha_0\beta_0}}\cap S$ and
        \begin{equation*}
            \mu(\Sigma;\chi^{\alpha_0\beta_0},\Upsilon^{\alpha_0\beta_0},V^{\beta_0})>0
        \end{equation*}
    \end{claim_02}

    Following the proof of Claim 1 in Theorem~\ref{thm::NecSufCondRemovSing} with $J_{\chi}$ replaced by $J_{\chi^{\alpha_0\beta_0}}$, we can construct $V^{\beta_0}(x)=\Phi_d(x)$ for all $x\in U_z$. Thus by defining $V=(0,\ldots,V^{\beta_0},\ldots,0)^{\T}\in C^1(U_z;\sR^{d'})$, it is easy to check that
    \begin{equation*}
        \mu(\Sigma;\chi^{\alpha_0\beta_0},\Upsilon^{\alpha_0\beta_0},V^{\beta_0})>0.
    \end{equation*}
    Comparing with the proof of Theorem~\ref{thm::NecSufCondRemovSing}, we emphasize that the new ingredient the definition of $V$ which is a vector valued function with only one non-zero entry. By this construction of $V$, we further see that for all $\fH^{d-1}$ measurable set $B$,
    \begin{equation*}
        \sum_{\beta=1}^{d'}\mu(B;\chi^{\alpha_0\beta},\Upsilon^{\alpha_0\beta},V^{\beta})=\mu(B;\chi^{\alpha_0\beta_0},\Upsilon^{\alpha_0\beta_0},V^{\beta_0}).
    \end{equation*}

    For the rest of the proof, we simply state the two claims and omit their proofs for which one can refer to the proof of Theorem~\ref{thm::NecSufCondRemovSing}.

    \begin{claim_02}
        There are $\alpha_0, \beta_0\in \{1,\ldots,d'\}$ and a ``$(d-1)$-dimensional cube'' $Q_{k}\subseteq B_{r_z/2}(z_{\tau}')$ with $\bar{Q}_k=Q_k\times\{0\}$ such that 
        \begin{equation*}
            \mu(\Psi(\bar{Q}_k);\chi^{\alpha_0\beta_0},\Upsilon^{\alpha_0,\beta_0},V^{\beta_0})>0.
        \end{equation*}
    \end{claim_02}

    \begin{claim_02}
        There are $\alpha_0, \beta_0\in \{1,\ldots,d'\}$ and a function $v\in C_c^1(\Omega;\sR^{d'})$ such that
        \begin{equation*}
            \sum_{\beta=1}^{d'}\mu(\Omega;\chi^{\alpha_0\beta},\Upsilon^{\alpha_0,\beta},v^{\beta})=\mu(\Omega;\chi^{\alpha_0\beta_0},\Upsilon^{\alpha_0,\beta_0},v^{\beta_0})>0. 
        \end{equation*}
    \end{claim_02}
\end{proof}

\subsection{Singularity in linear elliptic equations}\label{sec::SingLinEllipEq}

Now we apply the removable singularity theorems (Theorem~\ref{thm::NecSufCondRemovSing} and Theorem~\ref{thm::NecSufCondRemovSing4Sys}) to the case of linear elliptic equations as well as systems.

In Theorem~\ref{thm::NecSufCondRemovSingLin} (as well as its counterpart for system, namely Theorem~\ref{thm::NecSufCondRemovSingLinSys}), we construct a more smooth function $v_{\delta}$ in $C_c^{\infty}(\Omega)$ instead of a function $v$ just in $C_c^1(\Omega)$, as in the statement of Theorem~\ref{thm::NecSufCondRemovSing}. On the one hand, to guarantee the well-definedness of $\tilde{L}v_{\delta}$ in $L^2$ space, $v_{\delta}$ needs to be sufficiently regular, at least belonging to $W^{2,p}(\Omega)$. On the other hand, to quantify the deviation of the numerical solution, we have to achieve a particular solution $v_{\delta}$ to the modified equation \eqref{eq::ModEq} with a carefully chosen $f$. Therefore, this improvement on the regularity of $v_{\delta}$ is inevitable.

\begin{theorem}[characterization of removable singularity with smooth function for equation]\label{thm::NecSufCondRemovSingLin}
    Suppose that Assumption~\ref{assum::BVCoeff} holds with $d'=1$. Then there exists a $v_{\delta}\in C_c^{\infty}(\Omega)$ such that 
    \begin{equation}\label{eq::RemovSingLinEq}
        \mu(\Omega;\chi,\bar{A},v_{\delta})> 0.
    \end{equation} 
\end{theorem}
\begin{proof}
    Let $\Upsilon[w]=\bar{A}$ for all $w\in H_0^1(\Omega)$. By Theorem~\ref{thm::NecSufCondRemovSing} there is a $v\in C_c^1(\Omega)$ such that
    $$
        \mu(\Omega;\chi,\bar{A},v)=\int_{J_{\chi}}(\chi^+-\chi^-)\nu_{\chi}^{\T} \bar{A}Dv\diff{\fH^{d-1}}>0.
    $$
    To show~\eqref{eq::RemovSingLinEq}, we approximate $v$ by a smooth function $v_{\delta}\in C_c^\infty(\Omega)$. Similar to the proof of Lemma~\ref{lem::SmoothCutoffFunContrDer}, choose a mollifier $\rho\in C_c^{\infty}(\sR^d)$ with compact support $\overline{B_1(0)}$. For any $\delta>0$, define $\rho_{\delta}(x)=\delta^{-d}\rho(\delta^{-1}x)$, $x\in\sR^d$ and $v_{\delta} = v*\rho_{\delta}$ with $\delta<\frac{1}{4}\dist\left(\Psi(\bar{Q}_k),\partial U_z)\right)$ (See~\eqref{eq::DistCubeLin}). Thus We have $Dv_{\delta}\in C_c^{\infty}(\Omega)$ satisfying $Dv_{\delta} = v*D\rho_{\delta}$. 
    Since $v\in C_c^1(\Omega)$, we obtain 
    \begin{equation*}
        \lim_{\delta\to 0}v*D\rho_{\delta} = Dv\text{ uniformly in }\Omega.
    \end{equation*} 
    Since $\fH^{d-1}(J_{\chi})<+\infty$, we have for $\delta$ small enough
    \begin{equation*}
        \abs{\mu(\Omega;\chi,\bar{A},v_{\delta}-v)}
        \leq 2\chi_{\max}\bar{\Lambda}\fH^{d-1}(J_{\chi})\Abs{Dv_{\delta}-Dv}
        \leq \frac{1}{2}\mu(\Omega;J_{\chi},\bar{A},v).
    \end{equation*}
    Thus we complete the proof by 
    \begin{equation*}
        \mu(\Omega;\chi,\bar{A},v_{\delta})=\mu(\Omega;\chi,\bar{A},v_{\delta}-v)+\mu(\Omega;\chi,\bar{A},v)\geq \frac{1}{2}\mu(\Omega;\chi,\bar{A},v)>0.
    \end{equation*}
\end{proof}

\begin{theorem}[characterization of removable singularity with smooth function for system]\label{thm::NecSufCondRemovSingLinSys}
    Suppose that Assumption~\ref{assum::BVCoeff} holds with $d'\geq 1$. Then there exists a $v_{\delta}\in C_c^{\infty}(\Omega;\sR^{d'})$, such that 
    \begin{equation}\label{eq::RemovSingLinEqsystem}
        \sum_{\beta=1}^{d'}\mu(\Omega;\chi^{\alpha_0\beta},\bar{A}^{\alpha_0\beta},v_{\delta}^{\beta})> 0.
    \end{equation} 
\end{theorem}
\begin{proof}
     Apply Theorem~\ref{thm::NecSufCondRemovSing4Sys} with $\Upsilon^{\alpha\beta}[w]=\bar{A}^{\alpha\beta}$ for all $\alpha,\beta\in\{1,\ldots,d'\}$ and $w\in H_0^1(\Omega)$. Therefore, the proof is similar to that of Theorem~\ref{thm::NecSufCondRemovSingLin}. 
\end{proof}

\section{Deviation and implicit bias of RM methods }\label{sec::Lin}
Based on the characterization of the singularities, we are ready to study the deviation and implicit bias of RM formulation for linear elliptic equations and systems. In Section~\ref{sec::ExistAprioEstLinEllipEq}, we state some preliminary results on the existence and \emph{a priori} estimates of linear elliptic equations and systems which will be applied to equations~\eqref{eq::OriginEq} and \eqref{eq::ModEq}. In Section~\ref{sec::WorstCaseDevRMSol}, we show that the solution to the modified equation is deviated from the one to the original equation, which is a theoretical explanation to experiments mentioned in Section~\ref{sec::FailExPINN}. Consequently, we ask what kind of data $f$ gives rise to such deviation and to what extent does such $f$ occupy in $L^2$ space. We give a characterization by utilizing the RM-transformation $T$ in Section~\ref{sec::RMTransModEqLin}, which gives a complete answer to the above question. In addition to the deviation in Section~\ref{sec::WorstCaseDevRMSol}, we also study relative deviation in Section~\ref{sec::RelaDevia}.  Finally, in Section~\ref{sec::ImplicitBias}, we prove the implicit bias of RM method towards the solution to the modified equation. In fact, even if we choose initialization $\theta(0)$ such that the output function, $u_{\theta(0)}$, is close enough to the true solution of equation~\eqref{eq::OriginEq}, $u$, after training via gradient flow with RM risk, the parameters $\theta$ will evolve and converge to some $\theta(\infty)$ with $u_{\theta(\infty)}$ which is very close to the solution to equation~\eqref{eq::ModEq} $\tilde{u}$. This shows the RM methods at the exact solution $u$ is unstable under small perturbation, and also the RM methods implicitly biases towards the solution to the modified equation.

\subsection{Existence and \emph{a priori} estimates}\label{sec::ExistAprioEstLinEllipEq}
We prove the existence of solutions to the original and modified equations together with their \emph{a priori} estimates. These results are standard, and we include them here for completeness.

\begin{theorem}[existence and \emph{a priori} estimates for linear elliptic equations and systems]\label{thm::ExistAprioEstLinEllipEq&Sys}
~
    \begin{enumerate}[(i).]
        \item Suppose that Assumption~\ref{assum::BVCoeff} holds with $d'\geq  1$. For any $f\in L^2(\Omega;\sR^{d'})$, system~\eqref{eq::OriginEq} has a unique solution $u\in H^1_0(\Omega;\sR^{d'})$. Moreover, there exists a constant $C>0$ such that 
        \begin{equation*}
            \norm{u}_{H^1(\Omega;\sR^{d'})}\leq  C\norm{f}_{L^2(\Omega;\sR^{d'})}.
        \end{equation*}
        
        \item Suppose that Assumption~\ref{assum::BVCoeff} holds with $d'\geq 1$. For any $f\in H^{-1}(\Omega;\sR^{d'})$,  system~\eqref{eq::OriginEq} has a unique solution $u\in H^1_0(\Omega;\sR^{d'})$. Moreover, there exists a constant $C>0$ such that 
        \begin{equation*}
            \frac{1}{C}\norm{f}_{H^{-1}(\Omega;\sR^{d'})}\le\norm{u}_{H^1(\Omega;\sR^{d'})}\leq  C\norm{f}_{H^{-1}(\Omega;\sR^{d'})}.
        \end{equation*}
        
        \item Suppose that Assumption~\ref{assum::BVCoeff} and \ref{assum::AprioriEstLinSys} hold with $d'\geq 1$. For any $f\in X$, system~\eqref{eq::ModEq} has a unique solution $\tilde{u} \in H^1_0(\Omega;\sR^{d'})\cap H^2(\Omega;\sR^{d'})$.  Moreover, there exists a constant $C>0$ such that 
        \begin{equation*}
            \frac{1}{C}\norm{f}_{L^{2}(\Omega;\sR^{d'})}\le\norm{u}_{H^2(\Omega;\sR^{d'})}\leq  C\norm{f}_{L^{2}(\Omega;\sR^{d'})}.
        \end{equation*}
    \end{enumerate}
    The constants $C$'s is independent of and $f$.
\end{theorem}
    
\begin{proof}
~
\begin{enumerate}[(i).]
    \item See Theorem~\ref{thm::ExistThmLinProbDivForm} which is rephrased from standard textbook such as~\cite{gilbarg2001elliptic} for $d'=1$, and Theorem~\ref{thm::ExistThmLinSysProbDivForm} which is rephrased from standard textbook such as~\cite{ambrosio2018lectures} for $d'\ge 1$.
    \item The existence and the second inequality is also standard and can be found in Theorems~\ref{thm::ExistThmLinProbDivForm} and~\ref{thm::ExistThmLinSysProbDivForm}.   To show the first inequality, we notice for all $\varphi\in H^1_0(\Omega;\sR^{d'})$ with $d'\ge 1$ and $\norm{\varphi}_{H^1(\Omega;\sR^{d'})}=1$ 
    \begin{equation*}
        \langle f, \varphi\rangle_{H^{-1}(\Omega;\sR^{d'}),H^1_0(\Omega;\sR^{d'})}
        \leq C\norm{D\varphi}_{L^2(\Omega;\sR^{d'})}\norm{Du}_{L^2(\Omega;\sR^{d'})}\leq C\norm{Du}_{L^2(\Omega;\sR^{d'})},
    \end{equation*}
    where $C$ depends on $\Omega$ and $A$. Taking supremum, we obtain $\norm{f}_{H^{-1}(\Omega;\sR^{d'})}\le C\norm{\tilde{u}}_{H^1(\Omega;\sR^{d'})}$.
    \item For $d'=1$, $X=L^2(\Omega)$, Theorem~\ref{thm::ExistThmLinProbNonDivForm} with $p=2$ guarantees the existence of the solution $\tilde{u}$ to the equation
    \begin{equation*}
        -\sum_{i,j=1}^d\left(\bar{A}_{ij}D_{ij}\tilde{u}+\chi^{-1}D_{i}^{\mathrm{a}}A_{ij}D_{j}\tilde{u}\right)=\chi^{-1}f.
    \end{equation*}
    Its solution $\tilde{u}$ satisfies
    \begin{equation}\label{eq::PrioEstiNonDivLin}
        \norm{\tilde{u}}_{H^{2}(\Omega)}\leq C\norm{\chi^{-1}f}_{L^2(\Omega)}\leq \chi_{\min}^{-1}C\norm{f}_{L^2(\Omega)},
    \end{equation}
    where $C$ depends on $\Omega$, $\chi$, and $\bar{A}$.
    Recall $A=\chi\bar{A}$. Obviously, $\tilde{u}$ is also the solution to
    \begin{equation*}
        \tilde{L}\tilde{u}=-\sum_{i,j=1}^d\left(A_{ij}D_{ij}\tilde{u}+D_{i}^{\mathrm{a}}A_{ij}D_{j}\tilde{u}\right) =f.
    \end{equation*}

    Note that $\tilde{L}\tilde{u}$ is the weighted sum of $D_j\tilde{u}$ and $D_{ij}\tilde{u}$ with $L^\infty$ coefficients $A_{ij}$ and $D_{i}^{\mathrm{a}}A_{ij}$, respectively. We have the following inequality
    \begin{equation*} 
        \norm{f}_{L^2(\Omega)}= \norm{\tilde{L}\tilde{u}}_{L^2(\Omega)} \leq C\norm{\tilde{u}}_{H^{2}(\Omega)},
    \end{equation*}
    where $C$ is independent of $f$.

    For $d'> 1$, the set $X$ in \eqref{eq::Range4data} is well-defined, and hence the existence to the equation \eqref{eq::Range4data} holds. 
    And since for all $\tilde{u}\in X$, $\tilde{L}w$ is a linear combination of $\tilde{u}$ up to second order and $A^{\alpha\beta}\in L^{\infty}(\Omega;S^{d\times d})$ and $D^{\mathrm{a}}A^{\alpha\beta}\in L^{\infty}(\Omega;S^{d\times d})$ for all $\alpha, \beta$. Thus for all $\tilde{u}\in X$, there is $C>0$ such
    \begin{equation*}
        \norm{f}_{L^2(\Omega;\sR^{d'})}\le C\norm{\tilde{u}}_{H^2(\Omega;\sR^{d'})}.
    \end{equation*}
    and by Assumption~\ref{assum::AprioriEstLinSys}, second inequality in (\romannumeral3) and the uniqueness hold.
\end{enumerate}
\end{proof}

\subsection{Deviation occurs}\label{sec::WorstCaseDevRMSol}

In this subsection, we prove that, for some specific $f$, the distance between the solutions to the original and modified equations is non-zero. According to our modelling in Section \ref{sec::DeriveModEq} and numerical simulation, the solutions $u_{\theta}$ and $\tilde{u}$ are very close to each other. Therefore, in the worst case, the deviation of RM solution is not negligible for the elliptic equations. And this explains why PINN sometimes may fail as shown in Section \ref{sec::FailExPINN}. 

The goal is clear, and essential we need to find a particular data $f$ such that the corresponding $u$ is not equal to $\tilde{u}$. In fact, as we mentioned previously, the set of $f$ which induces a gap between $u$ and $\tilde{u}$ is identified by the kernel $\Ker(T-I)$. Hence the transformation $T$ is informative and required to be investigated. We will, in particular, show the compactness and the fact that it is generically not equal to identity. This is given by the following proposition. 
\begin{prop}[properties of RM-transformation]\label{prop::RMTransModEq}
     Suppose that Assumption~\ref{assum::BVCoeff} holds with $d'\ge 1$. Let $u$ and $\tilde{u}$ be solutions to the original equation/system~\eqref{eq::OriginEq} and modified equation/system~\eqref{eq::ModEq} with data $f\in X$, respectively. Let $T$ be the operator defined in~\eqref{eq::RMTransformation}. Then
    \begin{enumerate}[(i).]
        \item $L\colon\,\left(H^1_0, \norm{\cdot}_{H^1}\right)\to \left( H^{-1},\norm{\cdot}_{H^{-1}}\right)$
        is a bounded linear operator,
        \item $T\colon\,\left(X, \norm{\cdot}_{L^{2}}\right)\to \left( H^{-1},\norm{\cdot}_{H^{-1}}\right)$
        is a bounded linear operator,
        \item $(Tf-f)^{\alpha}=(L\tilde{u}-\tilde{L}\tilde{u})^{\alpha}=-\sum_{\beta=1}^{d'}(\chi^{\alpha\beta+}-\chi^{\alpha\beta-})\nu_{\chi^{\alpha\beta}}\bar{A}^{\alpha\beta}D\tilde{u}^{\beta}\diff \fH^{d-1}$ for all $\alpha\in \{1,\ldots,d'\}$ is a Radon measure and $Tf-f$ is in $H^{-1}(\Omega;\sR^{d'})$
        \item $T\neq I$, that is, $T$ is not the identity on $L^2(\Omega)$.
    \end{enumerate}
\end{prop}

\begin{proof}
    Clearly, $L$ and $T$ are both linear.
    \begin{enumerate}[(i).]
        \item This is proved by Theorem~\ref{thm::ExistAprioEstLinEllipEq&Sys} (\romannumeral2).
        \item  Estimate the $H^{-1}$ norm of $T f$ by Theorem~\ref{prop::RMTransModEq} (\romannumeral1)
        \begin{equation}\label{eq::PrioEstiDivLinDualSoblev}
         \norm{T f}_{H^{-1}(\Omega;\sR^{d'})}= \norm{L \tilde{u}}_{H^{-1}(\Omega;\sR^{d'})}\leq C\norm{\tilde{u}}_{H^1(\Omega;\sR^{d'})},  
        \end{equation}
        where $C$ depends on $\Omega$ and $A$. The natural embedding $H^1(\Omega)\hookrightarrow H^2(\Omega)$ with estimates~\eqref{eq::PrioEstiNonDivLin} and~\eqref{eq::PrioEstiDivLinDualSoblev} leads to
        the boundedness of $T$, that is,
        \begin{equation*}
            \norm{T f}_{H^{-1}(\Omega;\sR^{d'})}\leq C \norm{\tilde{u}}_{H^1(\Omega;\sR^{d'})}\leq C\norm{\tilde{u}}_{H^2(\Omega;\sR^{d'})} \leq C\norm{f}_{L^2(\Omega;\sR^{d'})},
        \end{equation*}
        where $C$'s are different from one inequality to another and depending on $\Omega$, $\chi$, $\bar{A}$.
        \item By direct calculation with the structure theorem (Theorem \ref{thm::StructureThmBV}), $Tf-f$ satisfies for all $\alpha\in\{1,\ldots,d'\}$
    \begin{equation*}
        (Tf-f)^{\alpha}=(L\tilde{u}-\tilde{L}\tilde{u})^{\alpha}=-\sum_{\beta=1}^{d'}(\chi^{\alpha\beta+}-\chi^{\alpha\beta-})\nu_{\chi^{\alpha\beta}}\bar{A}^{\alpha\beta}D\tilde{u}^{\beta}\diff \fH^{d-1}
    \end{equation*} 
    in the sense of Radon measure. Moreover, $Tf-f\in H^{-1}(\Omega;\sR^{d'})$ as a consequence of the Theorem~\ref{prop::RMTransModEq} (\romannumeral1) and (\romannumeral2). In fact, we have with some $C$ depending on $\Omega$, $\chi$, $\bar{A}$:
    \begin{equation*}
        \norm{Tf-f}_{H^{-1}(\Omega;\sR^{d'})}\le \norm{Tf}_{H^{-1}(\Omega;\sR^{d'})}+\norm{f}_{H^{-1}(\Omega;\sR^{d'})}\le C\norm{f}_{L^2(\Omega;\sR^{d'})}.
    \end{equation*}
    \item  By Theorem~\ref{thm::NecSufCondRemovSingLin} and~\ref{thm::NecSufCondRemovSingLinSys}, there are $\alpha_0\in \{1,\ldots,d'\}$ and $v_{\delta}\in C_c^{\infty}(\Omega;\sR^{d'})$ such that 
    \begin{equation*}
        \sum_{\beta=1}^{d'}\mu(\Omega;\chi^{\alpha_0\beta},\bar{A}^{\alpha_0\beta},v_{\delta}^{\beta})>0.
    \end{equation*}
    By setting $f=\tilde{L}v_{\delta}\in L^2(\Omega;\sR^{d'})$, we have 
    \begin{equation*}
        (Tf-f)^{\alpha_0}=-\sum_{\beta=1}^{d'}(\chi^{\alpha_0\beta+}-\chi^{\alpha_0\beta-})\nu_{\chi^{\alpha_0\beta}}\bar{A}^{\alpha_0\beta}Dv_{\delta}^{\beta}\diff \fH^{d-1}
    \end{equation*} in the sense of Radon measure. Let $\Omega_{n}=\{x\in \Omega\colon\,\dist(x,\partial\Omega)>\frac{1}{n}\}$ and use Lemma~\ref{lem::SmoothCutoffFunContrDer} to choose a cutoff function $\zeta_{n}$ such that $\Omega_{n}\prec \zeta_{n}\prec\Omega$. Then for sufficiently large $n$
    \begin{align*}
        \langle (f-Tf)^{\alpha_0}, \zeta_{n}\rangle_{H^{-1}(\Omega),H^{1}_0(\Omega)}
        &= \sum_{\beta=1}^{d'}\mu(\Omega;\chi^{\alpha_0\beta},\zeta_n \bar{A}^{\alpha_0\beta},v_{\delta}^{\beta})\\
        &=\sum_{\beta=1}^{d'}\mu(\Omega;\chi^{\alpha_0\beta}, \bar{A}^{\alpha_0\beta},v_{\delta}^{\beta})-\sum_{\beta=1}^{d'}\mu(\Omega;\chi^{\alpha_0\beta},(1-\zeta_n) \bar{A}^{\alpha_0\beta},v_{\delta}^{\beta})\\
        &\ge \sum_{\beta=1}^{d'}\mu(\Omega;\chi,\bar{A},v_{\delta})-2d'\chi_{\max}\bar{\Lambda}\norm{Dv_{\delta}}_{L^{\infty}(\Omega)}\fH^{d-1}((\Omega\backslash\Omega_{n})\cap J_{\chi})\\
        &>0,
    \end{align*}
    which implies $Tf\neq f$.
    \end{enumerate}
\end{proof}
The last statement immediately implies that for specific $f$, the deviation of solution $\tilde{u}$ from $u$ occurs as follows.

\begin{theorem}[deviation occurs for linear elliptic equations]\label{thm::WorstCaseDeviLinEllipEq}
    Suppose that Assumption~\ref{assum::BVCoeff} holds with $d'\ge1$. Then there exists $f\in L^2(\Omega;\sR^{d'})$ such that $\norm{u-\tilde{u}}_{H^1(\Omega;\sR^{d'})}>0$, where $u$ and $\tilde{u}$ are solutions to the original equation/system~\eqref{eq::OriginEq} and modified equation/system~\eqref{eq::ModEq}, respectively.
\end{theorem}

\begin{proof}
    By Proposition~\ref{prop::RMTransModEq} (\romannumeral4), there is $f\in L^2(\Omega;\sR^{d'})$ such that $\norm{Tf-f}_{H^{-1}(\Omega;\sR^{d'})}>0$, combining Theorem~\ref{thm::ExistAprioEstLinEllipEq&Sys} (\romannumeral2), there is $C>0$ such that 
    \begin{equation*}
        \norm{u-\tilde{u}}_{H^1(\Omega;\sR^{d'})}\ge \frac{1}{C}\norm{Tf-f}_{H^{-1}(\Omega;\sR^{d'})}>0,
    \end{equation*}
    and this completes the proof.
\end{proof}
Obviously, when $d'=1$, we only need Assumption~\ref{assum::BVCoeff}.

\subsection{Deviation occurs generically}\label{sec::RMTransModEqLin}

Based on Proposition~\ref{prop::RMTransModEq}, we further study the eigenvalue and eigenspace of $T$ in this subsection. 
We recall that an eigenvalue could be a complex number in general. Hence, for the full consideration of the eigenvalue problem, we would like to enlarge the space of $f$ from $X$ to the complex-valued one, namely $\bar{X}:=\{\tilde{L}w\colon\,w\in H^1_0(\Omega;\sC^{d'})\cap H^{2}(\Omega;\sC^{d'}) \}$. Meanwhile, we extend the operator $T$ to $\bar{T}: \bar{X}\to H^{-1}(\Omega;\sC^{d'})$. But for notational simplicity, we still write $T$ for $\bar{T}$. This will not cause any ambiguity, since we only need this extension for the discussion on eigenvalues. Similar remark applies to the system case in Theorem~\ref{thm::EigenValueEigenSpace}.

No matter whether we consider the eigenvalue problem over the complex field or the real field, the only eigenvalue of $T$ is $1$, as shown in Theorem~\ref{thm::EigenValueEigenSpace}. Hence there is no harm to regard all functions to be real-valued, and it is consistent to our situation, that is, to study the real-valued PDE problems. 
Next, to characterize the implicit bias via RM-transformation, we only need to consider the real-valued kernel $\Ker(T-I)$ which consists of all RM-invariant data $f$. Thus the $X\backslash\Ker(T-I)$ consists of data $f$ satisfying $u\neq \tilde{u}$. In the following theorems (Theorem~\ref{thm::EigenValueEigenSpace}, we show that the latter is more generic, and hence RM method fail for almost all data $f$ in equations considered in this paper.

\begin{theorem}[eigenvalue and eigenspace of $T$ for linear elliptic equation/system]\label{thm::EigenValueEigenSpace}
    Suppose that Assumption~\ref{assum::BVCoeff} and~\ref{assum::AprioriEstLinSys} hold with $d'\geq 1$ and $\bigcup_{\alpha,\beta=1}^{d'}J_{\chi^{\alpha\beta}}$ is not dense in $\Omega$. Then we have
    \begin{enumerate}[(i).]
        \item $\sigma(T)=\{1\}$;
        \item $\Ker(T-I)=\{f\in X\colon\,\exists w \text{ such that } \tilde{L} w=f,\  \forall \alpha\ \sum_{\beta=1}^{d'}\mu(\cdot;\chi^{\alpha\beta},\bar{A}^{\alpha\beta},w)=0\}$;
        \item $L^2(\Omega;\sR^{d'})\backslash \Ker(T-I)$ is dense in $L^2(\Omega;\sR^{d'})$;
        \item $X$ is a closed subspace of $L^2(\Omega;\sR^{d'})$ and $X\backslash \operatorname{Ker}(T-I)$ is relatively open in $L^2(\Omega;\sR^{d'})$.
    \end{enumerate}
\end{theorem}
As above, we also note that $\sum_{\beta=1}^{d'}\mu(\cdot;\chi^{\alpha\beta},\bar{A}^{\alpha\beta},w)=0$, $\forall \alpha\in\{1,\ldots,d'\}$ for all $\fH^{d-1}$ measurable set $B$ is equivalent to $\sum_{\beta=1}^{d'}\bar{A}^{\alpha\beta}Dw^{\beta}\cdot D^{\mathrm{j}}\chi^{\alpha\beta}=0$, $\forall \alpha\in\{1,\ldots,d'\}$ as Radon measures.
\begin{proof}
~
\begin{enumerate}[(i).]
    \item For any $f\in \bar{X}$, let $\tilde{u}$ be the solution to~\eqref{eq::ModEq} with data $f$. 
    
    For $z\in \sC\backslash\{1\}$, we have for all $\alpha\in\{1,\ldots,d'\}$
    \begin{equation*}
        Tf^{\alpha}-zf^{\alpha}=(1-z)f^{\alpha}+\sum_{\beta=1}^{d'}\bar{A}^{\alpha\beta}D\tilde{u}^{\alpha}\cdot D^{\mathrm{j}}\chi^{\alpha\beta}\quad \text{in the sense of Radon measure}.
    \end{equation*}
    Notice that the measures $(1-z)f^{\alpha}$ and $\sum_{\beta=1}^{d'}\bar{A}^{\alpha\beta}D\tilde{u}^{\alpha}\cdot D^{\mathrm{j}}\chi^{\alpha\beta}$ are mutually singular to each other. Hence
    \begin{equation*}
        Tf-zf=0 \text{ is equivalent to } (1-z)f^{\alpha}=0 \text{ and } \sum_{\beta=1}^{d'}\bar{A}^{\alpha\beta}D\tilde{u}^{\alpha}\cdot D^{\mathrm{j}}\chi^{\alpha\beta}=0,\quad \forall \alpha\in\{1,\ldots,d'\}.
    \end{equation*}
    Therefore for all $\alpha$, $f^{\alpha}=0$. By Assumption~\ref{assum::AprioriEstLinSys}, we have $\tilde{u} =0 $.

    For $z=1$, we obtain for all $\alpha\in\{1,\ldots,d'\}$
    \begin{equation*}
        Tf^{\alpha}-f^{\alpha}=\sum_{\beta=1}^{d'}\bar{A}^{\alpha\beta}D\tilde{u}^{\beta}\cdot D^{\mathrm{j}}\chi^{\alpha\beta}.
    \end{equation*}
    Since $\bigcup_{\alpha,\beta=1}^{d'}J_{\chi^{\alpha\beta}}$ is not dense in $\Omega$, there is a open ball $B_r(x)$ such that $B_r(x)\cap (\bigcup_{\alpha,\beta=1}^{d'}J_{\chi^{\alpha\beta}})=\emptyset$. By choosing $B_{\frac{r}{2}}(x)\prec \tilde{u} \prec B_{r}(x)$ and $f=\tilde{L}\tilde{u}$, we obtain
    \begin{equation*}
        Tf-f=0.
    \end{equation*}
    Hence $\sigma(T)=\{1\}$.
    \item This is obvious.
    \item  By Theorem~\ref{thm::NecSufCondRemovSing4Sys}, there are $\alpha_0$ and $v_{\delta}\in C_c^{\infty}(\Omega;\sR^{d'})$ such that
    \begin{equation*}
        \sum_{\beta=1}^{d'}\mu(\Omega;\chi^{\alpha_0\beta},\bar{A}^{\alpha_0\beta},v_{\delta}^{\beta})\neq 0.
    \end{equation*}
    By setting $g=\tilde{L}v_{\delta}$, it is easy to see that $Tg\neq g$. For each $f\in \Ker(T-I)$, let $f_{\eps}=f+\eps\frac{g}{\norm{g}_{L^2(\Omega;\sR^{d'})}}$. We have $f_{\eps}\notin \Ker(T-I)$ for all $\eps>0$ and that $\lim_{\eps\to 0}f_{\eps}=f$. Hence $X\backslash \Ker(T-I)$ is dense in $X$. Therefore $L^2(\Omega;\sR^{d'})\backslash \Ker(T-I)$ is dense in $L^2(\Omega;\sR^{d'})$.
    \item  We show that $X$ is closed under $L^2(\Omega;\sR^{d'})$ norm. By Theorem~\ref{thm::ExistAprioEstLinEllipEq&Sys}, there is $C>0$
    \begin{equation*}
        \norm{\tilde{L}w}_{L^2(\Omega;\sR^{d'})}\le C\norm{w}_{H^2(\Omega;\sR^{d'})}.
    \end{equation*}
    By choosing a Cauchy sequence $\{f_k\}_{k=1}^{\infty}\subseteq X$ with $f_k=\tilde{L}\tilde{u}_k$, there is $f\in L^2(\Omega;\sR^{d'})$ and $\tilde{u}\in H^1_0(\Omega;\sR^{d'})\cap H^2(\Omega;\sR^{d'})$ such that
    \begin{align*}
        \lim_{k\to\infty}\norm{f-f_k}_{L^2(\Omega;\sR^{d'})}
        &=0,\\
        \lim_{k\to\infty}\norm{\tilde{u}-\tilde{u}_k}_{H^2(\Omega;\sR^{d'})}
        &=0.
    \end{align*}
    Now we show that $\tilde{L}\tilde{u}=f$, $\fL^{d}$-a.e. and that $\norm{\tilde{u}}_{H^2(\Omega;\sR^{d'})}\le C \norm{f}_{L^2(\Omega;\sR^{d'})}$.

    Note that 
    \begin{align*}
        \norm{\tilde{L}\tilde{u}-f}_{L^2(\Omega;\sR^{d'})}
        &\le \norm{\tilde{L}\tilde{u}-\tilde{L}\tilde{u}_k}_{L^2(\Omega;\sR^{d'})}+\norm{\tilde{L}\tilde{u}_k-f}_{L^2(\Omega;\sR^{d'})}
        \\
        &\le C\norm{\tilde{u}-\tilde{u}_k}_{H^2(\Omega;\sR^{d'})}+\norm{f_k-f}_{L^2(\Omega;\sR^{d'})}.
    \end{align*}
    Taking $k\to\infty$, we obtain $\norm{\tilde{L}\tilde{u}-f}_{L^2(\Omega;\sR^{d'})}=0$. Moreover,
    \begin{equation}\label{eq::conrtolofsystem}
        \begin{aligned}
            \norm{\tilde{u}}_{H^2(\Omega;\sR^{d'})}&\le \norm{\tilde{u}_k-\tilde{u}}_{H^2(\Omega;\sR^{d'})}+\norm{\tilde{u}_k}_{H^2(\Omega;\sR^{d'})}
            \\&\le C\norm{f_k}_{L^2(\Omega;\sR^{d'})}+ \norm{\tilde{u}_k-\tilde{u}}_{H^2(\Omega;\sR^{d'})}
            \\&\le C\norm{f}_{L^2(\Omega;\sR^{d'})}+C\norm{f-f_k}_{L^2(\Omega;\sR^{d'})}+ \norm{\tilde{u}_k-\tilde{u}}_{H^2(\Omega;\sR^{d'})}.
        \end{aligned}
    \end{equation}
    Taking $k\to \infty$ again, we have $\norm{\tilde{u}}_{H^2(\Omega;\sR^{d'})}\le C\norm{f}_{L^2(\Omega;\sR^{d'})}$. Hence $C$ is a closed set.

    Now by letting $\{f_k\}_{k=1}^{\infty}\subseteq X\cap \Ker(T-I)$ be a Cauchy sequence under $L^2$ norm,  there is an $f$ such that
    \begin{equation*}
        \lim_{k\to\infty}f_k=f\in X.
    \end{equation*}
    In the rest, we show $f\in \Ker(T-I)$.

    Also let $u$ and $u_k$ be the solution to~\eqref{eq::OriginEq} with data $f$ and $f_k$, respectively.
    Since $f_k\in \Ker(T-I)$ for all $k$, then $u_k=\tilde{u}_k$, $\fL^d$-a.e..

    By~\eqref{eq::conrtolofsystem}, we have:
    \begin{equation*}
        \begin{aligned}
            \norm{u-u_k}_{L^2(\Omega;\sR^{d'})}&\le C\norm{f-f_k}_{L^2(\Omega;\sR^{d'})},
            \\ \norm{\tilde{u}-\tilde{u}_k}_{L^2(\Omega;\sR^{d'})}&\le C\norm{f-f_k}_{L^2(\Omega;\sR^{d'})},
        \end{aligned}
    \end{equation*}
    which indicates that $u=\tilde{u}$, $\fL^d$-a.e.. Thus 
    \begin{equation*}
        Tf=L\tilde{u}=Lu=f,
    \end{equation*}
    which means $Tf=f$ and hence $f\in \Ker(T-I)$.
\end{enumerate}
\end{proof}
When $d'=1$, we have $X=L^2(\Omega)$ and when $d'\ge 1$, $X$ is the largest space that makes equation~\eqref{eq::ModEq} solvable. Theorem~\ref{thm::EigenValueEigenSpace} (\romannumeral3) and (\romannumeral4) together claim that most of $f$ in $X$ make deviation occur.

To illustrate Theorem~\ref{thm::EigenValueEigenSpace}, we show in the following example that for $d'=1$ and $f\in \Ker(T-I)$, we have $u=\tilde{u}$ and the numerical solution $u_{\theta}$ matches them quite accurately. Of course, such $f$ should be very rare according to Theorem~\ref{thm::EigenValueEigenSpace}.

\begin{example}[back to 1-d]
    We again consider~\eqref{eq::1dEq} with coefficients and right hand side as follows:
    \begin{equation}\label{eq::OneDimNonDeviate}
        A(x)=\left\{
        \begin{aligned}
            & \tfrac{1}{2}, & & x\in (-1,0), \\
            &1, & & x\in [0,1),
        \end{aligned}
        \right.\quad 
        f(x)=\left\{
        \begin{aligned}
            & -1, & & x\in (-1,0), \\
            & -2, & & x\in [0,1).
        \end{aligned}
        \right.
    \end{equation}
    For this problem, we have $u=\tilde{u}=x^2-1$ and hence $D_{x}\tilde{u} = 2x$ which satisfies $D_{x}\tilde{u}(0) =0$. Thus by Theorem~\ref{thm::EigenValueEigenSpace}, $Tf=f$, which means the equalities hold: $u=\tilde{u}$. Here the numerical simulation is under the same setting as that of Section~\ref{sec::FailExPINN}. 
    
    \begin{figure}[H]
        \centering	
        \subfigure[Comparison of solutions]{\includegraphics[scale=0.35]{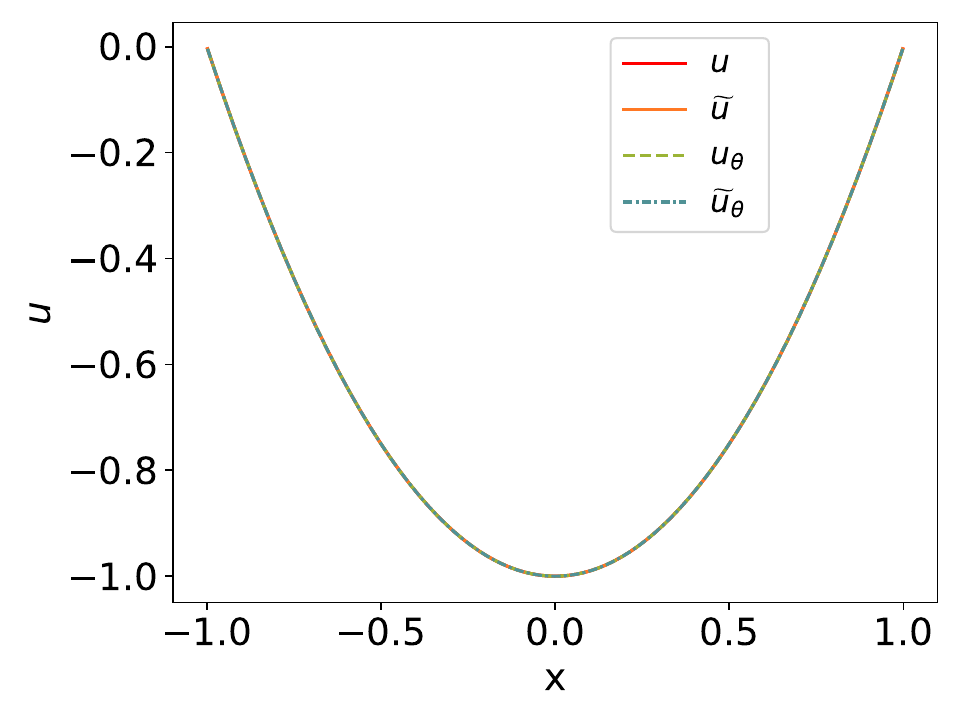}}
    	\subfigure[Comparison of gradients]{
    	\includegraphics[scale=0.35]{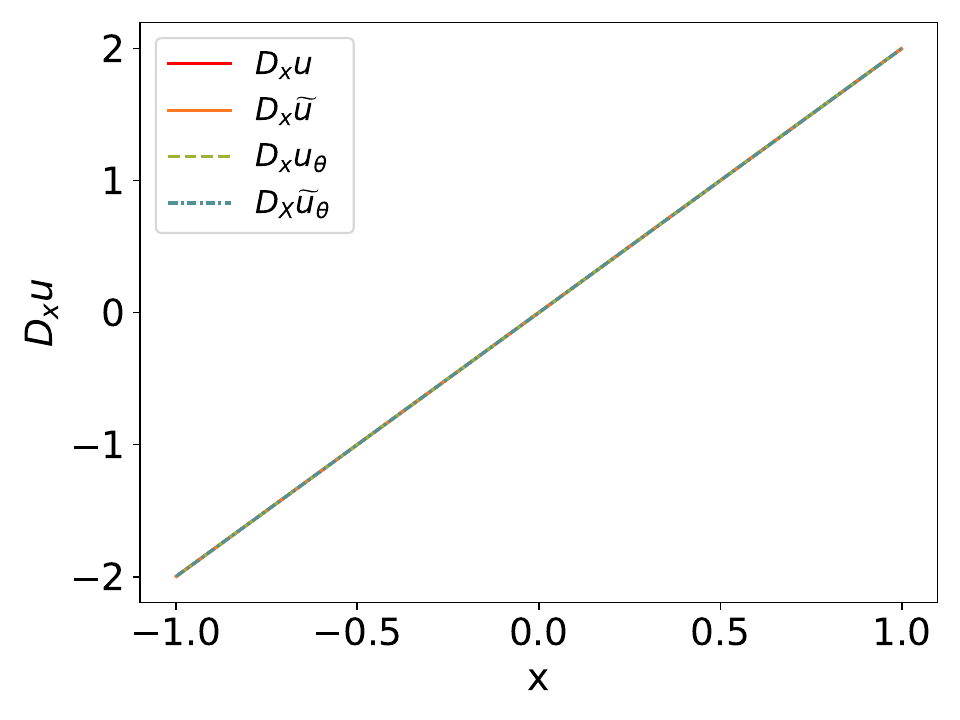}}
    	\caption{PINN can fit the exact solution to~\eqref{eq::1dEq} with coefficients as~\eqref{eq::OneDimNonDeviate}.}
    	\label{fig::OneDimNonDeviate}
    \end{figure}
    The numerical simulation indicates $u=\tilde{u}\approx u_{\theta}$, which verifies Theorem~\ref{thm::EigenValueEigenSpace}, though only for very rare $f$.
\end{example}

The behavior of RM method solving linear PDEs satisfying Assumption~\ref{assum::BVCoeff} can be summarized as follows: if $\mu(\cdot;\chi,\bar{A},\tilde{u})=0$, then $u_{\theta}\approx \tilde{u}_{\theta}\approx \tilde{u} = u$,  if $\mu(\cdot;\chi,\bar{A},\tilde{u})\neq 0$, then $u_{\theta}\approx \tilde{u}_{\theta}\approx \tilde{u} \neq u$.

\subsection{Deviation occurs severely}\label{sec::RelaDevia}

As we studied in Section~\ref{sec::WorstCaseDevRMSol}, for $f\notin \Ker(T-I)$, the deviation occurs, that is $\norm{u-\tilde{u}}_{H^1(\Omega)}>0$. From this, it is still unknown whether the relative deviation $\sup_{f\in L^2(\Omega)}\frac{\norm{\tilde{u}-u}_{H^1(\Omega)}}{\norm{\tilde{u}}_{H^1(\Omega)}}$ is bounded or not. 
In this section, we step further and prove that the supremum could be infinity, i.e., $\sup_{f\in L^2(\Omega)}\frac{\norm{\tilde{u}-u}_{H^1(\Omega)}}{\norm{\tilde{u}}_{H^1(\Omega)}}=+\infty$ (See Theorem~\ref{thm::UnbddRLinEllipEq}). Furthermore, we will obtain a stronger result (See Proposition~\ref{prop::relativedeviation}) showing that even for data $f$ sufficiently close to the RM-invariant subspace $\Ker(T-I)$, the relative deviation can still achieve have a finite value. 

\begin{theorem}[Unboundedness of relative deviation for linear elliptic equation]\label{thm::UnbddRLinEllipEq}
     Suppose that Assumption~\ref{assum::BVCoeff} holds with $d'\ge 1$. For the countably many $(d-1)$-dimensional $C^1$ manifolds $\{S_i^{\alpha\beta}\}_{i=1}^{\infty}$ such that $\fH^{d-1}\left(\bigcup_{\alpha,\beta=1}^{d'}\left(J_{\chi^{\alpha\beta}}-\bigcup_{i=1}^{\infty}S^{\alpha\beta}_i\right)\right)=0$ (see in Theorem~\ref{thm::PropOfApproxJumpSetForBV}), suppose that there is a constant $r_0>0$ such that for $(i,\alpha_1,\beta_1) \neq (j,\alpha_2,\beta_2)$, $\dist(S^{\alpha_1\beta_1}_i,S^{\alpha_2\beta_2}_j)=\inf_{x\in S^{\alpha_1\beta_1}_i,y\in S^{\alpha_2\beta_2}_j}\abs{x-y}\ge r_0$. Then 
     \begin{equation}
         \sup_{f\in L^2(\Omega;\sR^{d'})}\frac{\norm{\tilde{u}-u}_{H^1(\Omega;\sR^{d'})}}{\norm{\tilde{u}}_{H^1(\Omega;\sR^{d'})}}=+\infty.
     \end{equation}
     where $u$ and $\tilde{u}$ are solutions to the original system~\eqref{eq::OriginEq} and modified system~\eqref{eq::ModEq} with data $f\in L^2(\Omega;\sR^{d'})$, respectively.
\end{theorem}

\begin{proof} Let's start with the case $d'=1$: By      Theorem~\ref{thm::ExistAprioEstLinEllipEq&Sys},     we have
    \begin{equation}\label{eq::LowBounByDualSoblev}
        \frac{\norm{\tilde{u}-u}_{H^1(\Omega)}}{\norm{\tilde{u}}_{H^1(\Omega)}}\ge C \frac{\norm{Tf-f}_{H^{-1}(\Omega)}}{\norm{\tilde{u}}_{H^1(\Omega)}}=C\frac{\norm{(\chi^+-\chi^-)\nu_{\chi}\bar{A}D\tilde{u}\diff \fH^{d-1}}_{H^{-1}(\Omega)}}{\norm{\tilde{u}}_{H^1(\Omega)}},
    \end{equation}
    where for the last equality, we refer readers to Proposition~\ref{prop::RMTransModEq}.
    Combining Theorems~\ref{thm::NecSufCondRemovSing} and~\ref{thm::NecSufCondRemovSingLin}, there is a $v_{\delta}=\rho_{\delta}*v\in C_c^{\infty}(\Omega)$ with
    $v=\zeta V\in C_c^1(\Omega)$, 
    where $\rho_{\delta}$, $\zeta$ and $V$ are defined according to the proofs of Theorems~\ref{thm::NecSufCondRemovSing} and~\ref{thm::NecSufCondRemovSingLin}. Since
    $\norm{v}_{H^1(\Omega)}>0$ and $\mu(\Omega;\chi,\bar{A},v)>0$,
    we have
    \begin{equation}\label{eq::NonZeroLowBounOfOrigFunc}
        \norm{(\chi^+-\chi^-)\nu_{\chi}\bar{A}Dv\diff \fH^{d-1}}_{H^{-1}(\Omega)}>0.
    \end{equation}

    \begin{case}
        Proof of the theorem when $\norm{(\chi^+-\chi^-)\nu_{\chi}\bar{A}Dv\diff \fH^{d-1}}_{H^{-1}(\Omega)}=+\infty$.
    \end{case}
    For each $k\in \sN^+$, there is a $\varphi_k\in C_c^{\infty}(\Omega)$ such that
    \begin{equation*}
        \mu(\Omega; \chi, \varphi_k\bar{A}, v)\ge k\norm{\varphi_k}_{H^1(\Omega)}.
    \end{equation*}
    We can choose sufficiently small $\delta$ such that 
    \begin{equation*}
        \mu(\Omega; \chi, \varphi_k\bar{A},v_{\delta}-v)\le \norm{\varphi_k}_{H^1(\Omega)},
    \end{equation*}
    because $v_{\delta}=\rho_{\delta}*v\to v$ uniformly and $\norm{\varphi_k}_{L^{\infty}(\Omega)}<+\infty$. Hence
    \begin{equation*}
        \mu(\Omega; \chi, \varphi_k\bar{A}, v_{\delta})\ge (k-1)\norm{\varphi_k}_{H^1(\Omega)}.
    \end{equation*}
    Therefore
    \begin{equation}
        \norm{(\chi^+-\chi^-)\nu_{\chi}\bar{A}Dv_{\delta}\diff \fH^{d-1}}_{H^{-1}(\Omega)}\geq k-1.
    \end{equation}
    Moreover $\lim_{\delta\to 0}\norm{v_{\delta}}_{H^1(\Omega)}=\norm{v}_{H^1(\Omega)}$, which indicates that
    \begin{equation*}
        \lim_{\delta\to 0}\frac{\norm{\tilde{u}-u}_{H^1(\Omega)}}{\norm{\tilde{u}}_{H^1(\Omega)}}=+\infty,
    \end{equation*}
    where $u, \tilde{u}$ are solutions to equation~\eqref{eq::OriginEq}
    and~\eqref{eq::ModEq} respectively with the data $f=\tilde{L}v_{\delta}$.

    \begin{case}
        Proof of the theorem when $\norm{(\chi^+-\chi^-)\nu_{\chi}\bar{A}Dv\diff \fH^{d-1}}_{H^{-1}(\Omega)}<+\infty$.
    \end{case}
     Recall that
    \begin{equation}\label{eq::LevSetFunc}
        V(x)=(x_{\tau})_{d}-h(x'_{\tau}),
    \end{equation}
    where $x'_{\tau}=((x_{\tau})_1,\ldots,(x_{\tau})_{d-1})$, $h$ is a $C^1$ function, and for all $x\in S\cap U_z$, $V(x)=0$, and the unit normal vector at $x\in S\cap U_z$ is $\nu_{\chi}=\frac{DV}{\abs{DV}}$ by proof of Theorem~\ref{thm::NecSufCondRemovSing}.

    \begin{claim_03}
        For each $r\in (0,1]$, there is a function $\tv$ such that
        \begin{equation*}
            D\tv=\frac{1}{r}Dv, \text{\ \  for all $x\in S\cap U_z.$}  
        \end{equation*}
    \end{claim_03}
    By our construction in the proof of Theorems~\ref{thm::NecSufCondRemovSing} and~\ref{thm::NecSufCondRemovSingLin}, $v=\zeta V$ is supported in $O^{\eps}$ and $v_{\delta}$ is supported in $O^{2\eps}$ with
    \begin{equation*}
        \eps=\delta < \frac{1}{4}\min \left(\dist\left(\Psi(\bar{Q}_k), \partial U_z\right), r_0\right).
    \end{equation*}
    Since $V(x)=0$, $x\in S\cap U_z$ and $\norm{DV}_{L^{\infty}(O^{\eps})}\leq C$ according to~\eqref{eq::V(x)UpperBound}, then $\abs{V(x)}\leq 2C\eps$, $x\in O^{2\eps}$. 
    By choosing $\eps\leq \eps_1$ with a sufficiently small $\eps_1>0$, we can define for $0<r\leq 1$
    \begin{align*}
        \tC_r
        &=\{(y',y_d)\colon\,\exists\omega\in\sR\text{ such that } \Psi(y',\omega)\in O^{2\eps},\  \abs{y_d}\le 2Cr\eps\}\\
        &=\tB \times [-2Cr\eps,2Cr\eps],
    \end{align*}
    where $\tB=\{y'\colon\,\exists\omega\in\sR\text{ such that } \Psi(y',\omega)\in O^{2\eps}\}$ and
    $O^{2\eps}\subseteq \Psi(\tC_1) \subseteq U_z$. Thus we have $\norm{v_{\delta}}_{H^{1}(\Omega)}=\norm{v_{\delta}}_{H^1(\Psi(\tC_1))}$.

    Consider the integration over $\tC_1$
    \begin{align*}
        \int_{\Psi(\tC_1)}v_{\delta}(x)\diff{x}
        &=\int_{\Psi(\tC_1)}[\rho_{\delta}*(\zeta V)](x)\diff{x}
        \\ 
        &=\int_{\tB}\int_{-2C\eps}^{2C\eps}[\rho_{\delta}*(\zeta V)]\circ\tau^{-1}(y',y_d+h(y'))\diff{y_d}\diff{y'}.
    \end{align*}
    Denote $\tQ(y)=[\rho_{\delta}*(\zeta V)]\circ\tau^{-1}\circ\psi(y)$ with $y=(y',y_d)$ and consider the rescaled function $\tF(y)$ defined on $\tC_r$ and $\tv(x)$ defined on $x\in U_z$ respectively as
    \begin{equation}
        \tF(y)=\tF(y',y_d)=\tQ\left(y',\frac{1}{r}y_d\right)=[\rho_{\delta}*v]\circ\tau^{-1}\left(y', \frac{1}{r}y_d+h(y')\right)
    \end{equation}
    and
    \begin{equation*}
        \tv(x)
        =v\circ \tau^{-1} \left(x_{\tau}',\frac{1}{r}(x_{\tau})_{d}+\frac{r-1}{r}h(x_{\tau}')\right).
    \end{equation*}
    It is easy to verify that $\tv\in H_0^1(\Omega)$. For simplicity, we denote
    \begin{align*}
        \tz(x)&=\zeta\circ \tau^{-1} \left(x_{\tau}',\frac{1}{r}(x_{\tau})_{d}+\frac{r-1}{r}h(x_{\tau}')\right),\\
        \tV(x)&=V\circ \tau^{-1} \left(x_{\tau}',\frac{1}{r}(x_{\tau})_{d}+\frac{r-1}{r}h(x_{\tau}')\right).
    \end{align*}

    By the definition of $V$ as in~\eqref{eq::LevSetFunc}, we have for $x\in U_z$
    \begin{align*}
        \tV(x) &=V\circ \tau^{-1}\left(x'_{\tau},\frac{1}{r}(x_{\tau})_{d}+\frac{r-1}{r}h(x'_{\tau})\right)\\
        &=\frac{1}{r}(x_{\tau})_{d}+\frac{r-1}{r}h(x'_{\tau})-h(x'_{\tau})
        =\frac{1}{r}((x_{\tau})_{d}-h(x'_{\tau}))=\frac{1}{r}V(x).
    \end{align*}
    Hence $\tV(x)=V(x)=0$ and $D\tV(x)=\frac{1}{r}DV(x)$ for all $x\in S\cap U_z$. For $x\in S\cap U_z$, $(x_{\tau})_{d}=h(x'_{\tau})$ implies that for each $x\in S\cap U_z$,
    \begin{equation}
        \tz(x)=\zeta\circ\tau^{-1}\left(x'_{\tau},\frac{1}{r}(x_{\tau})_{d}+\frac{r-1}{r}h(x'_{\tau})\right)=\zeta\circ\tau^{-1}(x_{\tau})=\zeta(x).
    \end{equation}
    Thus by the product rule, we have for $x\in S\cap U_z$
    \begin{equation}\label{eq::PropOfRescFuncVOnMaini}
        D\tv=\tV D\tz+\tz D\tV=\tz D\tV=\frac{1}{r}\zeta DV=\frac{1}{r}Dv(x).
    \end{equation}

    \begin{claim_03}
    For $r\in (0,1]$, $\norm{\tv}_{H^1(\Omega)}\le \frac{C}{\sqrt{r}}$ for some constant $C>0$ independent of $r.$ 
    \end{claim_03}
    
    For all $x\in O^{2\eps}$, we have
    \begin{equation}\label{eq::BounOfRescFuncEtaDp}
        \Abs{\tz D\tV}=\frac{1}{r}\Abs{\tz DV}
        \leq \frac{1}{r}\sup_{x\in U_z}\Abs{DV(x)}\leq \frac{C}{r}.
    \end{equation}

    With a little bit abuse of notation, we 
    use $\tau(i)$ to denote the $i$-th entry of $\tau((1,\ldots,d)^{\T})$. For $i\in \{1,\ldots,d-1\}$
    \begin{align*}
        D_{\tau(i)}\tz&=D_{\tau(i)}[\zeta]\left(\tau^{-1}\left(x'_{\tau},\frac{1}{r}(x_{\tau})_{d}+\frac{r-1}{r}h(x'_{\tau})\right)\right)\\
        &~~~+\frac{r-1}{r}D_{\tau(d)}[\zeta]\left(\tau^{-1}\left(x'_{\tau},\frac{1}{r}(x_{\tau})_{d}+\frac{r-1}{r}h(x'_{\tau})\right)\right)D_{i}[h](x_{\tau}'),\\
        D_{\tau(d)}\tz
        &=\frac{1}{r}D_{\tau(d)}[\zeta]\left(\tau^{-1}\left(x'_{\tau},\frac{1}{r}(x_{\tau})_{d}+\frac{r-1}{r}h(x'_{\tau})\right)\right).
    \end{align*}
    Since $h$ is $C^1$ function in $U_z$ and $\tau^{-1}\left(x'_{\tau},\frac{1}{r}(x_{\tau})_{d}+\frac{r-1}{r}h(x'_{\tau})\right)\in O^{2\eps}\subseteq U_z$, combining~\eqref{eq::D_etaUpperBound}, we have for all $x\in \tC_r$,
    \begin{equation}
        \Abs{D\tz}\le \frac{d}{r}\frac{C}{\eps}.
    \end{equation}
    Thus 
    \begin{equation}\label{eq::BounOfRescFuncPDeat}
        \Abs{\tV D\tz}=\Abs{\frac{1}{r}V D\tz}
        \le \frac{1}{r}2Cr\eps\frac{d}{r}\frac{C}{\eps}\le \frac{C}{r}.
    \end{equation}
    Combining~\eqref{eq::BounOfRescFuncEtaDp} and $\eqref{eq::BounOfRescFuncPDeat}$, we obtain $\abs{D\tv(x)}\le \frac{2C}{r}$ for all $x\in \tC_r$. Therefore
    \begin{equation*}
        \norm{\tv}_{H^1(\Omega)}^2
        =\norm{\tv}_{L^2(\Omega)}^2+\norm{D\tv}_{L^2(\Omega)}^2
        \le r\int_{H^1(\Psi(\tC_1))}v^2\diff{x}+\frac{C}{r}\le \frac{C}{r},
    \end{equation*}
    where $C$ is independent of $r$.
    \begin{claim_03}
        For each $r\in (0,1]$, there is $\delta>0$ such that the convolution $\tv_{\delta}=\rho_{\delta}*\tv\in H_0^1(\Omega)\cap H^2(\Omega)$ satisfying
        \begin{equation*}
            \norm{(\chi^+-\chi^-)\nu_{\chi}\bar{A}D\tv_{\delta}\diff \fH^{d-1}}_{H^{-1}(\Omega)}>\frac{C}{\sqrt{r}}
        \end{equation*}
        for some constant $C>0$ independent of $r$.
    \end{claim_03}
    Note that
    \begin{equation*}
        \norm{(\chi^+-\chi^-)\nu_{\chi}\bar{A}D\tv\diff \fH^{d-1}}_{H^{-1}(\Omega)}=\frac{1}{r}\norm{(\chi^+-\chi^-)\nu_{\chi}\bar{A}Dv\diff \fH^{d-1}}_{H^{-1}(\Omega)}>0.
    \end{equation*}
    Hence there is $\varphi\in C_c^{\infty}(\Omega)$ such that
    \begin{equation*}
        \mu(\Omega;\chi,\varphi\bar{A},\tv)
        \ge \frac{3}{4}\norm{(\chi^+-\chi^-)\nu_{\chi}\bar{A}D\tv\diff \fH^{d-1}}_{H^{-1}(\Omega)} \norm{\varphi}_{H^1(\Omega)}.
    \end{equation*}
    Let $\tv_{\delta}=\rho_{\delta}*\tv\in H^1_0(\Omega)\cap H^2(\Omega)$, by the property of mollifier $\lim_{\delta\to 0}\norm{\tv_{\delta}-\tv}_{H^1(\Omega)}=0$. We thus can choose $\delta\le \eps_2$ with sufficiently small $\eps_2$ such that
    \begin{equation*}
        \frac{\norm{\tv}_{H^1(\Omega)}}{\norm{\tv_{\delta}}_{H^1(\Omega)}}\geq \frac{1}{2} \text{\ and \ }
        \mu(\Omega;\chi,\varphi\bar{A},\tv_{\delta})
        \ge \frac{1}{2}\norm{(\chi^+-\chi^-)\nu_{\chi}\bar{A}D\tv\diff \fH^{d-1}}_{H^{-1}(\Omega)} \norm{\varphi}_{H^1(\Omega)},
    \end{equation*}
    which implies
    \begin{equation*}
        \norm{(\chi^+-\chi^-)\nu_{\chi}\bar{A}D\tv_{\delta}\diff \fH^{d-1}}_{H^{-1}(\Omega)}
        \ge \frac{1}{2}\norm{(\chi^+-\chi^-)\nu_{\chi}\bar{A}D\tv\diff \fH^{d-1}}_{H^{-1}(\Omega)}.
    \end{equation*}
    Hence there is $\delta=\eps<\min (\frac{\dist(Q_k,\partial U_z)}{4},r_0,\eps_1,\eps_2)$ such that
    \begin{equation}
        \begin{aligned}
            &~~~~\frac{\norm{(\chi^+-\chi^-)\nu_{\chi}\bar{A}D\tv_{\delta}\diff \fH^{d-1}}_{H^{-1}(\Omega)}}{\norm{\tv_{\delta}}_{H^1(\Omega)}}\\
            & =\frac{\norm{(\chi^+-\chi^-)\nu_{\chi}\bar{A}D\tv_{\delta}\diff \fH^{d-1}}_{H^{-1}(\Omega)}}{\norm{(\chi^+-\chi^-)\nu_{\chi}\bar{A}D\tv\diff \fH^{d-1}}_{H^{-1}(\Omega)}}
            \times \frac{\norm{(\chi^+-\chi^-)\nu_{\chi}\bar{A}D\tv\diff \fH^{d-1}}_{H^{-1}(\Omega)}}{\norm{\tv}_{H^1(\Omega)}}\times\frac{\norm{\tv}_{H^1(\Omega)}}{\norm{\tv_{\delta}}_{H^1(\Omega)}}\\
            &\ge \frac{1}{4r}\frac{\norm{(\chi^+-\chi^-)\nu_{\chi}\bar{A}Dv\diff \fH^{d-1}}_{H^{-1}(\Omega)}}{\norm{\tv}_{H^1(\Omega)}}\\
            &\ge \frac{C}{\sqrt{r}}\norm{(\chi^+-\chi^-)\nu_{\chi}\bar{A}Dv\diff \fH^{d-1}}_{H^{-1}(\Omega)}.
        \end{aligned}
    \end{equation}
    Recalling~\eqref{eq::NonZeroLowBounOfOrigFunc} and taking $r\to 0^+$, we obtain
    \begin{equation}\label{eq::UnbddRelErrV_delta}
        \lim_{r\to 0^+}\frac{\norm{(\chi^+-\chi^-)\nu_{\chi}\bar{A}D\tv_{\delta}\diff \fH^{d-1}}_{H^{-1}(\Omega)}}{\norm{\tv_{\delta}}_{H^1(\Omega)}}
        =+\infty.
    \end{equation}  
    Let $\tilde{u}=\tv_{\delta}$ and $f=\tilde{L}\tv_{\delta}$ for sufficiently small $r$. Also recall the definition of $u$. Thus  
   ~\eqref{eq::LowBounByDualSoblev} and~\eqref{eq::UnbddRelErrV_delta} leads to the desired unboundedness of $\frac{\norm{\tilde{u}-u}_{H^1(\Omega)}}{\norm{\tilde{u}}_{H^1(\Omega)}}$.

    For the case of $d'> 1$, since Assumption~\ref{assum::BVCoeff} holds, by Theorem~\ref{thm::ExistAprioEstLinEllipEq&Sys} (\romannumeral2), there is constant $C>0$ such that $\norm{u}_{H^1(\Omega;\sR^{d'})}>\frac{1}{C}\norm{f^{\alpha}}_{H^{-1}(\Omega;\sR^{d'})}$.

    Thus we have
    \begin{align}
        \frac{\norm{\tilde{u}-u}_{H^1(\Omega;\sR^{d'})}}{\norm{\tilde{u}}_{H^1(\Omega;\sR^{d'})}}
        &\ge  \frac{\norm{(Tf)^{\alpha_0}-f^{\alpha_0}}_{H^{-1}(\Omega;\sR^{d'})}}{C\norm{\tilde{u}}_{H^1(\Omega;\sR^{d'})}}\nonumber\\
        &=\frac{\norm{\sum_{\beta=1}^{d'}(\chi^{\alpha_0\beta+}-\chi^{\alpha_0\beta-})\nu_{\chi^{\alpha_0\beta}}\bar{A}D\tilde{u}^{\beta}\diff \fH^{d-1}}_{H^{-1}(\Omega;\sR^{d'})}}{C\norm{\tilde{u}}_{H^1(\Omega;\sR^{d'})}}.\label{eq::LowBounByDualSoblevSys}
    \end{align}

    Recall the Claim 1 of proof of Theorem~\ref{thm::NecSufCondRemovSing4Sys}. Let $V=(0,\ldots,V^{\beta_0}(x),\ldots,0)^{\T}\in C^1(U_z;\sR^{d'})$, $v=\zeta V\in C^1(U_z;\sR^{d'})$ where $\zeta\in C_c^{\infty}(\Omega;\sR^{d'})$ and that $v_{\delta}=\rho_{\delta}*v$. Let $\tilde{u}=v_{\delta}$ and $u$ be the solutions to equation~\eqref{eq::OriginEq} and~\eqref{eq::ModEq} with $f=\tilde{L}\tilde{u}$, we obtain
    \begin{align*}
        \frac{\norm{\tilde{u}-u}_{H^1(\Omega;\sR^{d'})}}{\norm{\tilde{u}}_{H^1(\Omega;\sR^{d'})}}&\ge \frac{\norm{\sum_{\beta=1}^{d'}(\chi^{\alpha_0\beta+}-\chi^{\alpha_0\beta-})\nu_{\chi^{\alpha_0\beta}}\bar{A}D\tilde{u}^{\beta}\diff \fH^{d-1}}_{H^{-1}(\Omega;\sR^{d'})}}{C\norm{\tilde{u}}_{H^1(\Omega;\sR^{d'})}}\\
        &=\frac{\norm{(\chi^{\alpha_0\beta_0+}-\chi^{\alpha_0\beta_0-})\nu_{\chi^{\alpha_0\beta_0}}\bar{A}D\tilde{u}^{\beta_0}\diff \fH^{d-1}}_{H^{-1}(\Omega;\sR^{d'})}}{C\norm{\tilde{u}^{\beta_0}}_{H^1(\Omega;\sR^{d'})}}.
    \end{align*}
    The rest of that follows the previous one, and hence we omits it.
\end{proof}

\begin{prop}[relative deviation is large even for nearly RM-invariant data]\label{prop::relativedeviation}
    Suppose that Assumption~\ref{assum::BVCoeff} holds with $d'\ge 1$ and that $J_{\chi}$ is not dense in $\Omega$. For each $f\in L^2(\Omega;\sR^{d'})$, we define $f_{\parallel}$ and $f_{\bot}$ to be the projections of $f$ onto the closed spaces $\Ker(T-I)$ and $(\Ker(T-I))^{\bot}$, respectively.
    Then there exist constants $\eps_0>0$ and $C>0$ such that for any $0<\eps\leq \eps_0$, we have
    \begin{equation}
        \sup_{\norm{f_{\parallel}}_{L^2(\Omega;\sR^{d'})}=1,\norm{f_{\bot}}_{L^2(\Omega;\sR^{d'})}=\eps}\frac{\norm{\tilde{u}-u}_{H^1(\Omega;\sR^{d'})}}{\norm{\tilde{u}}_{H^1(\Omega;\sR^{d'})}}\ge C,
    \end{equation}
    where $u$ and $\tilde{u}$ are solutions to the original equation~\eqref{eq::OriginEq} and modified equation~\eqref{eq::ModEq}, respectively, with data $f=f_{\parallel}+f_{\bot}$.
\end{prop}
We stress that the constant $C>0$ is independent of $\eps$.

\begin{proof}
    The proof for case of equation and system are almost the same and here we provide a proof for case of equation. 
    
    Let $u_1$ and $\tilde{u}_1$ are solutions to the original equation~\eqref{eq::OriginEq} and modified equation~\eqref{eq::ModEq}, respectively, corresponding to the data $f_{\parallel}/\norm{f_{\parallel}}_{L^2(\Omega)}$.
    Let $u_2$ and $\tilde{u}_2$ are solutions to the original equation~\eqref{eq::OriginEq} and modified equation~\eqref{eq::ModEq}, respectively, corresponding to the data $f_{\bot}/\norm{f_{\bot}}_{L^2(\Omega)}$.
    Here $f_{\parallel}$ and $f_{\bot}$ will be determined later.
    
    By linearity, we have 
    \begin{equation*}
        \tilde{u}=\tilde{u}_{1}+\eps\tilde{u}_{2},\ u=u_{1}+\eps u_{2}.
    \end{equation*}    
    This gives rise to
    \begin{align}
        \frac{\norm{\tilde{u}-u}_{H^1(\Omega)}}{\norm{u}_{H^1(\Omega)}}&=\frac{\eps\norm{\tilde{u}_{2}-u_{2}}_{H^1(\Omega)}}{\norm{u}_{H^1(\Omega)}}\nonumber
        \\&\ge \frac{\eps\norm{\tilde{u}_{2}-u_{2}}_{H^1(\Omega)}}{\eps\norm{u_{2}}_{H^1(\Omega)}+\norm{u_{1}}_{H^1(\Omega)}}.\label{eq::Rel}
    \end{align}
    We then claim that we can choose $f_{\parallel}\in \Ker(T-I)$ to make $\norm{u_{1}}_{H^1(\Omega)}$ arbitrarily small. Notice that
    \begin{equation*}
        \norm{u_{1}}_{H^1(\Omega)}=\frac{\norm{u_{1}}_{H^1(\Omega)}}{\norm{u_{1}}_{H^2(\Omega)}}\norm{u_{1}}_{H^2(\Omega)}\le C\norm{f_{\parallel}}_{L^2}\frac{\norm{u_{1}}_{H^1(\Omega)}}{\norm{u_{1}}_{H^2(\Omega)}},
    \end{equation*}
    where the last inequality results from Theorem~\ref{thm::ExistAprioEstLinEllipEq&Sys}. Moreover, since $J_{\chi}$ is not dense in $\Omega$, we can choose a cube $Q_{r}(y)$ centered at $y$ with side length $r$ such that $Q_{r}(y)\cap J_{\chi} = \emptyset$ with $r$ to be determined later. We define for $x=(x_1,x_2,\ldots,x_d)\in Q_r(y)$ 
    \begin{equation*}
        u_{g}(x)=\prod_{i=1}^{d}\left[\cos \left( \frac{\pi(x_i-y_i)}{r}\right)+1\right].
    \end{equation*}
    Then
    \begin{equation*}
        \norm{u_{g}}_{H^1(\Omega)}^2\le (8r)^d+d\pi^2 (8r)^{d-1}\frac{1}{r}.
    \end{equation*}
    Since 
    \begin{equation*}
        \norm{u_{g}}_{H^2(\Omega)}^2\ge \norm{\Delta u_{g}}_{L^2(\Omega)}^2=d(3\pi r)^{d-1}\pi^4\frac{1}{r^3},
    \end{equation*}
    we have
    \begin{equation*}
        \frac{\norm{u_{g}}_{H^1(\Omega)}}{\norm{u_{g}}_{H^2(\Omega)}}\le \left(\frac{2\pi^2d8^{d-1}r^{d-2}}{\pi^{4}d(3\pi)^{d-1}r^{d-4}}\right)^{1/2}<r.
    \end{equation*}
    Let $g=\tilde{L}u_g$. It is clear that $g\in\Ker(T-I)$. By setting $f_{\parallel}=\frac{g}{\norm{g}_{L^2}}$, we have
    \begin{equation*}
        \norm{u_{1}}_{H^1(\Omega)}=\frac{\norm{u_{1}}_{H^1(\Omega)}}{\norm{u_{1}}_{H^2(\Omega)}}\norm{u_{1}}_{H^2(\Omega)}=\frac{\norm{u_{g}}_{H^1(\Omega)}}{\norm{u_{g}}_{H^2(\Omega)}}\norm{u_{1}}_{H^2(\Omega)}\le C r.
    \end{equation*}
    Choose a sufficiently small $r$, say $r\le \frac{\eps\norm{u_{2}}_{H^1(\Omega)}}{4C}$, and plugin it into \eqref{eq::Rel}. We obtain
    \begin{equation}\label{eq::RelErr_u}
        \frac{\norm{\tilde{u}-u}_{H^1(\Omega)}}{\norm{u}_{H^1(\Omega)}}>\frac{4}{5}\frac{\norm{\tilde{u}_{2}-u_{2}}_{H^1(\Omega)}}{\norm{u_{2}}_{H^1(\Omega)}}.
    \end{equation}
    Notice that
    \begin{equation*}
        \frac{\norm{u}_{H^1(\Omega)}}{\norm{\tilde{u}}_{H^1(\Omega)}}=\frac{\norm{u_1+\eps u_2}_{H^1(\Omega)}}{\norm{u_1+\eps\tilde{u}_2}_{H_1(\Omega)}}\geq \frac{\eps\norm{u_2}_{H^1(\Omega)}-\norm{u_1}_{H^1(\Omega)}}{\eps \norm{\tilde{u}_2}_{H_2(\Omega)}+\norm{u_1}_{H_1(\Omega)}}\ge \frac{\eps\norm{u_2}_{H^1(\Omega)}-Cr}{\eps\norm{\tilde{u}_2}_{H_1(\Omega)}+Cr }>\frac{1}{2}.
    \end{equation*}
    This with \eqref{eq::RelErr_u} gives rise to
    \begin{equation*}
        \frac{\norm{\tilde{u}-u}_{H^1(\Omega)}}{\norm{\tilde{u}}_{H^1(\Omega)}}>\frac{2}{5}\frac{\norm{\tilde{u}_{2}-u_{2}}_{H^1(\Omega)}}{\norm{u_{2}}_{H^1(\Omega)}}.
    \end{equation*}
    
    The last right hand side is a constant, and hence the proof is completed.
\end{proof}

\subsection{Implicit bias towards the solution to the modified equation}\label{sec::ImplicitBias}
For a given function $w$, the population risk in the residual minimization method for solving the modified equation/system reads as
\begin{equation}
    \tilde{R}(w)=\int_{\Omega}(\tilde{L}w-f)^2\diff{x} + \gamma \int_{\partial\Omega}(Bw-g)^2\diff{x}.
\end{equation}
In the following theorem and proposition, we show the implicit bias of RM methods via an energetic approach. Roughly speaking, for a function close to $u$, it has large (modified) risk $\tilde{R}$; while for any function has small (modified) risk $\tilde{R}$, it should be close to $\tilde{u}$ in the sense of $H^1$ norm.

\begin{theorem}
    [bias against the solution to the original equation]\label{thm::ImplicitBiasLinSys}
    Suppose that Assumption~\ref{assum::BVCoeff} and~\ref{assum::AprioriEstLinSys} hold with $d'\geq 1$. Let $u$ and $\tilde{u}$ be solutions to the original equation/system~\eqref{eq::OriginEq} and modified equation/system~\eqref{eq::ModEq} with data $f\in X\backslash \Ker(T-I)$, respectively. 

    There exist constants $\eps_0>0$ and $C_0>0$ such that for all $0<\eps<\eps_0$ and for all
    $v\in B_{\eps}(u)$, we have $\tilde{R}(v)\geq C_0>0$.
    Here $B_{\eps}(u)=\{w\in H^1_0(\Omega;\sR^{d'})\cap H^2(\Omega;\sR^{d'})\colon\,\norm{w-u}_{H^1(\Omega;\sR^{d'})}<\eps, \tilde{L}w\in X\}$.
\end{theorem}
We remark that Theorem~\ref{thm::ImplicitBiasLinSys} implies that there is $\eps>0$ such that $B_{\eps}(u)\cap \tilde{B}_{\eps}(\tilde{u})=\emptyset$, where $\tilde{B}_{\eps}(\tilde{u})=\{\tilde{v}\in H^1_0(\Omega;\sR^{d'})\cap H^2(\Omega;\sR^{d'})\colon\,\tilde{R}(\tilde{v})<\eps\}$.
\begin{proof}
    Since $f\in X\backslash \Ker(T-I)$, by Theorem~\ref{thm::EigenValueEigenSpace}, there is a constant $C_1$ such that 
    \begin{equation}
        \norm{u-\tilde{u}}_{H^1(\Omega;\sR^{d'})}\ge C_1.
    \end{equation}
    By linearity of the modified operator $\tilde{L}$, we have for all $v\in H^1_0(\Omega;\sR^{d'})\cap H^2(\Omega;\sR^{d'})$
    \begin{equation}\label{eq::BallInXSpace}
        \left(\tilde{L}v-\tilde{L}\tilde{u}\right)^{\alpha}=\left(\tilde{L}v\right)^{\alpha}-f^{\alpha}.
    \end{equation}
    Since $\tilde{L}v\in X$, by the definition of $X$ and Assumption~\ref{assum::AprioriEstLinSys}
    \begin{equation}
        \norm{\tilde{u}-v}_{H^{2}(\Omega;\sR^{d'})} \leq C\norm{f-\tilde{L}v}_{L^2(\Omega;\sR^{d'})}.
    \end{equation}
    Let $\eps_0=\frac{C_1}{2}$. The for all $\eps<\eps_0$ and $v\in B_{\eps}(u)$, we have
    \begin{equation*}
        \norm{v-\tilde{u}}_{H^{2}(\Omega;\sR^{d'})}\ge \norm{v-\tilde{u}}_{H^{1}(\Omega;\sR^{d'})}\ge \norm{\tilde{u}-u}_{H^{1}(\Omega;\sR^{d'})}-\norm{u-v}_{H^{1}(\Omega;\sR^{d'})}>\frac{C_1}{2}.
    \end{equation*}
    This together with~\eqref{eq::BallInXSpace} implies for all $v\in B_{\eps}(u)$
    \begin{equation*}
        \sqrt{\tilde{R}(v)}\geq\norm{\tilde{L}v-f}_{L^2(\Omega;\sR^{d'})}\geq \frac{1}{C}\norm{\tilde{u}-v}_{H^2(\Omega;\sR^{d'})}>\frac{C_1}{2C}.
    \end{equation*}
    The proof is completed by setting $C_0=\frac{C^2_1}{4C^2}$.
\end{proof}

\begin{proposition}[bias towards the solution to the modified equation/system]\label{prop::Bias}
    Suppose that Assumption~\ref{assum::BVCoeff} and~\ref{assum::AprioriEstLinSys} hold with $d'\geq 1$. Let $\tilde{u}$ be the solution to the modified equation~\eqref{eq::ModEq} with data $f\in X$. 
    \begin{enumerate}[(i).]
        \item Then for all $w\in H^1_0(\Omega;\sR^{d'})\cap H^2(\Omega;\sR^{d'})$, there is constant $C>0$ such that 
    \begin{equation}
        \norm{w-\tilde{u}}_{H^1(\Omega;\sR^{d'})}\le C\sqrt{\tilde{R}(w)}.
    \end{equation}
    \item Suppose that the function $w\in H_0^1(\Omega;\sR^{d'})\cap H^2(\Omega;\sR^{d'})$, the random variable $X$ is sampled from the uniform distribution over $\Omega$, the random variable $\rY=\Abs{\Omega}(Lw(\rX)-f)^2$ with covariance of $\rY$ satisfies $\sV[\rY]< +\infty$. Then for any $\delta>0$, with probability $1-\delta$ over the choice of independent uniformly distributed data $S:=\{x_i\}_{i=1}^n$ in $\Omega$, we have 
    \begin{equation*}
        \tilde{R}(w)-\tilde{R}_S(w)\le \frac{1}{\sqrt{n}}\sqrt{\frac{\sV[\rY]}{\delta}},
    \end{equation*}
    \end{enumerate}
    In particular, if $\tilde{R}(w)<\eps$ for some small $\eps$, then $\norm{w-\tilde{u}}_{H^1(\Omega;\sR^{d'})}\le C\sqrt{\eps}$.
\end{proposition}
\begin{proof}
    (\romannumeral1) is a direct consequence from Theorem~\ref{thm::ExistAprioEstLinEllipEq&Sys}. And (\romannumeral2) is a direct consequence from Monte Carlo integration.
\end{proof}

We remark that in Proposition~\ref{prop::Bias}, for $w\in H_0^1(\Omega;\sR^{d'})\cap H^2(\Omega;\sR^{d'})$ with $\sV[\rY]< +\infty$, if $\tilde{R}(w)\le \eps$, for $n$ sufficiently large, we have $\norm{w-\tilde{u}}_{H^1(\Omega;\sR^{d'})}\le C\sqrt{\eps}$. Moreover, the condition $\sV[\rY]< +\infty$ is equivalent to the inequality $\norm{(D_x^2w- D_x^2\tilde{u})^2}_{L^2(\Omega;\sR^{d'})}<+\infty.$
\begin{remark}
     Theorem~\ref{thm::ImplicitBiasLinSys} and Proposition~\ref{prop::Bias} together lead to the implicit bias of RM method in solving PDE: it will bias the numerical solution against the solution to original equation~\eqref{eq::OriginEq} and towards the solution to modified equation~\eqref{eq::ModEq}.
\end{remark}
\section{Discussion: extension to quasilinear elliptic equations}\label{sec::ExtQausiLinEllipEq}

In this section, we show that some of our results still work for the case of quasilinear elliptic equation. In Section~\ref{sec::equationQuasLin}, we introduce the equation as the Euler--Lagrange equation of a variational problem which arise typically in the materials science, together with its modified equation, and also provide some assumptions to ensure the existence. Besides, some auxiliary lemmas and discussions can be found in this subsection for later analysis. The techniques we developed in Section~\ref{sec::RemovSing} is also applicable to the case of quasilinear equation (See Section~\ref{sec::SingQuasiEllipEq}) and the deviation still occur in the case of quasilinear equation (See Section~\ref{sec::DeviQuasLin}).  

\subsection{Quasilinear elliptic equations with BV coefficients}\label{sec::equationQuasLin}
Assume that $\Omega \subseteq \sR^d$ is a bounded domain with $C^{1,1}$ boundary $\partial\Omega$, and $L: \sR^d \times \sR\times \Omega \to \sR$ is a function taking the form
\begin{equation*} 
    L(\xi,z,x) = \chi(x)W(\xi,x) - f(x)z.
\end{equation*}

Consider the minimization problem of the energy functional $I[u]=\int_{\Omega}L(D u,u,x)\diff{x}$, that is,
\begin{align}\label{eq::CBEnergyFun}
    \min_{u\in H^1_0(\Omega)} \int_{\Omega}\left(\chi W(D u,x) - fu\right)\diff{x}.
\end{align}
This type of problem arises from, for example, the coexistence of multiple phases in solids. In such case, $u$ is the displacement field of a solid, $f$ is the body force due to some external field, and $\chi(x)$ is a piece-wise constant function indicating different phases. In each phase, the stored energy density $W(D u, x)$ can be derived from the molecular potentials by using the Cauchy--Born rule. See for example~\cite{born1954dynamical,ming2007cauchy} and the references therein. We also remark that the displacement $u$ is usually a vector field, however, for simplicity, we assume it is a scalar field in this work. The results can be extended to the vector field setting.

The equation for the Euler--Lagrange equation of~\eqref{eq::CBEnergyFun} reads as
\begin{equation}\label{eq::QuasiLinOriginEq}
     \left\{\begin{aligned}
    Q[u]&=f & & \text {in}\  \Omega, \\
    u &=0 & & \text {on}\  \partial\Omega,
    \end{aligned}\right.
\end{equation}
where $Q$ is the quasilinear operator defined as
\begin{equation}
    Q[u]=-\divg\cdot\left(\chi D_{\xi}W(D u, x)\right).
\end{equation}
We emphasize that $Q[u] -Q[v] \neq Q[u-v]$ in general.

\begin{remark}[recover the linear case]\label{rmk::RecoverLin}
    Notice that~\eqref{eq::QuasiLinOriginEq} recovers the linear elliptic equation~\eqref{eq::OriginEq} when we set
    \begin{equation*} 
        W(\xi,x)= \frac{1}{2}\xi^\T \bar{A}\xi=\frac{1}{2}\sum_{i,j=1}^d\bar{A}_{ij}\xi_{i}\xi_{j}.
    \end{equation*}
\end{remark}
In the spirit of modified equation derived in Section~\ref{sec::DeriveModEq}, we shall consider the equation for the modified equation in non-divergence form as follows.

\begin{equation}\label{eq::QuasLinModEq}
\left\{\begin{aligned}
    \tilde{Q}[u]&=f & & \text {in}\  \Omega, \\
    u &= 0 & & \text{on}\  \partial\Omega,
\end{aligned}\right.
\end{equation}
where
\begin{equation}
    \tilde{Q}[u]=-\left(\sum_{i,j=1}^dD_{\xi_i}D_{\xi_j}W(Du,x)D_{ij}u+\sum_{i=1}^{d}D_{\xi_i}D_{x_i}W(Du,x)\right)\chi-\sum_{i=1}^{d}D_{i}^{\mathrm{a}}\chi D_{\xi_i}W(Du,x).
\end{equation}

To work on these quasilinear equation, we require the following technical assumptions.
\begin{assumption}[BV coeffcients]\label{assum::BvCoeQuasLin}
    Assume $\chi\in SBV^{\infty}(\Omega)$, $0<\fH^{d-1}(J_{\chi})<+\infty$, $f\in L^p(\Omega)$ with $p>d$, and $W\in C^3(\sR^d\times \Omega)$. Also assume there are constants $\chi_{\min},\chi_{\max}>0$ such that for all $x\in \Omega$
    \begin{equation*}
        \chi_{\min}\leq \chi(x)\leq \chi_{\max}.
    \end{equation*}
\end{assumption}

\begin{assumption}[coercivity, convexity, boundedness]\label{assum::CoerCvxBddQuasLin}
    Let $Q$ be the operator defined in~\eqref{eq::QuasiLinOriginEq}. Assume the following.
    \begin{enumerate}[(i)]
        \item (coercivity) There exist constants $c_1,c_2>0$ such that for all $\xi\in\sR^d, x\in\Omega$
        \begin{equation*} 
            W(\xi,x)\geq c_1\Abs{\xi}^2-c_2
        \end{equation*}
        and $2\chi_{\min}c_1\geq \sup_{u\in H^1_0(\Omega)}\frac{\norm{u}_{L^2(\Omega)}}{\norm{Du}_{L^2(\Omega)}}$ where the right hand side is the inverse of the best Poincare constant. 
        \item  (uniform convexity in $\xi$) There exist constants $\lambda,\Lambda>0$ such that for all $\xi, \xi' \in \sR^{d}, x \in \Omega$
        \begin{equation*}  
            \lambda\abs{\xi'}^{2}\leq(\xi')^\T D_\xi^2 W(\xi,x)\xi' \leq \Lambda\abs{\xi'}^{2}.
        \end{equation*}
        \item (boundedness) There exists $c_3>0$ such that for all  $\xi \in \sR^{d}, x \in \Omega$.
            \begin{align*}
                \abs{W(\xi, x)}&\leq c_3(\abs{\xi}^2+1),\\
                \Abs{D_{\xi} W(\xi, x)}&\leq c_3\left(\abs{\xi}+1\right).
            \end{align*} 
    \end{enumerate} 
\end{assumption}

\begin{assumption}[boundedness, Lipschitz continuity, growth condition]\label{assum::BddLipGrowQuasiLinNonDiv}
    Let $\tilde{Q}$ be the operator defined in~\eqref{eq::QuasLinModEq}. Also let $p>d$. Assume the following.
    \begin{enumerate}[(i)]
        \item (boundedness) There exists a non-negative function $b_1\in L^{2p}(\Omega)$ satisfying 
        \begin{equation*}
            \Abs{D_{x}\odot D_\xi W(\xi,x)}\leq b_1(x)(1+\Abs{\xi})
        \end{equation*}
        for all $\xi\in \sR^d, x\in \Omega$;
        \item (Lipschitz continuity) $D_{x}\odot D_\xi W(\xi,x)$ and $D_\xi W(\xi,x)$ are $c_4$-Lipschitz continuous in $\xi$, that is, 
        \begin{equation*} 
            \Abs{D_{x}\odot D_\xi W(\xi,x)- D_{x}\odot D_\xi W(\xi',x)} +\Abs{D_\xi W(\xi,x)-D_\xi W(\xi',x)}\leq c_4\Abs{\xi-\xi'}
        \end{equation*}
        for all $\xi,\xi'\in\sR^{d}$ and $\fL^d$-a.e. $x\in \Omega$;
        \item (growth condition) There exists a constant $c_5>0$ satisfying  
        \begin{equation*}
            \Abs{D_{x} D_\xi^2 W(\xi,x)}+
            \Abs{D_\xi^3 W(\xi,x)}\leq c_5\Abs{\xi}
        \end{equation*}
        for all $\xi\in \sR^d, x\in \Omega$.
    \end{enumerate}
\end{assumption}

\begin{theorem}[existence of solution to original quasilinear elliptic equation]\label{thm::ExistQuasiLinDivForm}
    Suppose that Assumption~\ref{assum::BvCoeQuasLin} and~\ref{assum::CoerCvxBddQuasLin} hold. Then for any $f\in L^2(\Omega)$, there is a solution $u\in H_0^1(\Omega)$ to the original equation~\eqref{eq::QuasiLinOriginEq}. 
\end{theorem}

\begin{theorem}[existence of solution to modified  quasilinear elliptic equation]\label{thm::ExistQuasiLinNonDivForm}
    Suppose that Assumption~\ref{assum::BvCoeQuasLin},~\ref{assum::CoerCvxBddQuasLin} and~\ref{assum::BddLipGrowQuasiLinNonDiv} hold with $p>d$. For any $f\in L^p(\Omega)$,   there is a solution $u\in W_{0}^{1,p}(\Omega) \cap W^{2,p}(\Omega)$ to the modified equation~\eqref{eq::QuasLinModEq}.
\end{theorem}
The proofs of these two theorems are left to Appendix~\ref{sec::ExisProof4QuasEq}. We emphasize that these two existence theorems are essentially proved in literature (for which we refer the readers to Appendix~\ref{sec::ExisThm4QuasLinInLitera}) under very general settings. Thus we check that the required assumptions are satisfied in our proofs of Theorems~\ref{thm::ExistQuasiLinDivForm} and~\ref{thm::ExistQuasiLinNonDivForm} in Appendix~\ref{sec::ExisProof4QuasEq}.

Note that in the quasilinear case the solution to~\eqref{eq::QuasiLinOriginEq} and~\eqref{eq::QuasLinModEq} may not be unique, and hence we have to use the notion of solutions sets. These solution sets as well as their Hausdorff distance are defined as follows. 

\begin{definition}[solution sets]
    Suppose that Assumptions~\ref{assum::BvCoeQuasLin},~\ref{assum::CoerCvxBddQuasLin} and~\ref{assum::BddLipGrowQuasiLinNonDiv} hold with $p>\max\{2,d\}$. We define the \textbf{solution sets} to the original equation~\eqref{eq::QuasiLinOriginEq} and modified equation~\eqref{eq::QuasLinModEq} with data $f\in L^p(\Omega)$. More precisely, we define
    \begin{align*}
        \tU 
        &= \{u\colon\,u\in H^1_0(\Omega) \text{ solves equation~\eqref{eq::QuasiLinOriginEq} with data $f$}\},\\
        \widetilde{\tU} 
        &= \{\tilde{u}\colon\,\tilde{u}\in W^{1,p}_0(\Omega)\cap W^{2,p}(\Omega) \text{ solves equation~\eqref{eq::QuasLinModEq} with data $f$}\}.
    \end{align*}
\end{definition}

\begin{definition}[Hausdorff distance from set of solutions of equation~\eqref{eq::QuasLinModEq} to equation~\eqref{eq::QuasiLinOriginEq}]
    \begin{equation*} 
        \dist_{\mathrm{H}}(\widetilde{\tU},\tU) = \sup_{\tilde{u}\in \widetilde{\tU}}\inf_{u\in \tU}\norm{u-\tilde{u}}_{H^1(\Omega)},
    \end{equation*}
\end{definition}
As we will see in Proposition~\ref{prop::Close}, for a given $f\in L^p(\Omega)$, $\tU$ is closed in $H^1(\Omega)$ norm. Thus we conclude that
(\romannumeral1) if $\dist_{\mathrm{H}}(\widetilde{\tU},\tU)=0$, then for each $\tilde{u}\in \widetilde{\tU}$, by the closeness of $\tU$ there is a $u\in \tU$ such that $u=\tilde{u}$ $\fL^d$-a.e. $x\in\Omega$, which means the $\tilde{u}$ will not deviate from $u$;
(\romannumeral2) if $\dist_{\mathrm{H}}(\tU,\widetilde{\tU})>0$, which means $\widetilde{\tU}\backslash \tU\neq \emptyset$, then there is $\tilde{u}$ deviated from $u$.

Therefore, $\dist_{\mathrm{H}}(\widetilde{\tU},\tU)$ characterize how much $\widetilde{\tU}$ deviates from $\tU$. Moreover, if $\widetilde{\tU},\tU$ are both singletons, the above discussion reduces to the deviation results studied in Section~\ref{sec::Lin}.

\subsection{Singularity in quasilinear elliptic equation}\label{sec::SingQuasiEllipEq}
In this section, we apply the removable singularity (Theorem \ref{thm::NecSufCondRemovSing}) to the case of quasilinear equation, and obtain Theorem~\ref{thm::NecSufCondRemovSingQuasiLin} which is parallel to Theorems~\ref{thm::NecSufCondRemovSingLin} and \ref{thm::NecSufCondRemovSing4Sys}. And to obtain the estimates in the quasilinear case, we need the following assumption.
\begin{assumption}\label{assum::Boun} Assume that there exists a constant $c_6>0$ such that for all $\xi\in\sR^d$, $x\in \Omega$
    \begin{equation*} 
        \Abs{D_{\xi}W(\xi,x)}\leq c_6\Abs{\xi}.
    \end{equation*}
    Also assume that the original equation with data $f=0$, 
    \begin{equation}\label{eq::UniqForDataZero}
        \left\{
        \begin{aligned}
        -\divg\cdot\left(\chi D_{\xi}W(Du,x)\right)&=0 & & \text {in}\  \Omega, \\
        u &= 0 & & \text{on}\  \partial\Omega,
        \end{aligned}
        \right.
    \end{equation}
    has a unique solution, namely $u_0= 0$ on $\Omega$. 
\end{assumption}

\begin{lemma}
    Suppose the Assumption~\ref{assum::BvCoeQuasLin} and~\ref{assum::Boun} hold. Then $u_0=0$ is also a solution to the modified equation~\eqref{eq::QuasLinModEq} with data $f=0$.
\end{lemma}
\begin{proof}
    Consider the divergence in~\eqref{eq::UniqForDataZero} and note that $\chi D_{\xi_i}W(Du_0,x)\in SBV(\Omega)$ for all $i\in\{1,\ldots,d\}$. By the structure theorem, namely Theorem~\ref{thm::StructureThmBV}, we have:
    \begin{equation*} 
        -\left(\sum_{i,j=1}^dD_{\xi_{i}}D_{\xi_{j}}W(Du_0,x)D_{ij}u_0+\sum_{i=1}^{d}D_{\xi_i}D_{x_i}W(Du_0,x)\right)\chi-\sum_{i=1}^{d}D_{i}\chi D_{\xi_{i}}W(Du_0,x)=0,
    \end{equation*}
    where $D_{i}\chi=D_{i}^{\mathrm{a}}\chi+D_{i}^{\mathrm{j}}\chi$. Since $\Abs{D_{\xi}W(\xi,x)}\leq c_6\Abs{\xi}$, this leads to $D_{\xi_i}W(Du_0,x)=0$. Thus
        \begin{equation*} 
        -\left(\sum_{i,j=1}^dD_{\xi_{i}}D_{\xi_{j}}W(Du_0,x)D_{ij}u_0+\sum_{i=1}^{d}D_{\xi_i}D_{x_i}W(Du_0,x)\right)\chi-\sum_{i=1}^{d}D^{\mathrm{a}}_{i}\chi D_{\xi_{i}}W(Du_0,x)=0,
    \end{equation*}
    and the proof is completed.
\end{proof}

\begin{theorem}[characterization of $\chi$ with test function $W(Du,x)$]\label{thm::NecSufCondRemovSingQuasiLin}
    Suppose that Assumption~\ref{assum::BvCoeQuasLin},~\ref{assum::CoerCvxBddQuasLin},~\ref{assum::BddLipGrowQuasiLinNonDiv} and~\ref{assum::Boun} hold. Then there exists $v_{\delta}\in C_c^{\infty}(\Omega)$ such that 
    \begin{equation}\label{eq::removablequasi}
        \int_{J_{\chi}}(\chi^+-\chi^-)\nu_{\chi}^{\T} D_{\xi}W(Dv_{\delta},x) \diff{\fH^{d-1}}> 0.
    \end{equation} 
\end{theorem}

\begin{proof}
    Since Assumption~\ref{assum::Boun} holds, we have $D_{\xi}W(0,x) =0$ for all $x\in\Omega$. By using the Newton–Leibniz theorem
    \begin{equation}
        D_{\xi}W(Du,x)=D_{\xi}W(0,x)+\int_{0}^{1}\frac{\D}{\diff{t}}D_{\xi}W(tDu,x)\diff{t}=\int_{0}^1A_u^tDu\diff{t},
    \end{equation}
    where in the second equality we use the fact that $D_{\xi}W(0,x)=0$, and $A_u^t$ is a matrix-valued function with entries defined as $[A_u^t]_{ij} = [D_{\xi_{i}}D_{\xi_{j}}W(tDu,x)]_{ij}$, $t\in[0,1]$. Thus 
    \begin{equation*}
        \int_{J_{\chi}}(\chi^+-\chi^-)\nu_{\chi}^{\T} D_{\xi}W(Du,x) \diff{\fH^{d-1}}=\mu(\Omega; \chi, \int_{0}^{1}A_u^t\diff{t}, Du).
    \end{equation*}
    Notice that, by Assumption~\ref{assum::CoerCvxBddQuasLin} (\romannumeral2), we have for all $\xi\in \sR^d,\ t\in [0,1]$
    \begin{equation*}
        \lambda\abs{\xi}^2
        \leq \xi^{\T}A_u^t\xi
        \leq \Lambda \abs{\xi}^2.   
    \end{equation*}
    By letting $\Upsilon[u]=\int_{0}^{1}A_u^t\diff{t}$ and Theorem~\ref{thm::NecSufCondRemovSing}, there is $v\in C_c^1(\Omega)$ such that
    \begin{equation}
        \mu(\Omega; \chi, \int_{0}^{1}A_v^t\diff{t}, Dv)>0,
    \end{equation}
    and this is equivalent to
    \begin{equation}\label{eq::RemovSingQuasLin}
        \int_{J_{\chi}}(\chi^+-\chi^-)\nu_{\chi}^{\T} D_{\xi}W(Dv,x) \diff{\fH^{d-1}}>0.
    \end{equation}
    Let $v_{\delta} =v *\rho_{\delta}$ with $\delta<\frac{1}{4}\dist\left(\Psi(\bar{Q}_k),\partial U_z)\right)$, we claim that
    \begin{equation}\label{eq::ApproxMolliQuasLin}
       \lim_{\delta\to 0
       } \int_{J_{\chi}}(\chi^+-\chi^-)\nu_{\chi}^{\T} \left(D_{\xi}W(Dv,x)-D_{\xi}W(Dv_{\delta},x)\right) \diff{\fH^{d-1}}=0.
    \end{equation}
    To prove \eqref{eq::ApproxMolliQuasLin}, we use the Newton–Leibniz theorem again:
    \begin{align}
    & D_{\xi_{i}}W\left(D v, x\right)=D_{\xi_{i}}W\left(D v_{\delta}, x\right)+\int_{0}^{1} \frac{d}{d t} D_{\xi_i}W\left(D v_{\delta}+t\left(D v-D v_{\delta}, x\right) d t\right.\nonumber\\
    =& D_{\xi_{i}}W\left(D v_{\delta}, x\right)+\int_{0}^{1} \sum_{j=1}^{d} D_{\xi_{i}}D_{\xi_{j}}W\left(D v_{\delta}+t\left(D v-D v_{\delta}\right), x\right)\left(D_{j} v -D_{j} v_{\delta}\right)\diff{t}.
    \end{align}
    Hence
    \begin{equation}
        D_{\xi}W(Dv,x)-D_{\xi}W(Dv_{\delta},x)=\int_{0}^{1}A_{v,v_{\delta}}^t D(v-v_{\delta})\diff{t},
    \end{equation}
    where $A_{v,v_{\delta}}^t= \{D_{\xi_{i}}D_{\xi_{j}}W\left(D v_{\delta}+t\left(D v-D v_{\delta}\right), x\right)\}_{ij}$. Recall that $Dv_{\delta}\to Dv$ uniformly as $\delta\to 0$. Thus
    \begin{align}
             &\int_{J_{\chi}}(\chi^+-\chi^-)\nu_{\chi}^{\T} \left(D_{\xi}W(Dv,x)-D_{\xi}W(Dv_{\delta},x)\right) \diff{\fH^{d-1}}\nonumber\\
             =& \int_{J_{\chi}}(\chi^+-\chi^-)\nu_{\chi}^{\T} \left(\int_{0}^{1}A_{v,v_{\delta}}^t D(v-v_{\delta})\diff{t}\right) \diff{\fH^{d-1}}\nonumber\\
             \le& 2\norm{\chi}_{L^{\infty}(\Omega)}\fH^{d-1}(J_{\chi})\Lambda\norm{Dv-Dv_{\delta}}_{L^\infty(\Omega)},
    \end{align}
    which means 
    \begin{equation}\label{eq::UniApprox4MolliQuasLin}
       \lim_{\delta\to 0
       } \int_{J_{\chi}}(\chi^+-\chi^-)\nu_{\chi}^{\T} \left(D_{\xi}W(Dv,x)-D_{\xi}W(Dv_{\delta},x)\right) \diff{\fH^{d-1}}=0.
    \end{equation}
    Combining \eqref{eq::RemovSingQuasLin} and \eqref{eq::UniApprox4MolliQuasLin}, we have
    \begin{equation*}
        \int_{J_{\chi}}(\chi^+-\chi^-)\nu_{\chi}^{\T} D_{\xi}W(Dv_{\delta},x) \diff{\fH^{d-1}}> 0,
    \end{equation*}
    for $\delta$ small enough. Hence the proof is completed.
\end{proof}

\subsection{Deviation occurs for quasilinear elliptic equations}\label{sec::DeviQuasLin}
We show that the deviation is not exclusive for the linear problems. The deviation occurs for the quasilinear elliptic equations. We start with some properties of the solution set $\tU$.

\begin{prop}\label{prop::Close}
   Suppose that Assumption~\ref{assum::BvCoeQuasLin},~\ref{assum::CoerCvxBddQuasLin} and~\ref{assum::BddLipGrowQuasiLinNonDiv} hold with $p>\max\{2,d\}$. Given any $f\in L^p(\Omega)$, we have $\tU$ is closed under $H^1$ norm. 
\end{prop}
\begin{proof}
   Let $\{u_k\} \subseteq{\tU} \subseteq H^1_0(\Omega)$ be a convergent sequence which belong to the solution set to equation~\eqref{eq::QuasiLinOriginEq} with data $f$. 
   
   According to the completeness of $H_0^1(\Omega)$, there exists a $u\in H_0^1(\Omega)$ such that \begin{equation*}
       \lim_{k\to\infty}\norm{u-u_k}_{H^1(\Omega)}=0.
   \end{equation*}
   Let $g = Q[u]\in H^{-1}(\Omega)$. To prove $\tU$ is closed under $H^1$ norm, we only require $g=f$.
    
    Using the Newton–Leibniz theorem
    \begin{equation*} 
        D_{\xi_{i}}W(Du,x)-D_{\xi_{i}}W(Du_k,x)=\int_{0}^{1}A_{u,u_k}^t D(u-u_k)\diff{t}.
    \end{equation*}
    Thus for each $\varphi\in H^1_0(\Omega)$
    
    \begin{align*} 
        \langle g-f,\varphi\rangle_{H^{-1}(\Omega),H^1_0(\Omega)}
        &=\int_{\Omega}\chi(x)D\varphi[D_{\xi}W(Du,x)-D_{\xi}W(Du_k,x)]\diff{x}\\
        &\leq C \norm{\varphi}_{H^1_0(\Omega)}\norm{u-u_k}_{H^1_0(\Omega)},
    \end{align*}
    with right hand term tends to 0 as $k\to \infty$. Hence $u$ is a solution to equation~\eqref{eq::QuasiLinOriginEq} with right hand term as $f$. 
\end{proof}

\begin{theorem}[deviation occurs for quasilinear elliptic equation]\label{thm::DeviQuasLin}
    Suppose that Assumption~\ref{assum::BvCoeQuasLin},~\ref{assum::CoerCvxBddQuasLin},~\ref{assum::BddLipGrowQuasiLinNonDiv} and~\ref{assum::Boun} hold with $p>d$. Then there exists $f\in L^p(\Omega)$ such that $ \dist_{\mathrm{H}}(\tU,\widetilde{\tU})>0$ where $\tU$ and $\widetilde{\tU}$ are solution sets to the original equation~\eqref{eq::QuasiLinOriginEq} and modified equation~\eqref{eq::QuasLinModEq}, respectively.
\end{theorem}

\begin{proof}
    By Thoerem~\ref{thm::NecSufCondRemovSingQuasiLin}, there is a $v_{\delta}\in C_c^{\infty}(\Omega)$ such that 
    \begin{equation*}
        \int_{J_{\chi}}(\chi^+-\chi^-)\nu_{\chi}^{\T} D_{\xi}W(Dv_{\delta},x)\diff{\fH^{d-1}}> 0,
    \end{equation*}
    which means $\left(\tilde{Q}[v_{\delta}]-Q[v_{\delta}]\right)(\Omega)=\int_{J_{\chi}}(\chi^+-\chi^-)\nu_{\chi}^{\T} D_{\xi}W(Dv_{\delta},x)\diff{\fH^{d-1}}> 0$. Here we regard $\tilde{Q}[v_{\delta}]-Q[v_{\delta}]$ as a Radon measure.
    
    Let $f= Q[v_{\delta}]$ and $g=\tilde{Q}[v_{\delta}]$. We claim that $v_{\delta}$ is not any solution to
    \begin{equation}
        \begin{cases}
            Q [u] =g ,\\
            u = 0.
        \end{cases}
    \end{equation}
    Otherwise, $v_{\delta}$ is a solution to the above equation and we have $f=Q[v_{\delta}]=g$, which leads to a contradiction. Hence $v_{\delta}\notin \tU(g)$. By Proposition~\ref{prop::Close}, $\tU(g)$ is closed under $H^1$ norm. Thus $\dist(v_{\delta},\tU(g))>0$, and hence $\dist_{\mathrm{H}}(\widetilde{\tU}(g),\tU(g))>0$.
\end{proof}

Now we conclude this work with a theorem showing the relative deviation can be very large, for the quasilinear elliptic equations, just as the linear case.

\begin{theorem}[Infinite relative deviation for quasilinear elliptic equation]\label{thm::UnbddRQuasLinEllipEq}
     Suppose that Assumption~\ref{assum::BvCoeQuasLin},~\ref{assum::CoerCvxBddQuasLin},~\ref{assum::BddLipGrowQuasiLinNonDiv} and~\ref{assum::Boun} hold with $p>d$. Also suppose that there are a constant $r_0>0$ and a countably many $(d-1)$-dimensional $C^1$ manifolds $\{S_i\}_{i=1}^{\infty}$ such that $\fH^{d-1}(J_{\chi}\backslash(\bigcup\nolimits_{i=1}^{\infty}S_i))=0$ and that for $i\neq j$, $\dist(S_i,S_j)=\inf_{x\in S_i,y\in S_j}\abs{x-y}\ge r_0>0$. Then 
     \begin{equation}
         \sup_{f\in L^p(\Omega)}\sup_{\tilde{u}\in \widetilde{\tU}}\inf_{u\in \tU}\frac{\norm{\tilde{u}-u}_{H^1(\Omega)}}{\norm{\tilde{u}}_{H^1(\Omega)}}=+\infty,
     \end{equation}
     where $\tU$ and $\widetilde{\tU}$ are the solution sets to the original equation~\eqref{eq::QuasiLinOriginEq} and modified equation~\eqref{eq::QuasLinModEq} with data $f\in L^p(\Omega)$, respectively. 
\end{theorem}

\begin{proof}
    Recall that
    \begin{equation*}
        Q[\tilde{u}]-Q[u]=(\chi^+-\chi^-)\nu_{\chi}D_{\xi}W(D\tilde{u},x)\diff{\fH^{d-1}}.
    \end{equation*}
    which is non-zero.
    As above, by using the Newton--Leibniz theorem again, we have
    \begin{equation*}
        D_{\xi}W(D\tilde{u},x)=D_{\xi}W(D\tilde{u},x)-D_{\xi}W(0,x)=\int_{0}^{1}A_{\tilde{u}}^{t}D\tilde{u}\diff{t},
    \end{equation*}
    where $[A_{\tilde{u}}^t]_{ij} = [D_{\xi_{i}}D_{\xi_{j}}W(tD\tilde{u},x)]_{ij}$, $t\in[0,1]$.

    By using Proposition~\ref{prop::H-1control}, combining the above two equations, we obtain
    \begin{equation*}
        \frac{\norm{\tilde{u}-u}_{H^1(\Omega)}}{\norm{\tilde{u}}_{H^1(\Omega)}}>C\frac{\norm{(\chi^+-\chi^-)\nu_{\chi}\int_{0}^{1}A_{\tilde{u}}^{t}D\tilde{u}\diff{t}\diff{\fH^{d-1}}}_{H^{-1}(\Omega)}}{\norm{\tilde{u}}_{H^1(\Omega)}}.
    \end{equation*}
    The rest of the proof is omitted and readers' can refer to the proof of Theorem~\ref{thm::UnbddRLinEllipEq}.
\end{proof}



        
\newpage
\appendix
\section{Notation}
For the readers' convenience, we collect the frequently-used notations and list them in Table~\ref{table:ListNotation}. 
Also for notation simplicity, unless specifically claimed, we abbreviate $g(x)$ as $g$ for univariate functions. In the table, equation is the abbreviation of boundary value problem and RM is the abbreviation of residual minimization.

\begin{center}\label{table:ListNotation} 
    \renewcommand\arraystretch{1.5}
    \begin{longtable}{m{3cm}<{\centering}|m{9cm}<{\centering}|m{2cm}<{\centering}} 
        \caption{List of notations.}
        \\ \toprule[1.5pt]  
        \textbf{Notation} 
        & \textbf{Definition/Meaning} 
        & \textbf{Refer to}  
        \\ \hline
        $(Y,\norm{\cdot}_{Y})$ & Banach space equipped with norm $\norm{\cdot}_{Y}$ &
        \\ \hline
        $S^{d\times d}$ 
        & set of $d\times d$ symmetric real-valued matrices
        &
        \\ \hline 
        $a\odot b$ 
        & entry-wise multiplication, e.g., if $a,b\in \sR^d$, then $c=a\odot b$ means $c_i=a_ib_i$, $i=\{1,\ldots,d\}$
        &
        \\ \hline
        $D_{\xi}\odot D_x (F)$ 
        &  entry-wise differentiation, e.g., for $F(\xi,x)$ with $\xi,x\in \sR^{d}$, $D_{\xi}\odot D_x (F)= (D_{\xi_1}D_{x_1}F,\ldots,D_{\xi_d}D_{x_d}F)$
        &
        \\ \hline
        $f$ 
        & right-hand-side of the PDE, (interior) data
        &
        \\ \hline
        $F \circ G$ 
        & composition of functions: $(F\circ G)(x)=F(G(x))$
        &
        \\ \hline
        $D_i[F](x)$, $D_{i}F$
        & $i$-th partial derivative of $F$ evaluated at $x$;
        if no confusion, simply write $D_i F=D_i[F](x)$ 
        \\ \hline
        $DF$ 
        & $DF=(D_1F,\ldots,D_d F)^{\T}$
        \\ \hline
        $J_{\chi},\chi^{\pm},v_{\chi}$  
        & set of approximate jump points and related functions 
        & Def.~\ref{def::ApproxJump}
        \\ \hline $D^{\mathrm{j}}\chi,D^{\mathrm{a}}\chi,D^{\mathrm{c}}\chi$
        & absolutely continuous/jump/Cantor part of $D \chi$
        & Thm.~\ref{thm::StructureThmBV}
        \\ \hline
        $d,d'$
        & input/output dimension
        & \multirow{4}{*}{Sec.~\ref{sec::LinEllipEqSysBVCoeff}}
        \\ \cline{1-2}
        $\chi_{\min},\chi_{\max}$  
        & minimum/maximum of $\chi^{\alpha\beta}$
        & \multirow{4}{*}{}
        \\ \cline{1-2}
        $\bar{\lambda},\bar{\Lambda}$ 
        & lower/upper bound of $\bar{A}^{\alpha\beta}$
        & \multirow{4}{*}{}
        \\ \cline{1-2}
        $\lambda,\Lambda$ 
        & lower/upper bound of $A$
        & \multirow{4}{*}{}
        \\ \hline
        $L$, $\tilde{L}$  
        & original/modified operator
        & \multirow{3}{*}{Sec.~\ref{sec::DeriveModEq}}
        \\ \cline{1-2}
        $u$, $\tilde{u}$ 
        & solution to original/modified equation
        & \multirow{3}{*}{}
        \\ \cline{1-2}
        $u_{\theta}$,$\tilde{u}_{\theta}$ 
        & RM solution to original/modified equation 
        & \multirow{3}{*}{}
        \\ \hline
        $\mu(B;\chi,\Upsilon,\varphi)$
        &
        $\mu(B;\chi,\Upsilon,\varphi)=\int_{B\cap J_{\chi}}(\chi^+-\chi^-)\nu_{\chi}^{\T} \Upsilon[\varphi]D\varphi\diff{\fH^{d-1}}$
        & Sec.~\ref{sec::CharacterRemovSing}
        \\ \hline
        $X$ 
        & $X=\{\tilde{L}w\colon\,w\in H^1_0(\Omega;\sR^{d'})\cap H^{2}(\Omega;\sR^{d'}) \}$
        & \multirow{5}{*}{Sec. \ref{sec::RMInvSubSpace}}
        \\ \cline{1-2}
        $Tf$ 
        & RM-transformed data $Tf=L\tilde{u}$
        & \multirow{5}{*}{}
        \\ \cline{1-2}
        $T$ 
        & RM-transformation $f\mapsto L\tilde{u}$
        & \multirow{5}{*}{}
        \\ \cline{1-2}
        $\sigma(T)$ 
        & spectrum of $T$
        & \multirow{5}{*}{}
        \\ \cline{1-2}
        $\Ker(T-I)$
        & $\Ker(T-I)=\Ker_{X}(T-I)=\{f\in X:Tf=f\}$
        & \multirow{5}{*}{}
        \\ \hline
        $\theta^{\mathrm{SV}}$ 
        & parameter obtained by supervised learning with target function $u$
        & \multirow{3}{*}{Sec. \ref{sec::ImplicitBiasness}}
        \\ \cline{1-2}
        $u_{\theta}^{\mathrm{SV}}$ 
        & neural network function with parameter $\theta=\theta^{\mathrm{SV}}$
        & \multirow{3}{*}{}
        \\ \cline{1-2}
        $u_{\theta}^{\mathrm{SV}\to \mathrm{RM}}$,$\tilde{u}_{\theta}^{\mathrm{SV}\to \mathrm{RM}}$ 
        & RM solution to original/modified equation with initial parameter $\theta_0=\theta^{\mathrm{SV}}$
        & \multirow{3}{*}{}
        \\ \hline        
        $x'$ 
        & $x'=(x_1,\ldots,x_{d-1})$ for $x=(x_1,\ldots,x_d)\in \sR^d$
        & \multirow{5}{*}{Sec.~\ref{sec::ThmRemovSing}}
        \\ \cline{1-2}
        $B_r(x)$ 
        & ball centered at $x$ with radius $r$
        & \multirow{5}{*}{}
        \\ \cline{1-2}
        $J_{\chi}^{\pm}$
        & $J_{\chi}^{\pm}=\{x\in J_{\chi}:\chi^+(x)-\chi^-(x)\gtrless 0\}$
        & \multirow{5}{*}{}
        \\ \cline{1-2}
        $U \prec\zeta\prec U'$ 
        & $\zeta\geq 0$ in $\sR^d$ with $\zeta=1$ in $U$ and $\zeta=0$ in $\sR^d\backslash U'$
        & \multirow{5}{*}{}
        \\ \cline{1-2}
        $\dist(U,U')$
        & Euclidean distance between two sets $U$ and $U'$
        & \multirow{3}{*}{}
        \\ \hline 
        $\bar{X}$ &  $\bar{X}=\{\tilde{L}w\colon\,w\in H^1_0(\Omega;\sC^{d'})\cap H^{2}(\Omega;\sC^{d'}) \}$ & Sec.~\ref{sec::RMTransModEqLin}
        \\ \hline
        $\tU,\widetilde{\tU}$
        & set of solutions to original/modified equation of quasilinear equation
        & \multirow{3}{*}{Sec.~\ref{sec::equationQuasLin}}
        \\ \cline{1-2}
        $\dist_{\mathrm{H}}(\widetilde{\tU},\tU)$
        & Hausdorff distance from $\widetilde{\tU}$ to $\tU$
        & \multirow{3}{*}{}
        \\ \cline{1-2}
        $A_{v,u}^t$ 
        & $[A_{u,v}^t]_{ij}=W_{\xi_{i} \xi_{j}}\left(D u+t\left(D v-D u\right), x\right)$ and if $u=0$, for simplicity, write $A_{v,u}^t$ as $A_{v}^t$
        & \multirow{3}{*}{}
        \\ \bottomrule[1.5pt]
    \end{longtable}   
\end{center}

\section{Functions of bounded variation}\label{sec::BVFun}
Here we list several definitions and theorems used in the main text. 

\begin{definition}[\textbf{function of bounded variation} (BV), pp.117--118~\cite{ambrosio2001functions}]\label{def::BVFun}
    Let $\Omega$ be an open subset of $\sR^d$ and $\chi \in L^{1}(\Omega)$. We say that $\chi$ is a \textbf{function of bounded variation} on $\Omega$ if the distributional derivative of $\chi$ is representable by a finite Radon measure in $\Omega$, that is, if
    \begin{equation*} 
        \int_{\Omega} \chi D_{i} \varphi \diff{x}=-\int_{\Omega} \varphi d D_{i} \chi \quad \forall \varphi \in C_{c}^{\infty}(\Omega), \quad i=1, \ldots, d
    \end{equation*}
    for some $\sR^{d}$-valued measure $D \chi=\left(D_{1} \chi \ldots . D_{d} \chi\right)^\T$ in $\Omega$. The vector space of all functions of bounded variation on $\Omega$ is denoted by $BV(\Omega)$.
\end{definition}

\begin{definition}[\textbf{approximate limit}, rephrased from pp.160~\cite{ambrosio2001functions}]\label{def::ApproxLim}
    Let $\chi \in L_{\mathrm{loc}}^{1}(\Omega)$. We say that $\chi$ has an \textbf{approximate limit} at $x \in \Omega$ if there exists $z \in \sR$ such that
\begin{equation}\label{eq::ApproxLim}
    \lim _{\rho \to 0} \frac{1}{\Abs{B_{\rho}(x)}}\int_{B_{\rho}(x)}\abs{\chi(y)-z} \diff{y}=0.
\end{equation}
The set $S_{\chi}$ of points where this property does not hold is called the \textbf{approximate discontinuity set}. For any $x \in \Omega \backslash S_{\chi}$, $z$ is uniquely determined by~\eqref{eq::ApproxLim}, denoted by $\chi^{\mathrm{ap}}(x)$, and called the \textbf{approximate limit} of $\chi$ at $x$.
\end{definition}

\begin{definition}[\textbf{approximate jump points}, pp.163~\cite{ambrosio2001functions}]\label{def::ApproxJump}
     Let $\chi \in L_{\mathrm{loc}}^{1}(\Omega)$ and $x \in \Omega$. We say that $x$ is an \textbf{approximate jump point} of $\chi$ if there exist $\chi^\pm(x) \in \sR$ and a unit vector $\nu_\chi(x) \in S^{d-1}$ such that $\chi^+(x) \neq \chi^-(x)$ and
    \begin{equation}\label{eq::ApproxJump}
        \lim _{\rho \to 0} \frac{1}{B_{\rho}^\pm(x,\nu)}\int_{B_{\rho}^\pm(x, v)}\abs{\chi(y)-\chi^\pm(x)} \diff{y}=0,
    \end{equation}
    where $B^\pm_{\rho}(x,\nu)=\{y\in B_{\rho}(x)\colon\,\pm\langle x-y,\nu\rangle>0\}$.
    The triplet $\left(\chi^{+}(x), \chi^{-}(x), \nu_{\chi}(x)\right)$ is uniquely determined by~\eqref{eq::ApproxJump} up to a permutation of $(\chi^{+}(x), \chi^{-}(x))$ and a change of sign of $\nu_{\chi}(x)$.
    The \textbf{set of approximate jump points} is denoted by $J_{\chi}$.
\end{definition}

\begin{theorem}[structure theorem of BV function, rephrased from  pp.183--185 \& pp.212--213~\cite{ambrosio2001functions}]\label{thm::StructureThmBV}
    By Radon--Nikodym theorem, we have the representation $D\chi= D^{\mathrm{a}}\chi +D^{\mathrm{s}}\chi$, where $D^{\mathrm{a}}\chi$ is the absolutely continuous part with respect to $\fL^d$ and $D^{\mathrm{s}}\chi$ is the \textbf{singular part with respect to $\fL^d$}. 
    For any $\chi \in BV(\Omega)$, the \textbf{jump part of the derivative} $D^{\mathrm{j}} \chi$ and the \textbf{Cantor part of the derivative} $D^{\mathrm{c}} \chi$ are defined as
    \begin{equation*} 
    D^{\mathrm{j}} \chi=D^{\mathrm{s}} \chi \lfloor_{J_{\chi}}, \quad D^{\mathrm{c}} \chi=D^{\mathrm{s}} \chi \lfloor_{\Omega \backslash S_{\chi}},
    \end{equation*}
    respectively.  Then we have 
    \begin{equation*} 
        D\chi=D^{\mathrm{a}}\chi+D^{\mathrm{j}}\chi+D^{\mathrm{c}}\chi,
    \end{equation*}
    where $D^{\mathrm{a}} \chi=\nabla \chi \fL^d$, $\nabla \chi$ is the Radon--Nykodim density of $D^{\mathrm{a}}\chi$ with respect to $\fL^d$ and $D^{\mathrm{j}} \chi= \left(\chi^{+}-\chi^{-}\right) \nu_{\chi}^{\T} \fH^{d-1}\lfloor_{J_{\chi}}$.
\end{theorem}

\begin{theorem}[Structure theorem of \textbf{approximate jump set} for BV functions, pp.150~\cite{lin2002geometric}]\label{thm::PropOfApproxJumpSetForBV}
   Let $\Omega$ be a bounded open set and $\chi\in BV(\Omega)$, then there exist countably many $C^1$-hypersurfaces  ($(d-1)$-dimensional $C^1$-manifolds)  $\left\{S_k\right\}_{k=1}^{\infty}$ such that
    \begin{equation*}
        \fH^{d-1}\left(J_{\chi}\backslash(\bigcup\nolimits_{k=1}^{\infty}S_k)\right)=0    
    \end{equation*} 
\end{theorem}

\begin{definition}[\textbf{special function of bounded variation} (SBV), pp.212~\cite{ambrosio2001functions}]\label{def::SBVFun}
    We say that $\chi \in BV(\Omega)$ is a \textbf{special function of bounded variation} and we write $\chi \in SBV(\Omega)$, if the Cantor part of its derivative $D^{c} \chi$ is zero. Thus, for $\chi \in SBV(\Omega)$, we have
        \begin{equation*} 
            D \chi=D^{\mathrm{a}} \chi+D^{\mathrm{j}} \chi=\nabla \chi \fL^d+\left(\chi^{+}-\chi^{-}\right) \nu_{\chi} \fH^{d-1}\lfloor_{J_{\chi}}.        
        \end{equation*}
        We further say that $\chi\in SBV^p(\Omega)$, $p\in [1,\infty]$ if and only if $\chi\in SBV(\Omega)$ and $\nabla \chi \in L^p(\Omega)$.        
\end{definition}

\begin{corollary}[property of \textbf{special funtion of bounded variation}, pp.212~\cite{ambrosio2001functions}]\label{cor::PropOfSBV}
    For any $\chi \in SBV(\Omega)$, we have 
        \begin{equation*} 
            \chi \in W^{1,1}(\Omega) \quad \Longleftrightarrow \quad \fH^{d-1}\left(J_{\chi}\right)=0.
        \end{equation*}
\end{corollary}

\section{Existence theorems for linear elliptic equations and systems in literature}\label{sec::ExistThmLinProb}
In this appendix, we rephrase some existence theorems for linear elliptic equation and system from the literature. Theorems~\ref{thm::ExistThmLinProbDivForm} and \ref{thm::ExistThmLinProbNonDivForm} focus on linear divergence and nondivergence equation, respectively, while Theorem~\ref{thm::ExistThmLinSysProbDivForm} works for linear divergence system.

\begin{theorem}[Theorem 8.3, rephrased from pp.181--182~\cite{gilbarg2001elliptic}]\label{thm::ExistThmLinProbDivForm}
    Let $\Omega$ be a bounded $C^{1,1}$ domain in $\sR^{d}$, and let the operator $L$
    \begin{equation*}
        Lu = -\operatorname{div}\cdot (ADu)=-\sum_{i,j=1}^{d}D_i(a_{ij}D_j u)
    \end{equation*}
    being uniformly elliptic in $\Omega$ with coefficients $a_{ij}\in L^{\infty}(\Omega)$, namely there is $\lambda>0$ satisfying $\xi^{\T}A\xi\geq\lambda\abs{\xi'}^2$ for all $\xi\in \sR^d$. 
    Then the equation: $L u=f$ in $\Omega, u=0$ on $\partial\Omega$ with $f\in L^2(\Omega)$ (or $f\in H^{-1}(\Omega)$) is uniquely solvable. Furthermore, there is a constant $C$ independent of $u$ such that
    \begin{equation*}
        \norm{u}_{H^1(\Omega)}\le C\norm{f}_{L^2(\Omega)} \text{ or } (\le C\norm{f}_{H^{-1}(\Omega)}).    
    \end{equation*}
\end{theorem}

\begin{theorem}[Lemma 9.17, rephrased from pp.241--243~\cite{gilbarg2001elliptic}]\label{thm::ExistThmLinProbNonDivForm}
    Let $\Omega$ be a bounded $C^{1,1}$ domain in $\sR^{d}$, and the operator $\tilde{L}$
    \begin{equation*}
        \tilde{L}u=\sum_{i,j=1}^da_{ij}D_{ij}u+\sum_{i=1}^{d}b_iD_{i}u+cu,
    \end{equation*}
    being uniformly elliptic in $\Omega$ with coefficients $a_{ij} \in C(\bar{\Omega}), b_i ,c \in L^{\infty}(\Omega)$, $c\le 0$, with $i, j\in\{1,\ldots,d\}$. If $f \in L^{p}(\Omega)$ for some $p\in (1,\infty)$, then the equation: $\tilde{L} u=f$ in $\Omega$, $u=0$ on $\partial\Omega$ has a unique solution $u \in W_0^{1,p}(\Omega)\cap W^{2, p}(\Omega)$. Furthermore, there is a constant $C$ independent of $u$ such that 
    \begin{equation*} 
        \norm{u}_{W^{2,p}(\Omega)}\leq C\norm{\tilde{L}u}_{L^p(\Omega)}.
    \end{equation*}
\end{theorem}

\begin{theorem}[Theorem 1.3, rephrased from pp.9--14~\cite{ambrosio2018lectures}]\label{thm::ExistThmLinSysProbDivForm}
    Let $\Omega$ be a bounded $C^{1,1}$ domain in $\sR^{d}$. Also define the operator $L$ as follows
    \begin{equation*}
     (L u)^{\alpha} = -\sum_{\beta=1}^{d'}\divg\cdot(A^{\alpha\beta}(x)Du^{\beta})
     =-\sum_{\beta=1}^{d'}\sum_{i,j=1}^{d}D_{i}(A_{ij}^{\alpha\beta}D_{j}u^{\beta}),
    \end{equation*}
    where $\alpha\in \{1,2,\ldots,d'\}$. Suppose that measurable functions $A^{\alpha\beta}\in S^{d\times d}$ are symmetric in $\alpha,\beta$, namely $A^{\alpha\beta}=A^{\beta\alpha}$, and that the Hadamard--Legendre condition holds, namely there is $\lambda>0$ satisfying $ \sum_{\alpha,\beta=1}^{d'}(\xi^{\alpha})^\T A^{\alpha\beta}(x)\xi^{\beta} \geq \lambda\abs{\xi}^{2}$ for all $\xi^{\alpha}\in \sR^d$, where $\abs{\xi}^2=\sum_{\alpha=1}^{d'}\abs{\xi^{\alpha}}^2$. 
    Then the equation: $L u=f$ in $\Omega, u=0$ on $\partial\Omega$ with $f\in L^2(\Omega;\sR^{d'})$ (or $f\in H^{-1}(\Omega;\sR^{d'})$) is uniquely solvable. Moreover there is a constant $C$ independent of $u$ such that
    \begin{equation*}
        \norm{u}_{H^1(\Omega;\sR^{d'})}\le C\norm{f}_{L^2(\Omega;\sR^{d'})}\quad  (\text{or } \le C\norm{f}_{H^{-1}(\Omega;\sR^{d'})}).    
    \end{equation*}
\end{theorem}

\section{Existence theorems for quasilinear elliptic equations in literature}\label{sec::ExisThm4QuasLinInLitera}
In this appendix, we collect some well-known results on the existence of solutions to the quasilinear elliptic equations from the literature. In particular, we summarize these results for elliptic problems in both divergence form and non-divergence form.

\subsection{Divergence form}

Suppose now $\Omega \subseteq \sR^{d}$ is a bounded, open set with $C^{1,1}$ boundary $\partial\Omega$, and we are given a Lagrangian
\begin{equation*} 
    L: \sR^{d} \times \sR \times \bar{\Omega} \to \sR.
\end{equation*}
We call $L=L(\xi,z,x)$ the Lagrangian and define the functional: 
\begin{equation*} 
    I[w]=\int_{\Omega} L(D w(x), w(x), x) \diff{x}.
\end{equation*} 

Consider the Euler--Lagrange equation:
\begin{equation}\label{eq::AppQuasLinDiv}
    \left\{
    \begin{aligned}
        -\divg\cdot\left(D_{\xi} L(D u, u, x)\right)+D_{z}L(D u, u, x) &=0 & & \text {in}\  \Omega, \\
        u &=0 & & \text {on}\  \partial\Omega,
    \end{aligned}
    \right.
\end{equation}

To guarantee the existence and uniqueness of solution, we assume the operator $L$ taking the following properties:
\begin{assumption}\label{assum::QuasLinDiv} Assume that the following hold.
    \begin{enumerate}[(i)]
    \item (coercivity for $L$) 
    There exist constants $c'_1>0$ and $c'_2>0$ such that for all $\xi\in\sR^{d}$, $z\in\sR$, $x\in\Omega$
    \begin{equation*} 
        L(\xi,z,x)\geq c'_1\Abs{\xi}^p-c'_2.
    \end{equation*}
    \item 
    (convexity in $\xi$) For all $\xi,\xi'\in \sR^d$, we have 
    \begin{equation*}  
            (\xi')^\T D_\xi^2L(\xi, x) \xi' \geq 0.
        \end{equation*}
    
    \item (growth condition of $L$) There exists $c'_3>0$ such that for all $\xi \in \sR^{d}, \ z \in \sR, \ x \in \Omega$
    \begin{align*} 
        \abs{L(\xi, z, x)} 
        &\leq c'_3\left(\abs{\xi}^{p}+\abs{z}^{p}+1\right),\\
        \Abs{D_{\xi} L(\xi, z, x)} 
        &\leq c'_3\left(\abs{\xi}^{p-1}+\abs{z}^{p-1}+1\right) ,\\
        \Abs{D_{z} L(\xi, z, x)} 
        &\leq c'_3\left(\abs{\xi}^{p-1}+\abs{z}^{p-1}+1\right).
    \end{align*}
    \end{enumerate}    
\end{assumption}

\begin{remark}
~
    \begin{enumerate}[(i)]
        \item The coercivity assumption will lead to the coervicity condition on $I[\cdot]$
        \begin{equation*} 
            I[w] \geq c'_1\norm{D w}_{L^{p}(\Omega)}^{p}-c'_2.
        \end{equation*}
        \item For the convexity assumption, we say $L$ is uniformly convex in the variable $\xi$, if $L=L(\xi, x)$ does not depend on $z$
        and there is $\lambda>0$ such that for all $\xi, \xi'\in \sR^d$ and $x\in \Omega$
        \begin{equation*}  
            (\xi')^\T D_\xi^2L_{\xi_{i} \xi_{j}}(\xi, x) \xi' \geq \lambda\abs{\xi'}^{2} .
        \end{equation*}
        
    \end{enumerate}
\end{remark}

\begin{theorem}[Theorem 7, pp.454--455~\cite{evans2010partial}] Suppose that Assumption~\ref{assum::QuasLinDiv} holds. Then there is a $u\in W^{1,p}_0$, $p\geq 2$, which minimizes $I[\cdot]$, and is the weak solution to~\eqref{eq::AppQuasLinDiv}. 
\end{theorem}

\subsection{Non-divergence form}\label{sec::QuasLinNonDiv}

\begin{definition}[Carathéodory function]\label{def::Cara}
Let $g:\sR^q\times\Omega  \to \sR$, $\Omega\subseteq \sR^d$. Then $g$ is a \textbf{Carathéodory function} if
\begin{enumerate}[(i)]
    \item $x\mapsto g(\xi,x)$ is $\fL^d$ measurable on $\Omega$ for any $\xi\in \sR^q$,
    \item $\xi\mapsto g(\xi,x)$ is continuous on $\sR^q$ for $\fL^d$-a.e. $x\in\Omega$.
\end{enumerate}
\end{definition}

  Consider the general quasilinear equation
    \begin{equation}\label{eq::AppQuasLinNonDiv} 
        \left\{\begin{aligned}
        Q[u] = A_{ij}(Du, u, x) D_{ij} u+b(Du, u, x) &=0 & & \text {in}\ \Omega, \\
        u &=0 & & \text{on}\  \partial\Omega,
        \end{aligned}\right.
    \end{equation}
Suppose $A_{ij}:\sR^{d} \times \sR\times \Omega  \to \sR$ are $C^{1}$ functions, $b: \sR^{d}\times \sR \times \Omega  \to \sR$ is a Carathéodory function and

\begin{assumption}\label{assum::QuasLinNonDiv}
~
    \begin{enumerate}[(i)]
    \item (quadratic gradient growth of $b$) There exists $p>d$, $b_{1} \in L^{p}(\Omega)$ and a non-decreasing function $v:[0, \infty) \to(0, \infty)$ such that
    \begin{equation*} 
        \abs{b(\xi, z, x)} \leq v(\abs{z})\left(b_{1}(x)+\abs{\xi}^{2}\right)
    \end{equation*}
    for $\fL^d$-a.e. $x \in \Omega, \forall(\xi, z) \in \sR^d \times \sR$.
    \item (monotonicity of $b$ with respect to $u$) There exists a non-negative function $b_{2} \in L^{d}(\Omega)$ such that
    \begin{equation*} 
        \operatorname{sign} z \cdot b(\xi, z, x) \leq \lambda(\abs{z}) b_{2}(x)(1+\abs{\xi})
    \end{equation*}
    for $\fL^d$-a.e. $x \in \Omega$ and for all $(\xi,z) \in \sR^{d} \times \sR$.
    \item (uniform ellipticity of $Q$) There exists a non-increasing function $\lambda$ : $[0,+\infty)\to (0,+\infty)$ such that for $\fL^d$-a.e. $x \in \Omega$, for all $z \in \sR$ and for all $\xi, \xi' \in \sR^{d}$
    \begin{equation*} 
        \lambda(\abs{z})\abs{\xi'}^{2} \leq (\xi')^{\T} A(\xi, z, x) \xi' \leq \frac{1}{\lambda(\abs{z})}\abs{\xi'}^{2}.
    \end{equation*}
    \item (growth conditions of $A_{ij}$) There exist $p>d, b_3 \in L^{p}(\Omega)$ and a non-decreasing function $\eta:[0,+\infty)\to (0,+\infty)$ such that for all $(\xi, z, x) \in$ $\sR^{d} \times \sR \times \Omega$
    \begin{align*}
        \Abs{D_{z} A_{ij}(\xi, z, x)}+\Abs{D_{x_{k}} A_{ij}(\xi, z, x)} 
        & \leq \eta(\abs{z}+\abs{\xi}) b_3(x) \\
        \Abs{D_{\xi_{k}} A_{ij}(\xi, z, x)} & \leq \eta(\abs{z}+\abs{\xi}) \\
        \Abs{D_{\xi_{k}} A_{ij}(\xi, z, x)-D_{\xi_{j}} A_{ik}(\xi, z, x)} 
        & \leq \eta(\abs{z})\left(1+\abs{\xi}^{2}\right)^{-1 / 2},
    \end{align*}
    \begin{align*}
        &\Abs{\sum_{k=1}^{d}\left(D_{z} A_{ij}(\xi, z, x) \xi_{k} \xi_{k}-D_{z} A_{kj}(\xi, z, x) \xi_{k} \xi_{i}
        +D_{x_{k}} A_{ij}(\xi, z, x) \xi_{k}-D_{x_{k}} A_{kj}(\xi, z, x) \xi_{i}\right)}\\
        &\leq \eta(\abs{z})\left(1+\abs{\xi}^{2}\right)^{1 / 2}(\abs{\xi}+b_3(x)).
    \end{align*}
    \item (local uniform continuity of $b$ with respect to $(\xi, z)$) There exists $p>d$ such that $b(\xi, z, \cdot) \in L^{p}(\Omega)$ for all $(\xi, z) \in \sR^d \times \sR$, and for all $M, \eps>0$ there exists $\delta>0$ such that
    \begin{equation*} 
        \int_{\Omega}\Abs{b(\xi, z, x)-b\left(\xi', z', x\right)}^{p} \diff{x} <\eps    
    \end{equation*}
    for $\fL^d$-a.e. $x \in \Omega$ and all $(\xi, z),\left(\xi', z'\right) \in \sR^d \times \sR$ with $\Abs{z-z'}+\Abs{\xi-\xi'}<\delta$ and $\Abs{z},\Abs{z'},\Abs{\xi},\Abs{\xi'} \leq M$.
    \end{enumerate}
\end{assumption}

\begin{theorem}[Theorem 3~\cite{palagachev2006applications}]\label{thm::QuasLinNonDiv}
    Suppose that Assumption~\ref{assum::QuasLinNonDiv} holds. Then there exists a solution $u \in W^{1, p}_{0}(\Omega)\cap W^{2, p}(\Omega)  $ of the problem~\eqref{eq::AppQuasLinNonDiv}.   
\end{theorem}

\section{A short proof of the failure of PINN in $1$-d}

To fix ideas and to understand the failure example of PINN given in Section~\ref{sec::FailExPINN}, we provide in this appendix a succinct explanation to the failure phenomenon with one-dimensional setting $\Omega=(-1,1)$. More general results and the complete analysis for the case of higher dimensional problem and even for quasilinear elliptic equations are given in the main text.

\begin{assumption}[uniform ellipticity condition and piece-wise continuous coefficients]\label{assum::UniEllipPWContCoeff}
    Assume there are constants $0<\lambda<\Lambda$ satisfying $\lambda\leq A(x)\leq \Lambda$ for all $x\in\Omega$ and
    \begin{equation*}
        A(x) = \bar{a}(x) + \sum_{k=1}^{n_{\mathrm{jump}}} a_k H(x-j_k),
    \end{equation*} 
    where $\bar{a}$ is an absolutely continuous function on $\Omega$, weights $a_k\in\sR$, and discontinuities $\{j_k\}_{k=1}^{n_{\mathrm{jump}}}\subseteq\Omega$ for $n_{\mathrm{jump}}\in \sN^+$. Here $H$ is the Heaviside step function.
\end{assumption}

As the general case in the main text, we introduce the modified problem:
\begin{equation}\label{eq::ModEq1D}
    \left\{
    \begin{aligned}
    -(\bar{a} D^2_{x}u+(D_{x}\bar{a}) D_{x} u)&=f & & \text{in}\  \Omega, \\
    u(\pm 1) &=0.
    \end{aligned}
    \right.
\end{equation}
and denote its solution as $\tilde{u}\in H^1_0(\Omega)\cap H^{2}(\Omega)$. The existence and regularity are guaranteed by Theorem~\ref{thm::ExistAprioEstLinEllipEq&Sys} in the general setting. 

The following proposition shows that the original operator $L$ maps $\tilde{u}$ to the RM-transformed data $f -\sum_{k=1}^{n_{\mathrm{jump}}}a_kD_{x}\tilde{u}(j_k)\delta_{j_k}$. Thus it is clear that the latter equals $f$ if and only if $D_{x}\tilde{u}$ vanishes at all jump points $j_k$.

\begin{proposition}[representation by RM-transformed data for 1-d linear elliptic equation]\label{prop::RepRMTransData1DEllipEq}
    Suppose that Assumption~\ref{assum::UniEllipPWContCoeff} holds and $f\in L^2$. Let $\tilde{u}$ be the solution to problem~\eqref{eq::ModEq1D}. Then $\tilde{u}$ is the weak solution to 
    \begin{equation}
        \left\{
        \begin{aligned}
            L\tilde{u} &= f -\sum_{k=1}^{n_{\text{jump}}}a_k D_{x}\tilde{u}(j_k)\delta_{j_k} & & \text{in}\ \Omega, \\
            \tilde{u}(\pm 1) &=0,
        \end{aligned}
        \right.
    \end{equation}
    where $\delta_{j_k}$ is the Dirac measure satisfying $\langle\delta_{j_k},\varphi\rangle=\delta_{j_k}(\varphi)=\varphi(j_k)$ for any $\varphi \in H^1_0(\Omega)$.
\end{proposition}

\begin{proof} 
    Without loss of generality, we assume that $j_k<j_{k'}$ if $1\leq k< k'\leq n_{\mathrm{jump}}$, and set $j_0=-1$, $j_{n_{\mathrm{jump}}+1}=1$. The solution $\tilde{u}\in H^{2}(\Omega)$ implies that $D_{x}\tilde{u}\in C(\Omega)$ and $AD_{x}\tilde{u} \in H^1((j_k,j_{k+1}))$ for $k=0,\ldots,n_{\mathrm{jump}}$. Using Lebesgue integral theorem, we have
    \begin{equation*}
        \int_{j_k}^{j_{k+1}}D_{x}(AD_{x}\tilde{u})\varphi \diff{x}+\int_{j_k}^{j_{k+1}}AD_{x}\tilde{u}D_{x}\varphi\diff{x} = AD_{x}\tilde{u} \varphi\Big|_{j_k^+}^{j_{k+1}^-},\quad k=0,\ldots,n_{\mathrm{jump}},
    \end{equation*}
    where 
    \begin{equation*}
        AD_{x}\tilde{u} \varphi\Big|_{j_k^+}^{j_{k+1}^-}=\lim_{\eps\to 0}[A(j_{k+1}-\eps)D_{x}\tilde{u}(j_{k+1}-\eps) \varphi(j_{k+1}-\eps)-A(j_{k}+\eps)D_{x}\tilde{u}(j_{k}+\eps) \varphi(j_{k}+\eps)].
    \end{equation*}
    Since there is no jump discontinuity of $A(x)$ for $x\in \Omega\backslash \{j_k\}_{k=1}^{n_{\mathrm{jump}}}$, we have $-D_{x}(AD_{x}\tilde{u}) = f$ on $(j_k,j_{k+1})$ and
    \begin{equation*}
        \int_{j_k}^{j_{k+1}}AD_{x}\tilde{u}D_{x}\varphi\diff{x} = \int_{j_k}^{j_{k+1}}f\varphi\diff{x} + AD_{x}\tilde{u}\varphi\Big|_{j_k}^{j_{k+1}},\quad \forall \varphi \in H^1_0(\Omega).
    \end{equation*}
    Thus the proof is completed by the following calculation
    \begin{align*}
        \int_{-1}^{1}(AD_{x}\tilde{u})D_{x}\varphi\diff{x} =\sum_{k=0}^{n_{\mathrm{jump}}}\int_{j_k}^{j_{k+1}}AD_{x}\tilde{u}D_{x}\varphi\diff{x}
        &= \sum_{k=0}^{n_{\mathrm{jump}}}\int_{j_k}^{j_{k+1}}f\varphi\diff{x} + AD_{x}\tilde{u}\varphi\Big|_{j_k^+}^{j_{k+1}^-} \nonumber\\
        &= \int_{-1}^{1}f\varphi\diff{x} -\sum_{k=1}^{n_{\mathrm{jump}}}a_kD_{x}\tilde{u}(j_k)\left\langle  \delta_{j_k} , \varphi \right\rangle.
    \end{align*}
\end{proof}

Next we describe the gap between $u$ and $\tilde{u}$ in terms of the $L^2$ norm. As in the main text, we will show this is non-zero.

\begin{proposition}[deviation occurs for 1-d linear elliptic equation]\label{prop::WorstCaseDev1DLinEllipEq}
    Suppose that Assumption~\ref{assum::UniEllipPWContCoeff} holds and $f\in L^2(\Omega)$. Let $u$ and $\tilde{u}$ be the solution to the problem~\eqref{eq::1dEq} and~\eqref{eq::ModEq1D}, respectively. If there exists $k\in \{1,\ldots,n_{\mathrm{jump}}\}$ such that $D_{x}\tilde{u}(j_k) \neq 0$, then we have that: 
    \begin{equation*}
        \norm{u -\tilde{u} }_{L^2(\Omega)}>0.    
    \end{equation*}
\end{proposition}
The existence and regularity are guaranteed by Theorem~\ref{thm::ExistThmLinProbDivForm} in a more general setting.
\begin{proof}
    By Theorem~\ref{prop::RepRMTransData1DEllipEq}, we have
    \begin{equation*}
        L(u - \tilde{u} ) = f-L\tilde{u} = \sum_{k=1}^{n_{\mathrm{jump}}}a_k D_{x}\tilde{u}(j_k)\delta_{j_k}.
    \end{equation*}
    Without loss of generality, we can choose some function $\varphi\in C_c^{\infty}(\Omega)$ such that $\varphi(j_k) = 1$ and $\varphi(a_{k'}) = 0$ for $k'\in \{1,\ldots,n_{\mathrm{jump}}\}\backslash \{k\}$ and $\norm{\varphi}_{H^1(\Omega)}>0$.
    
    Since there are $0<\lambda<\Lambda$ such that $\lambda \leq A(x) \leq \Lambda$ for all $x\in\Omega$, we have
    \begin{align*}
        \Lambda\norm{D_{x} u -D_{x}\tilde{u}}_{L^2(\Omega)}\norm{\varphi}_{H^1(\Omega)}
        \geq& \Abs{\int_{-1}^{1}(AD_{x}(u-\tilde{u}))D_{x}\varphi\diff{x}}\\
        =& \Abs{\int_{-1}^{1}L(u-\tilde{u})\varphi\diff{x}}
        = \Abs{a_kD_{x}\tilde{u}(j_k)}>0.
    \end{align*}
    This implies $\norm{u -\tilde{u} }_{L^2(\Omega)}>0$.
\end{proof}

\section{Proof of Theorems~\ref{thm::ExistQuasiLinDivForm} and~\ref{thm::ExistQuasiLinNonDivForm}}\label{sec::ExisProof4QuasEq}

\begin{proof}[Proof of Theorem~\ref{thm::ExistQuasiLinDivForm}]
    We use the calculus of variation and the proof is mainly based on Evans~\cite{evans2010partial} pp. 451--454.
    
    (\romannumeral1). First, we prove there is a minimizer $u\in H^1_0(\Omega)$ of the energy functional $I[u]$ defined as 
    \begin{equation*} 
        I[u] = \inf_{w\in H^1_0(\Omega)}I[w].
    \end{equation*}
    
    Set $m = \inf_{w\in H^1_0(\Omega)}I[w]$. If $m=+\infty$, we are done. We henceforth consider that $m$ is finite. Choosing a sequence $\{u_k\}_{k=1}^{\infty}$ such that: 
    \begin{equation*} 
        I[u_k]\to m.
    \end{equation*}
    
    By the coercivity condition, we have 
    \begin{equation*} 
        \begin{aligned}
            I[u_k] &= \int_{\Omega}\chi(x)W(Du_k,x)-fw\diff{x} \\ &\geq \chi_{\min}(c_1\norm{Du_k}^{2}_{L^2(\Omega)}-c_2\Abs{\Omega})-\frac{1}{2}\norm{u_k}_{L^2(\Omega)}^2-\frac{1}{2}\norm{f}_{L^2(\Omega)}^2\\ 
            &\geq \left(\chi_{\min}c_1-\frac{1}{2}\sup_{w\in H^1_0(\Omega)}\frac{\norm{w}_{L^2(\Omega)}}{\norm{Dw}_{L^2(\Omega)}}\right)\norm{Du_k}_{L^2(\Omega)}^{2}-\chi_{\min}c_2\Abs{\Omega}-\frac{1}{2}\norm{f}_{L^2(\Omega)}^2.
        \end{aligned}
    \end{equation*}
    Recall that $\chi_{\min}c_1-\frac{1}{2}\sup_{w\in H^1_0(\Omega)}\frac{\norm{w}_{L^2(\Omega)}}{\norm{Dw}_{L^2(\Omega)}}>0$ by Assumption \ref{assum::CoerCvxBddQuasLin}.
    Since $m$ is finite, we conclude that 
    \begin{equation*} 
        \sup_{k} \norm{Du_k}_{L^2(\Omega)}<+\infty.
    \end{equation*}
    By Poincare inequality, we have $\norm{u_k}_{L^2(\Omega)}\leq C\norm{Du_k}_{L^2(\Omega)}<+\infty$, and hence $\{u_k\}_{k=1}^{\infty}$ is bounded in $H^1(\Omega)$. Thus there is a subsequence $\{u_{k_j}\}_{j=1}^{\infty}$ such that 
    \begin{equation*} 
        u_{k_j}\rightharpoonup u\ \text{weakly in $H^1(\Omega)$}.
    \end{equation*}
    Notice that $u_k\in H^1_0(\Omega)$ and $H^1_0(\Omega)$ a closed linear subspace of $H^1(\Omega)$, hence, by Mazur's Theorem, is weakly closed. Hence, $u\in H^1_0(\Omega)$, which means $u=0$ on $\partial\Omega$ in the sense of trace. The existence of the minimizer is established.
    
    (\romannumeral2). By the way, we also have the uniqueness of minimizer due to the uniform convexity in Assumption~\ref{assum::CoerCvxBddQuasLin}. We omit the details since it is exactly the same as the one in Evans~\cite{evans2010partial} pp. 451--454.
    
    (\romannumeral3). Finally, we claim that the minimizer is indeed a weak solution to the equation \eqref{eq::QuasLinModEq}.
    
    For any $\varphi\in H^1_0(\Omega)$, define 
    \begin{equation*} 
        J(t)= I[u+t \varphi],\ (t\in \sR).
    \end{equation*}
    The difference quotient reads as
    \begin{equation*} 
        \begin{aligned}
            \frac{J(t)-J(0)}{t}&= \frac{1}{t}\int_{\Omega}\chi(x)(W(Du+t D\varphi,x)-W(Du,x))-ft \varphi\diff{x}\\ &=\int_{\Omega}W^{t}(x)\diff{x}-\int_{\Omega}f\varphi \diff{x},
        \end{aligned}
    \end{equation*}
    where 
    \begin{equation*} 
        W^{t}(x)=\frac{1}{t}\chi(x)(W(Du+t D\varphi,x)-W(Du,x)).
    \end{equation*}
    Thus 
    \begin{equation*} 
        W^{t}(x) \to \sum_{i=1}^{d}\chi(x)D_{\xi_{i}}W(Du,x)\varphi_{x_{i}}
    \end{equation*} 
    for $\fL^d$-a.e. $x\in \Omega$ as $t\to 0$. Furthermore:
    \begin{equation*} 
        \begin{aligned}
            W^{t}(x)=\frac{1}{t}\int_{0}^{t}\frac{\D}{\D t' }\chi(x)W(Du+t'D\varphi,x)\diff{t'}=\frac{1}{t}\int_{0}^{t}\sum_{i=1}^{d}\chi(x)D_{\xi_{i}}W(Du+t'D\varphi,x)\varphi_{x_{i}}\diff{t'}.
        \end{aligned}
    \end{equation*}
    By the growth condition on $D_{\xi_{i}}W$ and Cauchy--Schwartz inequality, we then have for all $u, \varphi\in H^1(\Omega)$
    \begin{equation*} 
        W^{t}(x)\leq C\norm{\chi}_{L^{\infty}(\Omega)}(\Abs{Du}^2+\Abs{D\varphi}^2+1),
    \end{equation*}
    where $C$ depends on $c_3$.
    Thus by dominated convergence theorem, we have
    \begin{equation*} 
        \frac{\D}{\D t}J(0) = \int_{\Omega}\sum_{i=1}^{d}\chi D_{\xi_{i}}W(Du,x)v_{x_{i}}-fv\diff{x}.
    \end{equation*}
    By our previous discussion on existence of minimizer and the existence of $\frac{\D}{\D t}J(0)$, we hence have $\frac{\D}{\D t}J(0)=0$. Therefore, $u$ is a weak solution to the equation \eqref{eq::QuasLinModEq}.
\end{proof}

\begin{remark}
    Notice that this proof still holds even if $f\in H^{-1}(\Omega)$, since we have
    \begin{equation*} 
        \langle f,\varphi\rangle_{H^{-1}(\Omega),H^1_0(\Omega)} \leq \norm{f}_{H^{-1}(\Omega)}\norm{\varphi}_{H^1(\Omega)}
    \end{equation*}
    for any $\varphi\in H^1_0(\Omega)$. Then the statement as well as the proof works for the energy functional defined by
    \begin{equation*} 
        I[w] = \int_{\Omega}\chi(x)W(Dw,x)\diff{x} - \langle f,w\rangle_{H^{-1}(\Omega),H^1_0(\Omega)}, \ \text{for $w\in H^1_0(\Omega)$}.
    \end{equation*}
\end{remark}

\begin{prop}[dependence on data in $L^2(\Omega)$]\label{prop::L2control}
    Suppose that Assumption~\ref{assum::CoerCvxBddQuasLin} holds. For any $f, f'\in L^2(\Omega)$, let $u,u' \in H^1_0(\Omega)$ be the corresponding solutions to the equation~\eqref{eq::QuasiLinOriginEq}. Then we have
    \begin{equation*} 
        \norm{u-u'}_{H^1(\Omega)} \leq C\norm{f-f'}_{L^2(\Omega)},
    \end{equation*}
    where $C$ depends on $\Omega,\chi$ and $\lambda$.
\end{prop}
This means the solution $u$ is continuous with respect to the interior data $f$.
\begin{proof}
    Using the Newton--Leibniz theorem, we have for $i\in\{1,\ldots,d\}$
    \begin{align*}
        & D_{\xi_{i}}W\left(D u, x\right)=D_{\xi_{i}}W\left(D u', x\right)+\int_{0}^{1} \frac{d}{d t} D_{\xi_i}W\left(D u'+t\left(D u-D u', x\right) d t\right.\nonumber\\
        =& D_{\xi_{i}}W\left(D u', x\right)+\int_{0}^{1} \sum_{j=1}^{d} D_{\xi_{i}}D_{\xi_{j}}W\left(D u'+t\left(D u-D u'\right), x\right)\left(D_{j} u-D_{j} u'\right)\diff{t}.
    \end{align*}
    Thus
    \begin{equation}\label{eq::tay}
        \begin{aligned}
            -\divg\cdot\left[\chi(x)D_{\xi}W(Du,x)-\chi(x)D_{\xi}W(D u',x)\right] &= f-f', \\
            \Rightarrow 
            -\divg\cdot\left[\chi \int_{0}^{1}A_{u,u'}^t D(u-u')\diff{t}\right] &= f-f', 
        \end{aligned}
    \end{equation}
    where $A_{u,u'}^t$ is a matrix-valued function with entries $(A_{u,u'}^t)_{ij}= D_{\xi_{i}}D_{\xi_{j}}W\left(D u'+t\left(D u-D u'\right), x\right)$.
    Next we use the property in Assumption~\ref{assum::BddLipGrowQuasiLinNonDiv}, that is the uniform ellipticity of $L$. By multiplying $u-u'$ and integrating on both sides, we have
    \begin{equation*} 
        \norm{u-u'}_{H^1(\Omega)} \leq C\norm{f-f'}_{L^2(\Omega)},
    \end{equation*}
    where $C$ depends on $\Omega,\chi$ and $\lambda$.
\end{proof}

More generally, for the $H^{-1}(\Omega)$ data $f$, we provide a similar result.
\begin{prop}(dependence on data in $H^{-1}(\Omega)$)\label{prop::H-1control}
    Suppose that Assumption~\ref{assum::CoerCvxBddQuasLin} holds. For any $f, f'\in H^{-1}(\Omega)$, let $u,u' \in H^1_0(\Omega)$ be the corresponding solutions to the equation~\eqref{eq::QuasiLinOriginEq}. Then we have
    \begin{equation*} 
        \frac{1}{C}\norm{f-f'}_{H^{-1}(\Omega)}\le\norm{u-u'}_{H^1(\Omega)}\leq C\norm{f-f'}_{H^{-1}(\Omega)},
    \end{equation*}
    where $C$ depends on $\Omega,\chi$, $\lambda$ and $\Lambda$.
\end{prop}

\begin{proof}
    First, similar to the proof of Proposition \ref{prop::L2control}, we have 
    \begin{equation*} 
        \langle f-g, u-u' \rangle_{H^{-1}(\Omega),H^1_0(\Omega)} = \int_{\Omega}\chi(x)D(u-u') \int_{0}^{1}A_{u,u'}^t D(u-u')\diff{t}\diff{x} \geq C\norm{u-u'}^2_{H^1(\Omega)},
    \end{equation*}
    where the last inequality results from $\chi_{\min}(x)>0$ and the uniform ellipticity of $A_{u,u'}^t$ for all $t\in [0,1]$. Here the constant $C$ depends on $\Omega,\chi$ and $\lambda$.
    
    Next we obtain the inequality in the other direction. We have for any $\varphi \in H^1_0(\Omega)$ with $\norm{\varphi}_{H^1_0(\Omega)}=1$ 
    \begin{equation*} 
        \langle f-f', \varphi \rangle_{H^{-1}(\Omega),H^1_0(\Omega)} = \int_{\Omega}\chi(x)D\varphi \int_{0}^{1}A_{u,u'}^t D(u-u')\diff{t}\diff{x} \leq C\norm{u-u'}_{H^1(\Omega)}\norm{\varphi}_{H^1(\Omega)},
    \end{equation*}
    where the last inequality results from Cauchy--Schwartz inequality and uniform ellipticity. In particular, the constant $C$ depends on $\Omega,\chi$ and $\Lambda$.
\end{proof}

\begin{proof}[Proof of Theorem~\ref{thm::ExistQuasiLinNonDivForm}]
    According to Theorem~\ref{thm::QuasLinNonDiv}, it is sufficient to show that Assumptions~\ref{assum::BvCoeQuasLin},~\ref{assum::CoerCvxBddQuasLin} and~\ref{assum::BddLipGrowQuasiLinNonDiv} together imply Assumption~\ref{assum::QuasLinNonDiv}.
        
    equation~\eqref{eq::QuasLinModEq} is equivalent to the following problem
    
    \begin{equation*} 
        \left\{
        \begin{aligned}
        -\left(\sum_{i,j=1}^dD_{\xi_{i}}D_{\xi_{j}}W(Du,x)D_{ij}u+\sum_{i=1}^{d}D_{\xi_{i}}D_{x_{i}}W(Du,x)+\sum_{i=1}^{d}\chi^{-1} D_{i}^{\mathrm{a}}\chi D_{\xi_{i}}W(Du,x)\right)&=\chi^{-1}f & & \text {in}\  \Omega, \\
        u &= 0 & & \text{on}\  \partial\Omega,
        \end{aligned}
        \right.    
    \end{equation*}
    
    (\romannumeral1). Let for all $\xi\in\sR^d$, $x\in\Omega$
    \begin{equation*} 
        b(\xi,x)= \sum_{i=1}^{d}D_{\xi_{i}}D_{x_{i}}W(\xi,x)+\sum_{i=1}^{d}\chi^{-1}D_{i}^{\mathrm{a}}\chi D_{\xi_{i}}W(\xi,x)+\frac{f}{\chi}.
    \end{equation*}
    By Assumption \ref{assum::CoerCvxBddQuasLin} (\romannumeral3), we have that
    \begin{equation*} 
        \Abs{\sum_{i=1}^{d}\chi^{-1}D_{i}^{\mathrm{a}}\chi D_{\xi_{i}}W(\xi,x)}
        \leq \frac{1}{2}\Abs{\chi^{-1}D^{\mathrm{a}}\chi}^2+\frac{1}{2}\Abs{D_{\xi}W(\xi,x)}^2\le C\left(\Abs{\chi^{-1}D^{\mathrm{a}}\chi}^2+\abs{\xi}^2+1\right),
    \end{equation*}
    where $C$ depends on $c_3$. By Assumption~\ref{assum::BddLipGrowQuasiLinNonDiv} (\romannumeral1) and the Cauchy--Schwartz inequality, we have
    \begin{equation*} 
        \Abs{\sum_{i=1}^{d}D_{\xi_{i}}D_{x_{i}}W(\xi,x)}
        \leq \frac{1}{2}\Abs{b_1(x)}+\frac{1}{2}(\Abs{\xi}+1)^2\leq \Abs{b_1(x)}^2+\Abs{\xi}^2+1.
    \end{equation*}
    Since $f\in L^p(\Omega)$, $b_1\in L^{2p}(\Omega)$, and $D^{\mathrm{a}}\chi\in L^\infty(\Omega)$, we have
    \begin{equation*} 
        \Abs{b(\xi,x)}\leq C(q(x)+\Abs{\xi}^2),
    \end{equation*}
    where $q(x)=b_1^2(x)+\Abs{\chi_{\min}^{-1}D^{\mathrm{a}}\chi}^2+\Abs{\chi_{\min}^{-1}f}+1\in L^p(\Omega)$ and $C$ depends on $c_3$. This verifies Assumption~\ref{assum::QuasLinNonDiv} (\romannumeral1).

    (\romannumeral2). By Assumption~\ref{assum::CoerCvxBddQuasLin} (\romannumeral2), (\romannumeral3) and Assumption~\ref{assum::BddLipGrowQuasiLinNonDiv} (\romannumeral1), namely, the conrtol on $D_{\xi_{i}}W(\xi,x)$ and $D_{\xi_{i}x_{i}}W(\xi,x)$ and uniform convexity, Assumption~\ref{assum::QuasLinNonDiv} (\romannumeral2) and (\romannumeral3) are satisfied.

    (\romannumeral3). 
    By Assumption~\ref{assum::BddLipGrowQuasiLinNonDiv} (\romannumeral3) and that $D_{\xi_{i}}D_{\xi_{j}}W(\xi,x)$ are independent of variable $z$, the first and second inequalities of Assumption~\ref{assum::QuasLinNonDiv} (\romannumeral4) holds naturally.
    
    Recall that $W\in  C^3(\sR^d\times \Omega)$. Thus for all $i,j,k\in \{1,\ldots,d\}$ 
    \begin{equation*} 
        D_{\xi_{k}}D_{\xi_{i}}D_{\xi_{j}}W-D_{\xi_{j}}D_{\xi_{i}}D_{\xi_{k}}W=0,
    \end{equation*}
    which implies the third inequality in Assumption~\ref{assum::QuasLinNonDiv} (\romannumeral4).

    Now consider the last inequality in Assumption~\ref{assum::QuasLinNonDiv} (\romannumeral4) and let
    \begin{equation*} 
        M=\Abs{\sum_{k=1}^{d}\left(D_{x_{k}} D_{\xi_{i}}D_{\xi_{j}}W(\xi, x) \xi_{k}-D_{x_{k}} D_{\xi_{k}}D_{\xi_{j}}W(\xi, x) \xi_{i}\right)}.
    \end{equation*}
    Thus we have the following estimate
    \begin{equation*} 
            M \leq2dc_5\abs{\xi}^2\leq 2dc_5\abs{\xi}(1+\abs{\xi}^2)^{1/2}(\abs{\xi}+b_3(x)),
    \end{equation*}
    where we let $b_3(x)=d\in L^{\infty}(\Omega)$.
    Thus Assumption~\ref{assum::QuasLinNonDiv} (\romannumeral4) is satisfied with constant function $\eta=2dc_5$.

    (\romannumeral4). Note that
    \begin{equation*}
        \int_{\Omega}\Abs{b(\xi,x)-b(\xi',x)}^p\diff{x}\leq I_1+I_2,
    \end{equation*}
    where 
    \begin{align*}
        I_1&=2^{p-1}\sum_{i=1}^{d}\int_{\Omega}\Abs{D_{\xi_{i}}D_{x_{i}}W(\xi,x)-D_{\xi_{i}}D_{x_{i}}W(\xi,x)}^p\diff{x},\\ 
        I_2&=2^{p-1}\sum_{i=1}^{d}\int_{\Omega}\Abs{\chi^{-1}D_{i}^{\mathrm{a}}\chi}^p\Abs{D_{\xi_{i}}W(\xi,x)-D_{\xi_{i}}W(\xi',x)}^p\diff{x}.
    \end{align*}
    By Assumption~\ref{assum::BddLipGrowQuasiLinNonDiv} (\romannumeral2), namely the Lipshcitz continuity of $D_{x_{i}}W$ and $D_{\xi_{i}}D_{x_{i}}W$, we have
    \begin{align*} 
        I_1 \leq 2^{p-1}d\Abs{\Omega}c_4^p\Abs{\xi-\xi'}^p
    \end{align*}
    and
    \begin{align*} 
        I_2
        &\leq 2^{p-1}\sum_{i=1}^{d}\left(\int_{\Omega}\Abs{\chi^{-1}D_{i}^{\mathrm{a}}\chi}^{2p}\diff{x}\right)^{1/2}\left(\int_{\Omega}\Abs{D_{\xi_{i}}W(\xi,x)-D_{\xi_{i}}W(\xi',x)}^{2p}\diff{x}\right)^{1/2} \\
        &\leq 2^{p-1}d\Abs{\Omega}c_4^pC\Abs{\xi-\xi'}^p,
    \end{align*}
    where $C$ depends on $\chi$.
    By setting $\Abs{\xi-\xi'}$ small enough, we obtain
    \begin{equation*} 
        \int_{\Omega}\Abs{b(\xi,x)-b(\xi',x)}^p\diff{x}\leq \eps,
    \end{equation*}
    and hence the Assumption~\ref{assum::QuasLinNonDiv} (\romannumeral5) is satisfied.
\end{proof}

\section*{Acknowledgement}
 This work is sponsored by the National Key R\&D Program of China  Grant No. 2022YFA1008200 (T. L.), the National Natural Science Foundation of China Grant No. 12101401 (T. L.), Shanghai Municipal Science and Technology Key Project No. 22JC1401500 (T. L.), Shanghai Municipal of Science and Technology Major Project No. 2021SHZDZX0102, and the HPC of School of Mathematical Sciences and the Student Innovation Center, and the Siyuan-1 cluster supported by the Center for High Performance Computing at Shanghai Jiao Tong University. The authors thank Yingzhou Li and Zhi-Qin John Xu for helpful discussions.

\bibliographystyle{plain}
\bibliography{DLRef}

\begin{thebibliography}{10}

\bibitem{ambrosio2018lectures}
Luigi Ambrosio, Alessandro Carlotto, and Annalisa Massaccesi.
\newblock {\em Lectures on elliptic partial differential equations}.
\newblock Edizioni della Normale Pisa, 2018.

\bibitem{ambrosio2001functions}
Luigi Ambrosio, Nicola Fusico, and Diego Pallara.
\newblock Functions of bounded variation and free discontinuity problems
  ({Oxford} {Mathematical} {Monographs}).
\newblock {\em Bulletin of the London Mathematical Society}, 33(4), 2001.

\bibitem{born1954dynamical}
Max Born and Kun Huang.
\newblock {\em Dynamical Theory of Crystal Lattices}.
\newblock Oxford University Press, 1954.

\bibitem{cai2022physics}
Shengze Cai, Zhiping Mao, Zhicheng Wang, Minglang Yin, and George Karniadakis.
\newblock Physics-informed neural networks ({PINNs}) for fluid mechanics: a
  review.
\newblock {\em Acta Mechanica Sinica}, 37, 01 2022.

\bibitem{cai2020deep}
Zhiqiang Cai, Jingshuang Chen, Min Liu, and Xinyu Liu.
\newblock Deep least-squares methods: An unsupervised learning-based numerical
  method for solving elliptic {PDE}s.
\newblock {\em Journal of Computational Physics}, 420:109707, 2020.

\bibitem{chen2021representation}
Ziang Chen, Jianfeng Lu, and Yulong Lu.
\newblock On the representation of solutions to elliptic {PDE}s in {Barron}
  spaces.
\newblock {\em Neural Information Processing Systems}, 2021.

\bibitem{cybenko1989approximation}
George Cybenko.
\newblock Approximation by superpositions of a sigmoidal function.
\newblock {\em Mathematics of control, signals and systems}, 2(4):303--314,
  1989.

\bibitem{weinan2021dawning}
Weinan E.
\newblock The dawning of a new era in applied mathematics.
\newblock {\em Notices of the American Mathematical Society}, 68(4):565--571,
  2021.

\bibitem{weinan2017deep}
Weinan E, Jiequn Han, and Arnulf Jentzen.
\newblock Deep learning-based numerical methods for high-dimensional parabolic
  partial differential equations and backward stochastic differential
  equations.
\newblock {\em Communications in Mathematics and Statistics}, 5(4):349--380,
  2017.

\bibitem{e2019apriori}
Weinan E, Chao Ma, and Lei Wu.
\newblock A priori estimates of the population risk for two-layer neural
  networks.
\newblock {\em Communications in Mathematical Sciences}, 17:1407--1425, 01
  2019.

\bibitem{ming2007cauchy}
Weinan E and Pingbing Ming.
\newblock Cauchy--{Born} rule and the stability of crystalline solids: static
  problems.
\newblock {\em Archive for Rational Mechanics and Analysis}, 183(2):241--297,
  2007.

\bibitem{weinan2018deep}
Weinan E and Bing Yu.
\newblock The deep {Ritz} method: A deep learning-based numerical algorithm for
  solving variational problems.
\newblock {\em Communications in Mathematics and Statistics}, 6(1):1--12, 2018.

\bibitem{evans2010partial}
Lawrance~C. Evans.
\newblock {\em Partial differential equations}.
\newblock American Mathematical Society, 2010.

\bibitem{gilbarg2001elliptic}
David Gilbarg and Neil~S. Trudinger.
\newblock {\em Elliptic partial differential equations of second order}.
\newblock Springer Berlin, Heidelberg, 2001.

\bibitem{goswami2020physics}
Somdatta Goswami, Aniruddha Bora, Yue Yu, and George~Em Karniadakis.
\newblock Physics-informed deep neural operator networks.
\newblock {\em arXiv preprint arXiv:2207.05748}, 2022.

\bibitem{han2018solving}
Jiequn Han, Arnulf Jentzen, and Weinan E.
\newblock Solving high-dimensional partial differential equations using deep
  learning.
\newblock {\em Proceedings of the National Academy of Sciences},
  115(34):8505--8510, 2018.

\bibitem{he2020relu}
Juncai He, Lin Li, Jinchao Xu, and Chunyue Zheng.
\newblock {ReLU} deep neural networks and linear finite elements.
\newblock {\em Journal of Computational Mathematics}, 38(3):502--527, 2020.

\bibitem{jiao2022arate}
Yuling Jiao, Yanming Lai, Dingwei Li, Xiliang Lu, Yang Wang, and Jerry~Zhijian
  Yang.
\newblock A rate of convergence of physics informed neural networks for the
  linear second order elliptic {PDE}s.
\newblock {\em Communications in Computational Physics}, 31(4):1272--1295,
  2022.

\bibitem{lawal2022physics}
Zaharaddeen~Karami Lawal, Hayati Yassin, Daphne Teck~Ching Lai, and Azam
  Che~Idris.
\newblock Physics-informed neural network ({PINN}) evolution and beyond: A
  systematic literature review and bibliometric analysis.
\newblock {\em Big Data and Cognitive Computing}, 6(4), 2022.

\bibitem{lee1990neural}
Hyuk Lee and In~Seok Kang.
\newblock Neural algorithm for solving differential equations.
\newblock {\em Journal of Computational Physics}, 91(1):110--131, 1990.

\bibitem{li2020amulti}
Xi-An Li, Zhi-Qin~John Xu, and Lei Zhang.
\newblock A multi-scale {DNN} algorithm for nonlinear elliptic equations with
  multiple scales.
\newblock {\em Communications in Computational Physics}, 28(5):1886--1906,
  2020.

\bibitem{li2006immersed}
Zhilin Li and Kazufumi Ito.
\newblock {\em The immersed interface method: numerical solutions of {PDEs}
  involving interfaces and irregular domains}.
\newblock Society for Industrial and Applied Mathematics, 2006.

\bibitem{li2020neural}
Zongyi Li, Nikola Kovachki, Kamyar Azizzadenesheli, Burigede Liu, Kaushik
  Bhattacharya, Andrew Stuart, and Anima Anandkumar.
\newblock Neural operator: Graph kernel network for partial differential
  equations.
\newblock {\em arXiv preprint arXiv:2003.03485}, 2020.

\bibitem{li2021fourier}
Zongyi Li, Nikola Kovachki, Kamyar Azizzadenesheli, Burigede Liu, Kaushik
  Bhattacharya, Andrew Stuart, and Anima Anandkumar.
\newblock Fourier neural operator for parametric partial differential
  equations.
\newblock {\em International Conference on Learning Representations}, 2021.

\bibitem{ming2021deep}
Yulei Liao and Pingbing Ming.
\newblock Deep {Nitsche} method: Deep {Ritz} method with essential boundary
  conditions.
\newblock {\em Communications in Computational Physics}, 29(5):1365--1384,
  2021.

\bibitem{lin2002geometric}
Fanghua Lin and Xiaoping Yang.
\newblock {\em Geometric measure theory: An introduction}.
\newblock International Press of Boston, Inc, 2002.

\bibitem{liu2020multi}
Ziqi Liu, Wei Cai, and Zhi-Qin~John Xu.
\newblock Multi-scale deep neural network ({MscaleDNN}) for solving
  poisson-boltzmann equation in complex domains.
\newblock {\em Communications in Computational Physics}, 28(5):1970--2001,
  2020.

\bibitem{lu2021deep}
Jianfeng Lu, Zuowei Shen, Haizhao Yang, and Shijun Zhang.
\newblock Deep network approximation for smooth functions.
\newblock {\em SIAM Journal on Mathematical Analysis}, 53(5):5465--5506, 2021.

\bibitem{lu2019deeponet}
Lu~Lu, Pengzhan Jin, and George~Em Karniadakis.
\newblock Deeponet: Learning nonlinear operators for identifying differential
  equations based on the universal approximation theorem of operators.
\newblock {\em arXiv preprint arXiv:1910.03193}, 2019.

\bibitem{lu2021deepxde}
Lu~Lu, Xuhui Meng, Zhiping Mao, and George~Em Karniadakis.
\newblock {DeepXDE}: A deep learning library for solving differential
  equations.
\newblock {\em SIAM Review}, 63(1):208--228, 2021.

\bibitem{lu2021priori}
Yulong Lu, Jianfeng Lu, and Min Wang.
\newblock A priori generalization analysis of the deep {Ritz} method for
  solving high dimensional elliptic partial differential equations.
\newblock {\em Conference on Learning Theory}, pages 3196--3241, 2021.

\bibitem{luo2019theory}
Tao Luo, Zheng Ma, Zhi-Qin~John Xu, and Yaoyu Zhang.
\newblock Theory of the frequency principle for general deep neural networks.
\newblock {\em arXiv preprint arXiv:1906.09235}, 2019.

\bibitem{mishra2020estimates}
Siddhartha Mishra and Roberto Molinaro.
\newblock {Estimates on the generalization error of physics-informed neural
  networks for approximating PDEs}.
\newblock {\em IMA Journal of Numerical Analysis}, 43(1):1--43, 01 2022.

\bibitem{osher2004level}
Stanley Osher, Ronald Fedkiw, and K~Piechor.
\newblock Level set methods and dynamic implicit surfaces.
\newblock {\em Applied Mechanics Reviews}, 57(3):B15--B15, 2004.

\bibitem{osher1988fronts}
Stanley Osher and James~A Sethian.
\newblock Fronts propagating with curvature-dependent speed: Algorithms based
  on hamilton-jacobi formulations.
\newblock {\em Journal of Computational Physics}, 79(1):12--49, 1988.

\bibitem{palagachev2006applications}
Dian Palagachev, Lutz Recke, and Lubomira Softova.
\newblock Applications of the differential calculus to nonlinear elliptic
  operators with discontinuous coefficients.
\newblock {\em Mathematische Annalen}, 336:617--637, 06 2006.

\bibitem{peskin1972flow}
Charles~S Peskin.
\newblock Flow patterns around heart valves: a numerical method.
\newblock {\em Journal of Computational Physics}, 10(2):252--271, 1972.

\bibitem{peskin2002immersed}
Charles~S Peskin.
\newblock The immersed boundary method.
\newblock {\em Acta Numerica}, 11:479--517, 2002.

\bibitem{raissi2019physics}
Maziar Raissi, Paris Perdikaris, and George~E Karniadakis.
\newblock Physics-informed neural networks: A deep learning framework for
  solving forward and inverse problems involving nonlinear partial differential
  equations.
\newblock {\em Journal of Computational Physics}, 378:686--707, 2019.

\bibitem{sethian1999level}
James~Albert Sethian.
\newblock {\em Level set methods and fast marching methods: evolving interfaces
  in computational geometry, fluid mechanics, computer vision, and materials
  science}, volume~3.
\newblock Cambridge university press, 1999.

\bibitem{shin2020convergence}
Yeonjong Shin, Jerome Darbon, and George~Em Karniadakis.
\newblock On the convergence of physics informed neural networks for linear
  second-order elliptic and parabolic type {PDE}s.
\newblock {\em Communications in Computational Physics}, 28(5):2042--2074,
  2020.

\bibitem{shin2020error}
Yeonjong Shin, Zhongqiang Zhang, and George~Em Karniadakis.
\newblock Error estimates of residual minimization using neural networks for
  linear {PDE}s, 2020.

\bibitem{sirignano2018DGM}
Justin Sirignano and Konstantinos Spiliopoulos.
\newblock {DGM}: A deep learning algorithm for solving partial differential
  equations.
\newblock {\em Journal of Computational Physics}, 375:1339--1364, 2018.

\bibitem{xu2019frequency}
Zhi-Qin~John Xu, Yaoyu Zhang, Tao Luo, Yanyang Xiao, and Zheng Ma.
\newblock Frequency principle: Fourier analysis sheds light on deep neural
  networks.
\newblock {\em Communications in Computational Physics}, 28(5):1746--1767,
  2020.

\bibitem{xu2019training}
Zhi-Qin~John Xu, Yaoyu Zhang, and Yanyang Xiao.
\newblock Training behavior of deep neural network in frequency domain.
\newblock In {\em Neural {Information} {Processing}}, Lecture {Notes} in
  {Computer} {Science}, pages 264--274, 2019.

\bibitem{bao2021weak}
Yaohua Zang, Gang Bao, Xiaojing Ye, and Haomin Zhou.
\newblock Weak adversarial networks for high-dimensional partial differential
  equations.
\newblock {\em Journal of Computational Physics}, 411:109409, 03 2020.

\bibitem{zhang2022mod}
Lulu Zhang, Tao Luo, Yaoyu Zhang, Zhi-Qin~John Xu, and Zheng Ma.
\newblock {MOD-Net}: A machine learning approach via model-operator-data
  network for solving {PDE}s.
\newblock {\em Communications in Computational Physics}, 32(2):299--335, 2022.

\bibitem{zhang2021linear}
Yaoyu Zhang, Tao Luo, Zheng Ma, and Zhi-Qin~John Xu.
\newblock A linear frequency principle model to understand the absence of
  overfitting in neural networks.
\newblock {\em Chinese Physics Letters}, 38(3):038701, 2021.

\end{thebibliography}
\end{document}